\documentclass[11pt,a4paper]{amsart}
\usepackage[left=2.5cm,right=2.5cm,top=3cm,bottom=3cm]{geometry}

\usepackage{amssymb, amsthm}

\usepackage{hyperref}
\usepackage{tikz,tikz-cd}

\usepackage[cal=euler, scr=boondox]{mathalfa}

\usepackage{bm} 
\usepackage{mathdots} 

\usepackage{booktabs}

\usepackage{dynkin-diagrams}

\usepackage{adjustbox}

\usepackage[backend=biber]{biblatex}
\newtheorem{thm}{Theorem}[section]
\newtheorem{prop}[thm]{Proposition}
\newtheorem{conj}[thm]{Conjecture}
\newtheorem{lemma}[thm]{Lemma}
\newtheorem{corol}[thm]{Corollary}
\theoremstyle{remark} \newtheorem{rmk}{Remark}
\theoremstyle{definition} \newtheorem{defn}[thm]{Definition}
\theoremstyle{definition} 
\theoremstyle{definition} \newtheorem{ej}[thm]{Example}

\newcommand{\git}{\mathbin{/\mkern-6mu/}}
\newcommand{\Sc}{\mathrm{sc}}

\newcommand{\C}{\mathbb C}
\newcommand{\CP}{\mathbb{CP}}

\newcommand{\Q}{\mathbb Q}
\newcommand{\RR}{\mathbb R}
\newcommand{\Z}{\mathbb Z}

\newcommand{\Gm}{\mathbb G_{\mathrm{m}}}
\newcommand{\Ga}{\mathbb G_{\mathrm{a}}}
\newcommand{\bbA}{\mathbb A}

\newcommand{\bbR}{\mathbb R}

\newcommand{\bbE}{\mathbb E}
\newcommand{\bbP}{\mathbb P}

\newcommand{\cA}{\mathcal A}

\newcommand{\cH}{\mathcal H}

\newcommand{\cM}{\mathcal M}

\newcommand{\cP}{\mathcal P}

\newcommand{\fB}{\mathfrak B}

\newcommand{\fD}{\mathfrak D}

\newcommand{\fL}{\mathfrak L}

\newcommand{\ft}{\mathfrak t}
\newcommand{\fg}{\mathfrak{g}}

\newcommand{\fc}{\mathfrak c}

\newcommand{\sO}{\mathscr O}

\newcommand{\sF}{\mathscr F}
\newcommand{\sL}{\mathscr L}

\newcommand{\sJ}{\mathscr J}
\newcommand{\sS}{\mathscr S}

\newcommand{\sfA}{\mathsf A}
\newcommand{\sfB}{\mathsf B}
\newcommand{\sfC}{\mathsf C}
\newcommand{\sfD}{\mathsf D}
\newcommand{\sfE}{\mathsf E}
\newcommand{\sfF}{\mathsf F}
\newcommand{\sfG}{\mathsf G}

\newcommand{\bD}{\mathbf{D}}
\newcommand{\bH}{\mathbf{H}}
\newcommand{\bW}{\mathbf{W}}
\newcommand{\bS}{\mathbf{S}}
\newcommand{\bM}{\mathbf{M}}
\newcommand{\R}{\mathrm{R}}

\newcommand{\Ad}{\mathrm{Ad}}
\newcommand{\ad}{\mathrm{ad}}
\newcommand{\der}{\mathrm{der}}
\newcommand{\Hit}{\mathrm{Hit}}
\newcommand{\Mon}{\mathrm{Mon}}
\DeclareMathOperator{\coker}{coker}
\DeclareMathOperator{\Aut}{Aut}
\DeclareMathOperator{\Out}{Out}
\DeclareMathOperator{\Hom}{Hom}
\DeclareMathOperator{\Map}{Map}

\DeclareMathOperator{\Pic}{Pic}
\DeclareMathOperator{\Jac}{Jac}

\DeclareMathOperator{\Env}{Env}
\DeclareMathOperator{\Spec}{Spec}
\DeclareMathOperator{\tr}{tr}

\DeclareMathOperator{\Bun}{Bun}
\DeclareMathOperator{\Nm}{Nm}

\DeclareMathOperator{\Mat}{Mat}
\DeclareMathOperator{\GL}{GL}
\DeclareMathOperator{\SL}{SL}
\DeclareMathOperator{\PGL}{PGL}

\DeclareMathOperator{\id}{id}
\DeclareMathOperator{\reg}{reg}
\DeclareMathOperator{\Ss}{ss}
\DeclareMathOperator{\rs}{rs}
\DeclareMathOperator{\gl}{gl}
\DeclareMathOperator{\ab}{ab}
\DeclareMathOperator{\QCoh}{QCoh}
\DeclareMathOperator{\treg}{\theta \text- reg}
\DeclareMathOperator{\tSs}{\theta \text- ss}
\DeclareMathOperator{\trs}{\theta \text- rs}
\DeclareMathOperator{\tgl}{\theta \text- gl}
\DeclareMathOperator{\Bil}{Bil}
\DeclareMathOperator{\Diff}{Diff}
\DeclareMathOperator{\mHiggs}{mHiggs}
\DeclareMathOperator{\bmHiggs}{\mathbf{mHiggs}}
\DeclareMathOperator{\Conn}{\mathbf{Conn}}
\DeclareMathOperator{\QC}{\mathbf{QC}}
\DeclareMathOperator{\dDiff}{\mathbf{Diff}}


\author{Guillermo Gallego}
\title{Multiplicative Hitchin fibrations and Langlands duality} 

\address{Guillermo Gallego Sánchez \newline
\indent Freie Universität Berlin \newline
\indent Institut für Mathematik \newline
\indent Arnimallee 3 \newline
\indent 14195 Berlin, Germany}
\email{\normalfont\href{mailto:guillermo.gallego.sanchez@fu-berlin.de}{guillermo.gallego.sanchez@fu-berlin.de}}

\urladdr{\normalfont\href{https://guillegallego.xyz}{https://guillegallego.xyz}}

\thanks{The author's research is supported by the FU Berlin under a postdoctoral contract (DFG-CoNaRe-Projekt).}

\makeatletter
\hypersetup{
	pdfauthor={\authors},
	pdftitle={\@title},
	colorlinks,
	linkcolor=[rgb]{0.2,0.2,0.6},
	citecolor=[rgb]{0.2,0.6,0.2},
	urlcolor=[rgb]{0.6,0.2,0.2}}
\makeatother

\begin{document}
\maketitle

\begin{abstract}
We identify pairs of (twisted) multiplicative Hitchin fibrations which are ``dual" in the sense that their bases are identified and their generic fibres are dual Beilinson $1$-motives. More precisely, we match the following: (1) an untwisted multiplicative Hitchin fibration associated with a simply-laced semisimple group $G$ with an untwisted multiplicative Hitchin fibration associated with the Langlands dual group $G^\vee$;  (2) a twisted multiplicative Hitchin fibration associated with a simply-laced and simply-connected semisimple group $G$, without factors of type $\sfA_{2\ell}$, and a diagram automorphism $\theta \in \Aut(G)$ with an untwisted multiplicative Hitchin fibration associated with the Langlands dual group $H^\vee$ of the invariant group  $H=G^\theta$;  (3) two twisted multiplicative Hitchin fibrations associated with $G=\SL_{2\ell +1}$ and two special automophisms of order $2$ and $4$, respectively. These results are consistent with a conjecture of Elliott and Pestun \cite{Elliott-Pestun}.
\end{abstract}

\section{Introduction}
\subsection{Hitchin fibrations and multiplicative Hitchin fibrations} 
 Let $C$ be a smooth complex projective curve and let $G$ be a reductive algebraic group over $\C$.
The \emph{Hitchin fibration} is an algebraically completely integrable system introduced by Hitchin in his seminal paper \cite{Hitchin_Systems}. The total space of this fibration is the moduli space of \emph{Higgs bundles}, pairs $(E,\varphi)$ formed by a principal $G$-bundle $E\rightarrow C$ and a section  $\varphi \in H^0(C,\ad(E)\otimes K_C)$, where $K_C$ is the canonical line bundle of $C$. If  $K_C$ is replaced by a different line bundle $L$, one can  analogously define ``$L$-twisted" Higgs bundles and formulate a generalization of the Hitchin fibration. When $L\neq K_C$, the complex symplectic structure is lost, but some of the symmetry is retained. In particular, if the degree of  $L$ is large enough, the generic fibres are (discrete unions of) abelian varieties. The Hitchin fibration can be extended to the whole moduli stack of ($L$-twisted) Higgs bundles which, as explained by Ngô \cite{Ngo_Lemme}, can be identified as the stack of maps from $C$ to the quotient stack $[\fg/G\times \Gm]$ lying over the natural map $C\rightarrow B\Gm$ induced by  $L$. Here, $\fg$ denotes the Lie algebra of $G$, which is acted on by $G$ through the adjoint action, and by $\Gm$ by homothecy. The Hitchin fibration can in turn be understood as a ``global analogue"  of the Chevalley restriction map
\begin{equation*}
\chi: \fg \longrightarrow \fg \git G \cong \ft/W \cong \bbA^r.
\end{equation*} 
Here, $r$ is the rank of  $G$, $\ft$ denotes the Lie algebra of a maximal torus  $T\subset G$ and $W$ is the Weyl group. At the stacky level, generic fibres are commutative group stacks whose inertia groups are groups of multiplicative type and whose coarse moduli spaces are discrete unions of abelian varieties. Stacks of this form are known as \emph{Beilinson $1$-motives} \cite[Section 5.2]{Donagi-Pantev}.

The reductive group $G$ acts on itself by conjugation. More generally, it acts by the adjoint action on any other reductive group $G'$ with $\Ad(G)=\Ad(G')$. In particular, if $G$ is semisimple, it acts on its simply-connected cover $G^{\Sc}$.
\emph{Multiplicative Hitchin fibrations} arise as a way to formulate a global version of the restriction map
\begin{equation*}
\chi: G' \longrightarrow G' \git G \cong T'/W.
\end{equation*} 
Here, $T'\subset G'$ is a maximal torus. Roughly, one wants to consider the mapping stack from $C$ to the quotient stack  $[G'/G]$. Thus, one is led to study \emph{multiplicative Higgs bundles}: pairs $(E,\varphi)$ formed by a principal $G$-bundle $E\rightarrow C$ and a section of the associated group bundle $\varphi \in H^0(C,E\times^G G')$. However, note that, in constrast with the Lie algebra, if $G'$ is also semisimple, then it is not naturally endowed with a $\Gm$-action. Therefore, there is no natural way a priori to introduce twisting or singularities on the ``multiplicative Higgs fields". There are two natural ways to remedy this:
\begin{itemize}
	\item The ``meromorphic approach". This was the original approach proposed by Hurtubise and Markman \cite{Hurtubise-Markman} and also used, for example, in the works of Charbonneau--Hurtubise \cite{Charbonneau-Hurtubise} and Elliott--Pestun \cite{Elliott-Pestun}. With this approach, one considers pairs $(E,\varphi)$ as above, but such that $\varphi$ is defined only away from a finite subset $\left\{p_1,\dots,p_n\right\}\subset C$, and has a rational singularity at each point $p_i$. The ``pole order" of the rational singularity at a point $p_i$ is determined by a dominant cocharacter $\lambda_i$ of $T'$. 
	\item The ``monoid approach". Another way of introducing singularities is given by first extending the semisimple group $G'$ to a larger reductive group  $G'_+$ with a central torus $Z$, and then partially compactifying the group $G'_+$ over the curve $C$. We obtain a \emph{reductive monoid} $M$ with unit group $G'_+$. The adjoint action of $G$ on $G'$ extends to an action on $M$. One is then led to consider the stack of maps from  $C$ to the quotient stack $[M/G\times Z]$ lying over a map $C\rightarrow BZ$ determined by a prescribed $Z$-torsor $L\rightarrow C$. Prescribing the  $Z$-torsor $L$ and a section of the associated bundle $L\times^Z \bbA_M$ where  $\bbA_M=M\git(G' \times G')$ is the \emph{abelianization} of $M$, amounts to prescribing the points $p_i$ and ``pole orders" on the ``meromorphic approach". This ``monoid approach" was originally proposed in the work of Frenkel and Ngô \cite{Frenkel-Ngo} and developed in the subsequent works of Bouthier, J. Chi and G. Wang \cite{Bouthier_Dimension,Bouthier_Hitchin,JChi,Griffin_Lemma}. 
\end{itemize}

Multiplicative Hitchin fibrations were originally introduced in the mathematics literature by Hurtubise and Markman \cite{Hurtubise-Markman}. In that paper, they show that, when $C$ is assumed to be an elliptic curve, then multiplicative Hitchin fibrations are in fact algebraically completely integrable systems. The same arguments show that, for a general choice of $C$ and with a good ``ampleness" assumption on the ``pole data", generic fibres of a multiplicative Hitchin fibration are Beilinson $1$-motives. This fact also follows, in the generality we use it here, from the more recent study of the symmetries of a multiplicative Hitchin fibration developed by G. Wang in \cite{Griffin_Lemma}.  

\subsection{Langlands duality} Let us assume now that $G$ is semisimple and fix the maximal torus $T\subset G$. The real vector space $\mathbb{E}=X^*(T) \otimes_{\Z} \bbR$ is endowed with the Killing form $(-,-)$, and the pair $(G,T)$ determines a root system $\Phi=\Phi(G,T)$ on $(\mathbb{E},(-,-))$. The dual vector space of the complexification of $\mathbb{E}$ is naturally identified with  the Lie algebra $\ft$ of $T$. The set of closed points of $T$ is the complex torus $T(\C)=\ft/X_*(T)$.
The dual algebraic torus $T^\vee$ of $T$ is the Cartier dual of the cocharacter lattice $X_*(T)$; its set of closed points  is the complex torus $T^\vee(\C)=\ft^*/X^*(T)$.

Let $G^\vee$ denote the Langlands dual group of the reductive group $G$. By definition, the dual torus $T^\vee$ is a maximal torus of $G^\vee$, and the root system $\Phi(G^\vee,T^\vee)\subset \mathbb{E}^*$ is the dual root system 
\begin{equation*}
\Phi^\vee = \left\{\alpha^\vee=2(\alpha,-)/(\alpha,\alpha): \alpha \in \Phi\right\}
\end{equation*} 
of  $\Phi=\Phi(G,T)$. The Killing form is $W$-equivariant, so it induces an isomorphism $$\ft/W \longrightarrow \ft^*/W.$$ In turn, this allows us to identify the base of the Hitchin fibration associated with  $G$ with the base of the Hitchin fibration associated with the Langlands dual group  $G^\vee$. Given any point of the base of the Hitchin fibration associated with $G$ such that its fibre is a Beilinson $1$-motive $\cP$, the corresponding fibre of the Hitchin fibration associated with $G^\vee$ is the dual Beilinson $1$-motive $\bD(\cP)$. More precisely, the inertia group of $\cP$ is Cartier dual to the component group of  $\bD(\cP)$ and vice-versa, and the underlying abelian varieties  $P$ and  $\hat{P}$ are dual abelian varieties.   This phenomenon was first detected by Hausel and Thaddeus \cite{Hausel-Thaddeus} in the case of type $\sfA_n$ and proven in general by Donagi and Pantev \cite{Donagi-Pantev}. It is also worth mentioning that there is an analogous result in positive characteristic proven by T.H. Chen and X. Zhu \cite{Chen-Zhu}.

Multiplicative Hitchin fibrations are different. Indeed, in general one cannot match the maximal tori of $G$ and $G^\vee$ in a $W$-equivariant way. However, if $G$ is simply-laced (that is, if all the roots in $\Phi$ have the same length), then we can normalize the Killing form in such a way that $(\alpha,\alpha)=2$ for every root, and then the isomorphism $\ft^* \rightarrow \ft, x \mapsto (x,-)$ induced by the Killing form exchanges roots with coroots and weights with coweights. That is, when $G$ is simply-laced, we have $\alpha^\vee=(\alpha,-)$. In particular, this gives an identification of the simply-connected covers and the adjoint groups of $G$ and of $G^\vee$; that is $G^{\Sc}=(G^\vee)^{\Sc}$ and $\Ad(G)=\Ad(G^\vee)$. Therefore, if $G$ is simply-laced, and we prescribe some ``pole data" (either through the ``meromorphic approach" or by fixing a monoid $M$ whose unit group is a reductive group with semisimple part $G^{\Sc}$ and with center $Z$, and a $Z$-torsor), then we can consider multiplicative Hitchin fibrations $h_{G}:\cM_{G,G^{\Sc}}\rightarrow \cA_{G^{\Sc}}$ and $h_{G^\vee}:\cM_{G^\vee,G^{\Sc}}\rightarrow \cA_{G^{\Sc}}$ associated with the adjoint action on $G^{\Sc}$ of $G$ and $G^\vee$, respectively. Our first result is the following (this is Theorem \ref{thm:duality_simplylaced_text}).

\begin{thm} \label{thm:duality_simplylaced}
If $G$ is simply-laced, then, assuming some ``ampleness" conditions on the pole data, for a generic $a\in \cA_{G^{\Sc}}(\C)$, the fibres $h_G^{-1}(a)$ and $h_{G^\vee}^{-1}(a)$ are dual Beilinson $1$-motives.
\end{thm}

\subsection{Twisted multiplicative Hitchin fibrations} When the structure group is not simply-laced, we detect a new kind of duality, which is not observed at the Lie algebra level. Multiplicative Hitchin fibrations associated with non simply-laced groups are dual to \emph{twisted multiplicative Hitchin fibrations} associated with simply-laced groups.

Let $G$ be a semisimple simply-laced and simply-connected group, and let $\theta \in \Aut(G)$ be a \emph{diagram automorphism} of $G$. By definition, $\theta$ is the representative of an outer class $[\theta]\in \Out(G)=\Aut(G)/\Ad(G)$ preserving a prescribed pinning of $G$. We can now consider the $\theta$-twisted conjugation action of $G$ on itself,
\begin{equation*}
g * h = g h \theta(g)^{-1}.
\end{equation*} 
Equivalently, we can consider the exterior component $G\theta$ of the semidirect product $G \rtimes \langle \theta \rangle$ and consider the conjugation action of  $G$ on $G\theta$.
Twisted multiplicative Hitchin fibrations arise as a global analogue of the restriction map
 \begin{equation*}
\chi_\theta: G \longrightarrow G\theta \git G.
\end{equation*} 
That is, we consider maps $C\rightarrow [G\theta/G]$ (``singular", in some sense, and with an appropriate notion of ``pole data"). Thus we are led to consider pairs $(E,\varphi)$ formed by a principal $G$-bundle $E\rightarrow C$ and a section of the associated group bundle  $\varphi \in H^0(C,E\times^{G,\theta} G)$.

The data $(G,\theta)$ determines two different reductive groups: the \emph{invariant group} $G^\theta$ and the \emph{coinvariant group} $G_\theta$. The invariant group is just the subgroup of fixed points of $G$ under the automorphism $\theta$. Since  $G$ is simply-connected, $G^\theta$ is connected. If $T$ is a $\theta$-stable maximal torus of $G$, then $T^\theta$ is a maximal torus of $G^\theta$. The root system  $\Phi(G^\theta,T^\theta)$ is the \emph{folded root system} associated with the diagram automorphism $\theta$. The coinvariant group is more conveniently defined in terms of root data. The coinvariant torus  $T_\theta=T/(1-\theta)(T)$ is a maximal torus of $G_\theta$, while its root system is the dual root system of  the folded root system. The character lattice of $T_\theta$ is the fixed point subgroup of the character lattice of $T$, and thus it is the weight lattice of $G_\theta$. This implies that  $G_\theta$ is also simply-connected. The invariant group $G^\theta$ is simply-connected except if $G$ has a component isomorphic to $\SL_{2\ell+1}$ for some $\ell$.

The possible pairs $(G,\theta)$, with $G$ simple and $\theta \neq \id_G$, are classified by the \emph{twisted affine Dynkin diagrams}, see Table \ref{tab:DiagramAutomorphisms}. The \emph{folded Dynkin diagrams} associated with the folded root systems can be read from Table \ref{tab:DiagramAutomorphisms} by removing the node marked with a red cross {\color{red} $\times$} in the the corresponding twisted affine Dynkin diagram.

\begin{table}[ht]
\caption{Diagram automorphisms of the simple simply-connected groups, with their corresponding twisted affine Dynkin diagrams, and their folded and finite parts. Notations: $N_{r}$ is the antidiagonal $r\times r$ matrix $\mathrm{antidiag}(1,-1,1,\dots)$, $P_\ell$ is the  diagonal $(2\ell+2) \times (2\ell +2)$ matrix $\mathrm{diag}(1,\dots,1,-1)$. The red cross {\color{red}$\times$} marks the node of the affine Dynkin diagram one has to remove to obtain the folded diagram. The finite part is given by the black nodes in the affine Dynkin diagram.}
\begin{adjustbox}{max width=\textwidth}
\begin{tabular}{@{}cccccccc@{}}
\toprule
$G$ &
  $\theta$ &
  $G^\theta$ &
  $G_\theta$ &
  Folding pattern &
  Affine Dynkin diagram &
  Folded diagram &
  Finite part \\ \midrule
$\SL_{2\ell}$ &
  $A \mapsto N_{2\ell} (A^T)^{-1} N_{2\ell}^{-1}$ &
  $\mathrm{Sp}_{2\ell}$ &
  $\mathrm{Spin}_{2\ell+1}$ &
    \dynkin[involutions={19;28;37;46}]{A}{***.***.***}&
  \dynkin[labels={{\color{red} \times},}] A[2]{odd} $\sfA_{2\ell-1}^{(2)}$ &   
 \dynkin C{} $\sfC_{\ell}$  &
  \dynkin C{} $\sfC_{\ell}$  \\
$\SL_{3}$ &
  $A \mapsto N_{3} (A^T)^{-1} N_{3}^{-1}$ &
  $\mathrm{SO}_{3}\cong \PGL_2$ &
  $\mathrm{Sp}_{2} \cong \SL_2$ &
    \dynkin[involutions={12}]{A}{**}&
  \dynkin[labels={,{\color{red} \times}}] A[2]{*} $\sfA_{2}^{(2)}$ &   
 \dynkin A{o} $\sfA_{1}$  &
  \dynkin A{1} $\sfA_{1}$  \\
$\SL_{2\ell+1}$ &
  $A \mapsto N_{2\ell+1} (A^T)^{-1} N_{2\ell+1}^{-1}$ &
  $\mathrm{SO}_{2\ell+1}$ &
  $\mathrm{Sp}_{2\ell}$ &
    \dynkin[involutions={18;27;36;45}]{A}{**.****.**}&
  \dynkin[labels={,,,,,,{\color{red} \times}}] A[2]{even} $\sfA_{2\ell}^{(2)}$ &   
 \dynkin B{**.**o} $\sfB_{\ell}$  &
  \dynkin C{} $\sfC_{\ell}$  \\
$\mathrm{Spin}_{2\ell+2}$ &
  $A \mapsto P_\ell A P_\ell^{-1}$ &
  $\mathrm{Spin}_{2\ell+1}$ &
  $\mathrm{Sp}_{2\ell}$ &
    \dynkin[involutions={65}]{D}{**.****}&
  \dynkin[labels={{\color{red} \times},}] D[2]{} $\sfD_{\ell+1}^{(2)}$ &   
 \dynkin B{} $\sfB_{\ell}$  &
  \dynkin B{} $\sfB_{\ell}$  \\
$\sfE_{6}$ &
   - &
  $\sfF_4$ &
  $\sfF_4$ &
    \dynkin[involutions={16;35}]{E}{6}&
  \dynkin[labels={{\color{red} \times},}] E[2]{6} $\sfE_{6}^{(2)}$ &   
 \dynkin F{4} $\sfF_{4}$  &
  \dynkin F{4} $\sfF_{4}$  \\
$\sfD_{4}$ &
   - &
  $\sfG_2$ &
  $\sfG_2$ &
    \dynkin[involutions={31;43;14}]{D}{4}&
  \dynkin[labels={{\color{red} \times},}] D[3]{4} $\sfD_{4}^{(3)}$ &   
 \dynkin G{2} $\sfG_{2}$  &
  \dynkin G{2} $\sfG_{2}$  \\
 \bottomrule
\end{tabular}
\end{adjustbox}
	\label{tab:DiagramAutomorphisms}
\end{table}

The adjoint action of the coinvariant group on itself is related to the twisted adjoint action of the larger group on itself. Indeed, Mohrdieck \cite{Mohrdieck} showed that there is a natural isomorphism of the GIT quotients
\begin{equation*}
G\theta \git G \cong G_\theta \git G_\theta.
\end{equation*} 

Let $(G,\theta)$ be as above, and let us assume that $G$ does not have a component isomorphic to $\SL_{2\ell+1}$, for some $\ell$. We denote by $H=G^{\theta}$ the corresponding invariant group, and let $H^\vee$ denote its Langlands dual group. Note that, under our assumptions, $(H^\vee)^{\Sc}=G_\theta$ is the coinvariant group, so $H^\vee$ acts on $G_\theta$ through the adjoint action. We denote by $h_{G,\theta}:\cM_{G,\theta}\rightarrow \cA_{G_\theta}$ the twisted multiplicative Hitchin fibration associated with the pair $(G,\theta)$, and by $h_{H^\vee}:\cM_{H^\vee,G_\theta}\rightarrow \cA_{G_\theta}$ the (untwisted) multiplicative Hitchin fibration associated with the adjoint action of $H^\vee$ on  $G_\theta$. Note that, by Mohrdieck's isomorphism $G\theta \git G\cong G_\theta \git G_\theta$, the bases of these two fibrations coincide. Duality for non simply-laced groups is summarized in the following result (this is Theorem \ref{thm:duality_twisted_text}).

\begin{thm}\label{thm:duality_twisted}
Asumming ``ampleness" conditions on the pole data, for a generic $a\in \cA_{G_\theta}(\C)$, the fibres $h_{G,\theta}^{-1}(a)$ and $h_{H^\vee}^{-1}(a)$ are dual Beilinson $1$-motives. 	
\end{thm}

The case in which $G$ has a component isomorphic to $\SL_{2\ell +1}$ requires an special treatment. Indeed, in this case, apart from the diagram automorphism $\theta$ of $\SL_{2\ell+1}$, we also have to consider the ``standard automorphism" of order $4$ introduced by Kac in his book \cite{Kac}. Here, we are using the terminology of \cite{Besson-Hong}, where they use the name ``standard automorphism" to refer to the automorphism $\vartheta$ of a Lie algebra $\fg$ such that the fixed point subalgebre $\fg^{\vartheta}$ coincides with the standard finite part of the associated twisted affine Lie algebra determined by the corresponding twisted affine Dynkin diagram. In Table \ref{tab:DiagramAutomorphisms}, this finite part is the Lie algebra determined by the black nodes in the affine Dynkin diagram. We have $\vartheta=\theta$ in all simple cases except in the case $G=\SL_{2\ell+1}$. In this case, we have $G^\vartheta=\mathrm{Sp}_{2\ell}$. We can also consider the twisted multiplicative Hitchin fibration $h_{G,\vartheta}:\cM_{G,\vartheta}\rightarrow \cA_{G_\theta}$ associated with the pair $(G,\vartheta)$. Duality in the case $\sfA_{2\ell}^{(2)}$ is summarized in the following result (this is Theorem \ref{thm:duality_Aeven2_text}).

\begin{thm}\label{thm:duality_Aeven2}
	Let $G=\SL_{2\ell+1}$, $\theta\in \Aut_2(G)$ the diagram automorphism and $\vartheta \in \Aut_4(G)$ the standard automorphism defined as above.
Asumming ``ampleness" conditions on the pole data, for a generic $a\in \cA_{G_\theta}(\C)$, the fibres $h_{G,\theta}^{-1}(a)$ and $h_{G,\vartheta}^{-1}(a)$ are dual Beilinson $1$-motives. 	
\end{thm}

Note that in this case of $G=\SL_{2\ell+1}$, the moduli stacks $\cM_{G,\theta}$ and $\cM_{G,\vartheta}$ are indeed isomorphic, so the above result can be reinterpreted as the existence of an involution on the moduli space relating dual fibres. In section \ref{sec:A22} we provide an explicit description of this involution in the case of $\SL_3$.

\subsection{Duality of affine Dynkin diagrams}
Our duality results can be succintly expressed in terms of the duality of affine Dynkin diagrams. Indeed, we can understand untwisted multiplicative Hitchin fibrations as twisted multiplicative Hitchin fibrations associated with pairs of the form $(G,\id)$ and, with each pair $(G,\theta)$, we can associate an affine Dynkin diagram. If we do this, then the multiplicative Hitchin fibrations which are matched under duality are precisely those with dual affine Dynkin diagrams associated with them. 
When matching two untwisted multiplicative Hitchin fibrations, we act on both sides on the simply-connected group, but interchange the acting group by its Langlands dual (that is, we swap the centre and the fundamental group). When matching a twisted multiplicative Hitchin fibration with an untwisted one, the acting group on the twisted side is simply-connected, and on the untwisted side it is of adjoint type.
We sumarize the situation for simple groups in Table \ref{tab:PairsWithDiagram}.

\begin{table}[ht]
\caption{Pairs of simple groups matched under duality. The notation $(G_1,G_2\theta)$ where $G_1$ and $G_2$ are groups and $\theta$ is an automorphism indicates that the group $G_1$ acts on $G_2$ by  $\theta$-twisted conjugation. We omit $\theta$ when it is the identity. We denote by  $\theta_3$ the order $3$ diagram automorphism of  $\mathrm{Spin}_8$. The horizontal lines separate the three situations: untwisted simply-laced vs untwisted simply-laced, twisted simply-laced vs untwisted non simply-laced and twisted vs twisted.}
\begin{adjustbox}{max width=\textwidth}
\begin{tabular}{@{}cccc@{}}
\toprule
Pair              & Affine Dynkin diagram & Dual affine Dynkin diagram & Dual pair \\ \midrule
$(\SL_r,\SL_r)$               & \dynkin A[1]{} $\sfA_r^{(1)}$                 & \dynkin A[1]{} $\sfA_r^{(1)}$                       & $(\PGL_r,\SL_r)$                  \\
$(\mathrm{SO}_{2r},\mathrm{Spin}_{2r})$             & \dynkin D[1]{} $\sfD_r^{(1)}$                 & \dynkin D[1]{} $\sfD_r^{(1)}$                      & $(\mathrm{SO}_{2r},\mathrm{Spin}_{2r})$                \\
$(\mathrm{Spin}_{2r},\mathrm{Spin}_{2r})$           & \dynkin D[1]{} $\sfD_r^{(1)}$                  & \dynkin D[1]{} $\sfD_r^{(1)}$                       & $(\mathrm{PSO}_{2r},\mathrm{Spin}_{2r})$               \\
$(\sfE_6,\sfE_6)$                    & \dynkin E[1]{6} $\sfE_6^{(1)}$                  & \dynkin E[1]{6} $\sfE_6^{(1)}$                       & $(\sfE_6,\sfE_6)$                    \\
$(\sfE_7,\sfE_7)$                    & \dynkin E[1]{7} $\sfE_7^{(1)}$                  & \dynkin E[1]{7} $\sfE_7^{(1)}$                       & $(\sfE_7,\sfE_7)$                    \\
$(\sfE_8,\sfE_8)$                    & \dynkin E[1]{8} $\sfE_8^{(1)}$                  & \dynkin E[1]{8} $\sfE_8^{(1)}$                       & $(\sfE_8,\sfE_8)$                    \\
\midrule
$(\SL_{2\ell},\SL_{2\ell}\theta)$ & \dynkin A[2]{odd} $\sfA_{2\ell-1}^{(2)}$              & \dynkin B[1]{****.***}  $\sfB_{\ell}^{(1)}$                     & $(\mathrm{SO}_{2\ell+1},\mathrm{Spin}_{2\ell+1})$                \\
$(\mathrm{Spin}_{2\ell+2},\mathrm{Spin}_{2\ell+2}\theta)$     & \dynkin D[2]{}  $\sfD_{\ell+1}^{(2)}$                & \dynkin C[1]{} $\sfC_{\ell}^{(1)}$                       & $(\mathrm{PSp}_{2\ell},\mathrm{Sp}_{2\ell})$                 \\
$(\sfE_6,\sfE_6\theta)$         & \dynkin E[2]{6} $\sfE_6^{(2)}$                  & \dynkin F[1]{4} $\sfF_4^{(1)}$                       & $(\sfF_4,\sfF_4)$                  \\
$(\mathrm{Spin}_8,\mathrm{Spin}_8 \theta_3)$    & \dynkin D[3]{4}   $\sfD_4^{(3)}$               & \dynkin G[1]{2}  $\sfG_2^{(1)}$                     & $(\sfG_2,\sfG_2)$                 \\
\midrule
$(\SL_3,\SL_3 \theta)$       & \dynkin A[2]{2}   $\sfA_2^{(2)}$              & \dynkin A[2]{2} $\sfA_2^{(2)}$                    & $(\SL_3,\SL_3 \vartheta)$       \\
$(\SL_{2\ell+1},\SL_{2\ell+1} \theta)$       & \dynkin A[2]{even}   $\sfA_{2\ell}^{(2)}$              & \dynkin A[2]{even} $\sfA_{2\ell}^{(2)}$                    & $(\SL_{2\ell+1},\SL_{2\ell+1} \vartheta)$       \\
\bottomrule
\end{tabular}
\end{adjustbox}
\label{tab:PairsWithDiagram}
\end{table}

\subsection{Strategy of the proofs}
The proofs of Theorems \ref{thm:duality_simplylaced}, \ref{thm:duality_twisted} and \ref{thm:duality_Aeven2} follow the same strategy, which is in fact very similar to the one for proving the analogous result for usual Hitchin fibrations. 

The first step is to understand the symmetries of any of the ``Hitchin-type" fibrations that we study. In general, all these fibrations are associated with the action of some reductive group $G$ on some variety $M$ and the natural map $[M/G]\rightarrow M\git G$. For any element $a$ of the  ``Hitchin base", the fibre of $a$ is acted on by the stack of $J_a$-torsors, where $J_a\rightarrow C$ is a group scheme naturally constructed from the group scheme $J\rightarrow M\git G$ of regular stabilizers of the action of  $G$ on $M$, which in turn is a descent of the group scheme $I\rightarrow M$ of stabilizers of the $G$-action. What makes the fibrations that we study specially nice is the existence of ``abelianization". By this we mean that there is a ``cameral cover" of $M\git G$ such that the group scheme $J$ can be described as an open and closed subscheme of the Weil restriction of a torus through the cameral cover. 

These ``cameral covers" are generally more complicated than the ``standard" cameral covers studied by Donagi and Gaitsgory \cite{Donagi-Gaitsgory}, which are étale locally pullbacks of the cover $\pi:\mathfrak{a}\rightarrow \mathfrak{a}/W_{H}$, where $\mathfrak{a}\subset \mathfrak{h}$ is the maximal Cartan subalgebra of a semisimple Lie algebra $\mathfrak{h}$ and $W_H$ is the corresponding Weyl group. However, as Hurtubise and Markman observed already in \cite{Hurtubise-Markman}, away from a codimension $2$ locus on $M\git G$, the cameral cover is indeed a cameral cover in the sense of Donagi-Gaitsgory. Cameral covers of this kind are determined by the semisimple Lie algebra $\mathfrak{h}$, and we observe that the Killing form matches cameral covers determined by $\mathfrak{h}$ with cameral covers determined by its dual Lie algebra $\mathfrak{h}^\vee$. If we prescribe the extra data of a semisimple group $H$ with Lie algebra $\mathfrak{h}$, then we can consider the following group schemes over $\mathfrak{a}/W_H$: first we let $J_H^1=\pi_*(\mathfrak{a} \times A)^{W_H}$ denote the invariant Weil restriction of the maximal torus $A$ of $H$ along the cameral cover; we can also consider the fibrewise connected component $J_H^0\subset J_H^1$; finally, we define $J_H$ as in section \ref{sec:cameral_simple_Galois} in terms of the kernels of the roots, regarded as characters $\alpha:A\rightarrow \Gm$. For all the fibrations that we study, with the exception of the case  $\sfA_{2\ell}^{(2)}$, the group scheme $J\rightarrow M\git G$ is identified, away from the codimension $2$ locus that we mention, with one of these $J_H$. In Table \ref{tab:PairsWithCameral}, we summarize the situation for the simple groups. In the case $\sfA_{2\ell}^{(2)}$, the group schemes associated with $(G,\theta)$ and with $(G,\vartheta)$ are open and closed subschemes of $J^1_{G^\theta}$ and $J^1_{G^\vartheta}$, respectively, but not isomorphic to any $J_H$.
 
\begin{table}[ht]
\caption{Pairs matched under duality, with the types of their cameral covers and of their regular centralizers. Horizontal lines separate the three situations: untwisted simply-laced vs untwisted simply-laced, twisted simply-laced vs untwisted non simply-laced and twisted vs twisted. The notation $(*)$ in the twisted vs twisted cases indicates that the regular centralizers are ``non-standard" in the sense that they are not of the form $J_H$, for some group $H$.}
\begin{adjustbox}{max width=\textwidth}
\begin{tabular}{@{}ccccc@{}}
\toprule
Pair              & Type of $J$ & Type of cameral cover & Type of dual $J$ & Dual pair \\ \midrule
$(\SL_r,\SL_r)$               & $\SL_r$ & $\sfA_r$ &    $\PGL_r$                                    & $(\PGL_r,\SL_r)$                  \\
$(\mathrm{SO}_{2r},\mathrm{Spin}_{2r})$             & $\mathrm{SO}_{2r}$       & $\sfD_r$          & $\mathrm{SO}_{2r}$                      & $(\mathrm{SO}_{2r},\mathrm{Spin}_{2r})$                \\
$(\mathrm{Spin}_{2r},\mathrm{Spin}_{2r})$           & $\mathrm{Spin}_{2r}$  & $\sfD_r$                 & $\mathrm{PSO}_{2r}$                       & $(\mathrm{PSO}_{2r},\mathrm{Spin}_{2r})$               \\
$(\sfE_6,\sfE_6)$                    & $\sfE_6$ & $\sfE_6$                 & $\sfE_6$                       & $(\sfE_6,\sfE_6)$                    \\
$(\sfE_7,\sfE_7)$                    & $\sfE_7$       & $\sfE_7$           & $\sfE_7$                       & $(\sfE_7,\sfE_7)$                    \\
$(\sfE_8,\sfE_8)$                    & $\sfE_8$  & $\sfE_8$                & $\sfE_8$                       & $(\sfE_8,\sfE_8)$                    \\
\midrule
$(\SL_{2\ell},\SL_{2\ell}\theta)$ & $\mathrm{Sp}_{2\ell}$ & $\sfC_\ell \sim \sfB_{\ell}$              & $\mathrm{SO}_{2\ell+1}$                     & $(\mathrm{SO}_{2\ell+1},\mathrm{Spin}_{2\ell+1})$                \\
$(\mathrm{Spin}_{2\ell+2},\mathrm{Spin}_{2\ell+2}\theta)$     & $\mathrm{Spin}_{2\ell+1}$    &  $\sfB_{\ell}\sim \sfC_{\ell}$            & $\mathrm{PSp}_{2\ell}$                       & $(\mathrm{PSp}_{2\ell},\mathrm{Sp}_{2\ell})$                 \\
$(\sfE_6,\sfE_6\theta)$         & $\sfF_4$ & $\sfF_4$                  & $\sfF_4$                       & $(\sfF_4,\sfF_4)$                  \\
$(\mathrm{Spin}_8,\mathrm{Spin}_8 \theta_3)$    & $\sfG_2$    & $\sfG_2$           & $\sfG_2$                     & $(\sfG_2,\sfG_2)$                 \\
\midrule
$(\SL_3,\SL_3 \theta)$       & $\mathrm{SO}_3$ $(*)$ & $\sfA_1$              & $\mathrm{Sp}_2$ $(*)$                   & $(\SL_3,\SL_3 \vartheta)$       \\
$(\SL_{2\ell+1},\SL_{2\ell+1} \theta)$       & $\mathrm{SO}_{2\ell+1}$ $(*)$ &  $\sfB_{\ell}\sim \sfC_{\ell}$              & $\mathrm{Sp}_{2\ell}$  $(*)$                  & $(\SL_{2\ell+1},\SL_{2\ell+1} \vartheta)$       \\
\bottomrule
\end{tabular}
\end{adjustbox}
\label{tab:PairsWithCameral}
\end{table}

Finally, our statements about the duality of the fibres follow from the main result of Donagi and Pantev \cite{Donagi-Pantev}: that the Picard stacks of $J_H$-torsors and of $J_{H^\vee}$-torsors are dual Beilinson $1$-motives. The conclusion for the case $\sfA_{2\ell}^{(2)}$ follows from a slight modification of this. Indeed, the corresponding $J_{(G,\theta)}$ and $J_{(G,\vartheta)}$ are not of the form $J_H$ for any $H$, but they are both open and closed subschemes of some group schemes of the form $J^1_H$ and  $J^1_{H^\vee}$, respectively. Duality then follows from the fact that the group $J^1_H/J_{(G,\theta)}$ is exchanged with the component group $J_{(G,\vartheta)}/J^0_{H^\vee}$ under the natural isomorphism $J^1_H/J^0_H \rightarrow J^1_{H^\vee}/J^0_{H^\vee}$.

\subsection{Global abelianized duality}
One of the most important properties of dual Beilinson $1$-motives is that the Fourier-Mukai functor determines an isomorphism between their corresponding bounded derived categories of quasi-coherent sheaves. This isomorphism respects some structure: namely, it exchanges skyscraper sheaves with line bundles naturally determined by the universal line bundle, and intertwines tensorization and convolution. 

The Beilinson $1$-motives we study are quite special: they are stacks of torsors under group schemes which are essentially Weil restrictions of tori through finite ramified covers. This allows us to define \emph{translation and tensorization operators} on these Picard stacks, associated with points of the cameral curve and characters and cocharacters of the corresponding tori. At the level of the derived categories of sheaves, translation and tensorization determine \emph{abelian Hecke and Wilson operators}. As Donagi and Pantev already noted in \cite{Donagi-Pantev}, the duality of Beilinson $1$-motives associated with cameral covers intertwines the abelian Hecke and Wilson operators.

The existence of a ``Hitchin section" in each situation determines the total space of the corresponding fibration, over the generic locus we are considering, as a relative Beilinson $1$-motive. The fibrewise isomorphisms from Theorems \ref{thm:duality_simplylaced}, \ref{thm:duality_twisted} and \ref{thm:duality_Aeven2} can then be extended as generic global isomorphisms of Beilinson $1$-motives. The abelian Hecke and Wilson operators can in turn be defined in terms of the total space (that is, in terms of multiplicative Higgs bundles, instead of in terms of cameral data). These are the \emph{abelianized} Hecke and Wilson operators. 

Let $h:\cM \rightarrow \cA$ and $\check{h}:\check{\cM}\rightarrow \cA$ any pair of ``dual" fibrations in the sense of Theorems \ref{thm:duality_simplylaced}, \ref{thm:duality_twisted} and \ref{thm:duality_Aeven2}, so that the generic fibres of $h$ and $\check{h}$ are dual Beilinson $1$-motives. We denote by $\cA^{\sharp}\subset \cA$ the corresponding ``generic locus" where the cameral curves are smooth and of Donagi-Gaitsgory type. The ``global abelianized duality" theorem is the following.

\begin{corol}
There exists an equivalence of derived categories
\begin{equation*}
\mathbf{S}: D^b(\QCoh(\cM/\cA^{\sharp})) \longrightarrow D^b(\QCoh(\check{\cM}/\cA^{\sharp}))
\end{equation*} 
which satisfies the following properties:
\begin{enumerate}
	\item \emph{(Normalization)}. The structure sheaf of the ``Hitchin section"  is mapped to the structure sheaf of the total space. 
	\item \emph{(Hecke compatibility)}. The equivalence intertwines Hecke and Wilson operators. 
\end{enumerate}
\end{corol}

\subsection{Mirror symmetry} In their seminal paper \cite{Hausel-Thaddeus}, Hausel and Thaddeus gave an interpretation of the duality of generic fibres of the Hitchin fibration for (usual, $K_C$-twisted) Higgs bundles as an example of Strominger--Yau--Zaslow (SYZ) \emph{mirror symmetry} \cite{SYZ}. In particular, this means that the Hitchin moduli spaces for Langlands dual groups should be mirror Calabi-Yau manifolds. A consequence of this is a certain symmetry in the $E$-polynomials of these two spaces. Such a symmetry was explicitly formulated and conjectured by Hausel and Thaddeus \cite{Hausel-Thaddeus} for the moduli spaces of $\PGL_n$-Higgs bundles and  (twisted) $\SL_n$-Higgs bundles with the name \emph{topological mirror symmetry}. Their conjecture was later verified by Groechenig, Wyss and Ziegler \cite{GWZ_Mirror} using techniques of $p$-adic integration. A different proof using perverse sheaves was given by Maulik and Shen \cite{Maulik-Shen_Mirror}. It is also worth mentioning that the SYZ mirror symmetry of Hitchin fibrations is an important ingredient in the alternative proof of Ngô's geometric stabilization theorem (the main result of \cite{Ngo_Lemme}, which implies the Fundamental Lemma of Langlands-Shelstad) provided by Groechenig, Wyss and Ziegler \cite{GWZ_Ngo}.

Mirror Calabi-Yau manifolds are also conjectured to satisfy \emph{homological mirror symmetry}: an equivalence between the derived category of coherent sheaves of one of the manifolds and the Fukaya category of the other. Kapustin and Witten \cite{Kapustin-Witten} provided an interpretation of the Fukaya category of Hitchin moduli space as the category of $D$-modules on the stack $\Bun_G$, in turn giving a physical explanation of the geometric Langlands program as mirror symmetry in Hitchin's moduli space. To account for the extra structures of geometric Langlands duality (normalization and Hecke compatibility), Kapustin and Witten explain that this mirror symmetry is in fact ``hyper-Kähler enhanced", in the sense that it tracks down the holomorphicity or Lagrangianity of branes with respect to each of the complex structures of Hitchin's moduli space. In particular, for example, branes of type (BBB), which are supported on holomorphic symplectic submanifolds, must be mirror partners to (BAA)-branes, supported on holomorphic Lagrangian submanifolds. We refer to Hausel's survey \cite{Hausel_Enhanced} for more details on this program.

A priori, multiplicative Hitchin fibrations are algebraically integrable systems only in the case in which $C$ is an elliptic curve, which was the original situation studied by Hurtubise and Markman. As explained by Elliott and Pestun, one can also obtain multiplicative Hitchin fibrations which are algebraically integrable systems by working over $\bbP^1$ with framings. Morally, this amounts to the base curve being $\bbA^1$ or $\Gm$, so precisely we obtain something hyper-Kähler in cases where the base curve is itself Calabi-Yau. It is worth mentioning that in these cases there is indeed a ``nonabelian Hodge theory", mediated by the Charbonneau--Hurtubise correspondence with singular monopoles \cite{Charbonneau-Hurtubise} (so with the Bogomolny equations playing the role of the Hitchin equations) and the correspondences with difference modules studied by Mochizuki \cite{Mochizuki_Doubly,Mochizuki_Periodic,Mochizuki_Triply} (which provide the ``de Rham side"). A ``Betti side" is also proposed in unpublished work of Kontsevich and Soibelman \cite{Kontsevich-Soibelman}.

In these cases where the multiplicative Hitchin fibration is an algebraically integrable system, we can indeed interpret our results as instances of SYZ mirror symmetry. Homological mirror symmetry for the moduli space of multiplicative Higgs bundles should be related to the ``multiplicative geometric Langlands program" proposed by Elliott and Pestun \cite{Elliott-Pestun}. We refrain from trying to formulate a ``multiplicative analogue" of topological mirror symmetry, but nonetheless comment that multiplicative Hitchin fibrations play a fundamental role in G. Wang's proof of the Fundamental Lemma for the spherical Hecke algebras.  

\subsection{Physical interpretation} Kapustin and Witten \cite{Kapustin-Witten} explain that the geometric Langlands conjecture can be understood as a consequence of the S-duality between two $4$-dimensional supersymmetric Yang-Mills quantum field theories. Elliott and Pestun similarly formulate their ``multiplicative geometric Langlands program" as a physical consequence of the S-duality of $5$-dimensional field theories. More generally, S-duality of both the $4$-dimensional and $5$-dimensional theories are conjectured to be consequences of these theories arising from different dimensional reductions of a certain $6$-dimensional theory \cite{Witten_6D}. The physics papers of Tachikawa \cite{Tachikawa} and Duan--Lee--Nahmgoong--Wang \cite{Duanetal} explore this in detail, and have been great sources of inspiration for this paper. It is worth noting that these physics papers give a physical interpretation of the special features of the case $\sfA_{2\ell}^{(2)}$ as the manifestation of the presence of a ``discrete theta-angle", which is not detected by the $4$-dimensional theories.

\subsection{Contents of the paper} This paper consists of 6 sections (including this introduction) and 3 appendices. 

In section \ref{sec:AbelianCameral} we recall the ``abelian duality" associated with cameral covers of Donagi--Gaitsgory type with simple Galois ramification. We start by recalling some basic notions about Picard stacks in general and Beilinson $1$-motives in particular. We then recall the definition of cameral cover given by Donagi and Gaitsgory \cite{Donagi-Gaitsgory}, and the main results of Donagi and Pantev \cite{Donagi-Pantev} concerning the duality of certain Beilinson $1$-motives associated with cameral covers. We also explain how this duality behaves with respect to the abelian Hecke and Wilson operators, and how the general theory of abelianization of Hitchin-type fibrations works. We finish the section by providing once again a summary of the dualities we prove in this paper.

In section \ref{sec:UntwistedMultHitchin} we review the theory of untwisted multiplicative Hitchin fibrations as developed in the papers of Bouthier, J. Chi and G. Wang \cite{Bouthier_Dimension,Bouthier_Hitchin,JChi,Griffin_Lemma}. We start by giving the main definitions and explaining the formulation in terms of reductive monoids. We also explain the construction of the Steinberg quasi-section and its extension to monoids. We continue by recalling the study of the symmetries of multiplicative Hitchin fibrations and their cameral description. We finish by proving the first of our results, Theorem \ref{thm:duality_simplylaced} (=Theorem \ref{thm:duality_simplylaced_text}): the duality in the simply-laced case. We also include the example of $\SL_2$ and $\PGL_2$, worked out in  detail.

Section \ref{sec:TwistedInvTheory} is dedicated to twisted invariant theory. We start by explaining folding and the construction of the invariant and coinvariant groups. We continue by reviewing the main results of invariant theory for the twisted conjugation action, originally formulated by Mohrdieck and reviewed and extended in unpublished work by G. Wang \cite{Griffin_Twisted}. We then explain the construction of a ``coinvariant monoid", which is relevant for our results and review G. Wang's construction of ``twisted Steinberg sections". We introduce G. Wang's cameral description of the regular stabilizer associated with the twisted adjoint action. We finish the section by explaining the subtleties of the case $\sfA_{2\ell}^{(2)}$.

In section \ref{sec:TwistedMultHitchin} we introduce the global incarnation of twisted invariant theory: twisted multiplicative Hitchin fibrations. This section roughly mirrors section \ref{sec:UntwistedMultHitchin}, by explaining the global Steinberg section, the symmetries and the cameral description of twisted multiplicative Hitchin fibrations. In the last two subsections we prove our second and third results, Theorem \ref{thm:duality_twisted} (=Theorem \ref{thm:duality_twisted_text}) and Theorem \ref{thm:duality_Aeven2} (=Theorem \ref{thm:duality_Aeven2_text}): duality for twisted multiplicative Hitchin fibrations, in the twisted vs. untwisted case, and in the special case $\sfA_{2\ell}^{(2)}$.

Our last section \ref{sec:A22} is dedicated to the detailed study of the case $\sfA_{2}^{(2)}$. Twisted multiplicative Higgs bundles in this case can be interpreted as \emph{twisted bilinear bundles} of rank $3$. The twisted multiplicative Hitchin fibration is then a  Hitchin-type fibration which is associated with these bilinear bundles, and we can explicitly describe the fibres in terms of spectral data. Twisted bilinear bundles up to rank $3$ are related to the \emph{conic bundles} studied by Gómez and Sols in \cite{Gomez-Sols} as a preliminary step in the study the moduli of orthogonal bundles. A Hitchin-type fibration for rank $2$ (symmetric) bilinear bundles and the corresponding spectral data were studied by Gothen and Oliveira \cite{Gothen-Oliveira}. The results of this section can then be interpreted as a rank $3$ analogue of their results. This section was deeply inspired and in fact follows very closely the arguments of section 4 in Hitchin's paper \cite{Hitchin_G2}.

The paper contains 3 appendices. The first appendix, appendix \ref{app:MultLanglands} is dedicated to review the ``multiplicative nonabelian Hodge theory" developed by Mochizuki \cite{Mochizuki_Doubly,Mochizuki_Periodic,Mochizuki_Triply}  and to give a rather informal explanation of the ``multiplicative geometric Langlands program" of Elliott and Pestun \cite{Elliott-Pestun}. Appendix \ref{app:Proof} is included in order to make the paper as self-contained as possible, and explains the main topological arguments behind the duality of Beilinson $1$-motives associated with cameral covers. The proof explained in this appendix combines the arguments of the papers of Donagi--Pantev \cite{Donagi-Pantev} and of Chen--Zhu \cite{Chen-Zhu}. Finally, in appendix \ref{app:Prym}, we give a succint review of Prym varieties which is important to understand section \ref{sec:A22}. It is worth mentioning that the construction of an abelian variety associated with a subset of the ramification divisor explained in section \ref{sec:AbelianVariety_Subset} is original, and that we have not found any other previous appearance of these varieties in the literature.

\subsection{Acknowledgements} The main project leading to this paper was started in collaboration with Benedict Morrissey, through several conversations initiated in the autumn of 2022 and continued during the spring and summer of 2023. Benedict's insights and conjectures, the references provided by him and our joint discussions were crucial for the development of this project at its early stages. I am thus deeply grateful to him and would like to remark his contribution to the results of this paper.

I also want to acknowledge the tremendous help provided by Griffin Wang. As the reader might notice, many of the results of this paper depend directly on his study of twisted invariant theory and twisted multiplicative Hitchin fibrations developed in his unpublished work \cite{Griffin_Twisted}. Griffin has been very kind by providing early access to his work and by being available for questions and discussions. 

Many of the conversations leading to this project, as well as some early expositions of parts of this paper, have been mantained under the support of the Hitchin--Ngô Laboratory at ICMAT. Thus I would like to thank the two chairmen of the Laboratory as well as its coordinator, Oscar García-Prada. I would also like to thank Oscar García-Prada and Jacques Hurtubise for many comments, discussions and references.

\section{Review of abelian duality and cameral covers} \label{sec:AbelianCameral}
\subsection{Recollections on Picard stacks} \label{sec:PicardStacks}
Recall that a (strictly commutative) \emph{Picard groupoid} is a (small) symmetric monoidal groupoid such that every object has a monoidal inverse. A \emph{Picard stack} or \emph{commutative group stack} over a scheme $S$ is a stack over the fpqc site $S_{\mathrm{fpqc}}$ taking values in the category of Picard groupoids. Picard stacks form a $2$-category, which by a theorem of Deligne \cite{Deligne} is equivalent to the $2$-category whose objects are $2$-term complexes of sheaves of abelian groups $(K_{-1}\rightarrow K_0)$ such that $K_{-1}$ is injective, and whose $1$-morphisms and $2$-morphisms are given by chain maps and homotopies, respectively. We fix an inverse $(-)^\flat$ of this equivalence; that is, to each Picard stack $\cP$ we associate a $2$-term complex $\cP^\flat=(K_{-1}\rightarrow K_0)$ with $K_{-1}$ injective. For any such complex $K$, we denote by $\mathrm{ch}(K)$ the corresponding Picard stack.

Every Picard stack $\cP$ has an \emph{inertia group} $\Aut_0(\cP)$ and a \emph{coarse moduli space} $|\cP|$. We can also consider the neutral connected component  $P=|\cP|^0$ and the group of connected components  $\pi_0(\cP)$. Under Deligne's equivalence, we obtain
\begin{equation*}
|\cP| = H^0(\cP^\flat) \text{ and } \Aut_0(\cP)=H^{-1}(\cP^\flat).
\end{equation*} 

\begin{ej}
Let $J\rightarrow S$ be a commutative group over  $S$ and let $\sJ$ denote its sheaf of sections. We can realize $J$ as a Picard stack $\cP=J$ by identifying it with the complex concentrated in degree $0$, that is, we set $J^\flat=(0\rightarrow \sJ)$. We can also consider $BJ$, the classifying stack of $J$-torsors, by taking $(BJ)^\flat$ to be a complex quasi-isomorphic to the shifted complex $\sJ[1]=(\sJ\rightarrow 0)$. If $p:X\rightarrow S$ is another scheme over  $S$, we can consider the \emph{relative Picard stack of $J$-torsors over $X$} to be
\begin{equation*}
\Bun_{J/X}=p_* BJ := \mathrm{ch}(\tau_{\leq 0} \mathrm{R}p_*(\sJ[1])),
\end{equation*} 
where $\tau_{\leq 0}$ denotes the standard truncation to degrees $\leq 0$.
In particular, $\Aut_0(\Bun_{J/X})=p_* \sJ$ and $|\Bun_{J/X}|=\mathrm{R}^1 p_* \sJ$. At the level of $S$-points, we have
\begin{equation*}
\Aut_0(\Bun_{J/X})(S)= H^0(X,J) \text{ and } |\Bun_{J/X}|(S)=H^1(X,J).
\end{equation*} 
When $J=\Gm$ is the multiplicative group over  $S$, we recover the relative Picard stack
 \begin{equation*}
\Pic_{X/S} = \Bun_{\Gm/X}.
\end{equation*} 
\end{ej}

The \emph{dual Picard stack} of a Picard stack $\cP$ is defined as
\begin{equation*}
\bD(\cP):= \underline{\Hom}(\cP,B\Gm) =\mathrm{ch}(\tau_{\leq 0}\mathrm{R}\underline{\Hom}(\cP^\flat,\sO^*_S[1])).
\end{equation*} 
There is a canonical $1$-morphism $\cP \rightarrow \bD(\bD(\cP))$. When this morphism is an isomorphism, the Picard stack $\cP$ is said to be \emph{dualizable}. In that case, there is the universal \emph{Poincaré line bundle} $\mathbf{P}\rightarrow \cP \times \bD(\cP)$, and we can define the \emph{Fourier-Mukai functor}
\begin{align*}
\Phi_{\cP}: D^b(\QCoh(\cP)) & \longrightarrow D^b(\QCoh(\bD(\cP))) \\
\sF & \longmapsto \mathrm{R}p_{2,*}(\mathrm{L}p_1^* \sF \otimes \mathbf{P}),
\end{align*} 
where $p_1$ and $p_2$ are the natural projections from $\cP \times \bD(\cP)$. When this functor is an equivalence of categories, we say that $\cP$ is a \emph{good Picard stack} in the sense of Arinkin \cite[Appendix A]{Donagi-Pantev_Arinkin}.

An immediate property of Fourier-Mukai transforms is that, since $\mathbf{P}$ is the universal line bundle, a skyscraper sheaf $\sO_{x}$ at a point $x$ of  $\cP$ corresponding to a line bundle $\mathscr{L}^x\rightarrow \bD(\cP)$ is mapped to the line bundle
 \begin{equation*}
\Phi_{\cP}(\sO_{x}) = \mathbf{P}|_{\left\{x\right\}\times \bD(\cP)}= \mathscr{L}^x.
\end{equation*} 

\begin{ej}
Duality for Picard stacks unifies several types of duality associated with some commutative group schemes. In particular, it generalizes the duality of abelian varieties. Namely, if $P\rightarrow S$ is an abelian scheme, then
 \begin{equation*}
\bD(P)= \underline{\Hom}(P,B\Gm) = \mathrm{Ext}^1(P,\Gm)=\hat{P},
\end{equation*} 
is the dual abelian scheme of $P$. It is a classical result of Mukai that the Fourier-Mukai transform for abelian schemes determines an equivalence.

Duality of Picard stacks also generalizes Cartier duality. Recall that if $\Lambda$ is a finitely generated abelian group over $S$, its Cartier dual $\Lambda^*$ is the affine group scheme $\Spec(\sO_S[\Lambda])$, which represents the sheaf $U\mapsto \Hom_U(\Lambda \times_S U,\mathbb{G}_{\mathrm{m},U})$. One can easily show that
\begin{equation*}
\bD(\Lambda) = B\Lambda^*.
\end{equation*} 
Conversely, a group scheme $J$ is \emph{of multiplicative type} if $J=\Lambda^*$ is the Cartier dual of some finitely generated abelian group over  $S$. The group $\Lambda$ can be recovered as the character group $\Lambda=\Hom(J,\Gm)$. Again, we have 
\begin{equation*}
\bD(BJ) = \Lambda.
\end{equation*} 
The Fourier-Mukai transform in this case is determined by the correspondence between characters and irreducible representations, and thus it also gives an equivalence of categories.

A less trivial example is the Picard stack $\Pic_{C}$ of a smooth complex projective curve of genus $g$. In this case we can write  $\Pic_{C}= B\Gm \times \Jac(C) \times \Z$, where $\Jac(C)$ is the \emph{Jacobian} of the curve, an abelian variety of dimension $g$ parametrizing degree  $0$ line bundles on  $C$. The Jacobian of a curve is well known to be self-dual, which makes $\Pic_{C}$ a self-dual Picard stack. More precisely, there is a canonical isomorphism
\begin{equation*}
\Pic_C \overset{\sim}{\longrightarrow} \bD(\Pic_C)
\end{equation*} 
mediated by the universal line bundle $\mathscr{U}\rightarrow C \times \Pic_C$, with fibres $\mathscr{U}|_{C\times \left\{L\right\}}=L$. In particular, a line bundle of the form $\sO_C(p)$ is sent to $\mathscr{U}|_{\left\{p\right\}\times \Pic_C}\rightarrow \Pic_C$.

More generally, we can consider an algebraic torus $T$, and the stack of  $T$-bundles
\begin{equation*}
\Bun_{T/C}= BT \times (\Jac(C)\otimes X_*(T)) \times X_*(T).
\end{equation*} 
It follows from the self-duality of $\Jac(C)$ that its dual Picard stack is the stack of $T^\vee$-bundles
\begin{equation*}
\bD(\Bun_{T/C})\cong X^*(T) \times (\Jac(C)\otimes X^*(T)) \times BT^\vee = \Bun_{T^\vee/C}.
\end{equation*} 
Again, there is the canonical isomorphism
\begin{equation*}
\Bun_{T/C} \overset{\sim}{\longrightarrow} \bD(\Bun_{T^\vee/C})
\end{equation*} 
mediated by the universal $T^\vee$-bundle $\mathscr{U}\rightarrow C \times \Bun_{T^\vee/C}$. Let $\lambda \in X_*(T)$ be a cocharacter and let $\sO_C(\lambda p)$ denote the $T$-bundle determined by taking $z^\lambda$ as transition function near $p$. Then the canonical isomorphism sends this $T$-bundle $\sO_C(\lambda p)$ to the line bundle $\mathscr{U}_\lambda \rightarrow \Bun_{T^\vee/C}$ with fibres
\begin{equation*}
	(\mathscr{U}_\lambda)_L = L \times^{T^\vee,\lambda} \bbA^1,
\end{equation*} 
where we understand $\lambda \in X_*(T)=X^*(T^\vee)$ as a character of $T^\vee$.
\end{ej}

The examples considered above are particular cases of \emph{Beilinson $1$-motives}. These are Picard stacks $\cP$ admitting a  $2$-step filtration $W_\bullet \cP$
 \begin{equation*}
W_{-1}=0 \subset W_0 \subset W_1 \subset W_2=\cP
\end{equation*} 
such that
\begin{itemize}
	\item $\mathrm{Gr}^W_0 \cong BJ$ is isomorphic to the classifying stack of a group $J$ of multiplicative type,	
	\item $\mathrm{Gr}^W_1\cong P$ is an abelian scheme, and
	\item $\mathrm{Gr}^W_2\cong \Lambda$ is a finitely generated abelian group.
\end{itemize}
Note that, if $\cP$ is a Beilinson $1$-motive, then $\Aut_0(\cP)=\pi_0(\bD(\cP))^*$. 

Beilinson $1$-motives are good in the sense of Arinkin. Moreover, on a Beilinson $1$-motive $\cP$, the Fourier-Mukai transform $\Phi_\cP$ intertwines tensorization with the \emph{convolution product}
\begin{align*}
*: D^b(\QCoh(\cP)) \times D^b(\QCoh(\cP)) & \longrightarrow  D^b(\QCoh(\cP)) \\
(\sF_1,\sF_2) & \longmapsto \R m_* (p_1^*\sF_1 \otimes p_2^*\sF_2),
\end{align*} 
where $p_1,p_2:\cP \times \cP \rightarrow \cP$ are the natural projections, and $m:\cP \times \cP \rightarrow \cP$ is the group law on  $\cP$. More precisely, there are canonical isomorphisms
 \begin{equation*}
\Phi_\cP(\sF_1 * \sF_2) \cong \Phi_{\cP}(\sF_1) \otimes \Phi_{\cP}(\sF_2)
\end{equation*} 
 \begin{equation*}
\Phi_\cP(\sF_1 \otimes \sF_2) = (\Phi_{\cP}(\sF_1) * \Phi_{\cP}(\sF_2)) \otimes \omega_{\bD(\cP)/S}[g],
\end{equation*} 
for $\omega_{\bD(\cP)/S}$ the dualizing sheaf of $\bD(\cP)$ over $S$ and $g$ the relative dimension of $\cP/S$.

\subsection{Cameral covers with simple Galois ramification} \label{sec:cameral_simple_Galois}
Let $\Phi$ be a root system and let $W$ denote its Weyl group. We let $\mathbf{R}(\Phi)$ denote the root lattice and write $\ft=\mathbf{R}(\Phi)^\vee\otimes_{\Z}\C$. This $\ft$ is a vector space acted on by $W$. The projection $\ft\rightarrow \ft/W$ is the \emph{standard cameral cover} associated with $\Phi$.
More generally, a \emph{cameral cover} associated with the root system $\Phi$ is a finite flat morphism of $\C$-schemes $\tilde{X}\rightarrow X$ with Galois group $W$ and such that, étale-locally on $X$, it is a pull-back of the standard cameral cover $\ft \rightarrow \ft/W$. 
If $\pi:\tilde{X}\rightarrow X$ is a cameral cover, we denote by $B\subset X$ the branch locus of $\pi$ and, for each point $b\in B$, we denote by $W_b$ the corresponding inertia group.

We say that the cameral cover $\pi$ has \emph{simple Galois ramification} if
\begin{enumerate}
	\item for each point $\tilde{x}\in \tilde{X}$, the stabilizer $W_{\tilde{x}}=\left\{w\in W: w\cdot \tilde{x}=\tilde{x}\right\}$	has order at most $2$,
	\item for each root $\alpha \in \Phi$, fixed point subset
		\begin{equation*}
		\tilde{X}_\alpha = \left\{\tilde{p}\in \tilde{X}: s_\alpha(\tilde{p})=\tilde{p}\right\}
		\end{equation*} 
		is non-empty.
\end{enumerate} 
Here, we denote by $s_\alpha \in W$ the reflection associated with the root  $\alpha$. If a cameral cover has simple Galois ramification and we fix the positive roots $\Phi_+\subset \Phi$, we can describe the branch locus as a union
\begin{equation*}
B = \bigcup_{\alpha \in \Phi_+} B_\alpha, \text{ for } B_\alpha=\pi(\tilde{X}_\alpha),
\end{equation*} 
and there is a natural isomorphism of the inertia group $W_b$, for $b\in B_{\alpha}$, with the subgroup of $W$ generated by $s_\alpha$. Moreover, $\tilde{X}$ is smooth if and only if the divisors $\tilde{X}_{\alpha}$ (and in turn the $B_\alpha$) are pairwise disjoint.

Let $G$ be a semisimple group with maximal torus $T\subset G$ such that  $\Phi(G,T)=\Phi$.
We can consider the Weil restriction
\begin{equation*}
\Pi=\Pi_{\pi,G} := \pi_*(\tilde{X}\times T)
\end{equation*} 
and the invariant subgroup schemes
\begin{equation*}
J^1=J^1_{\pi,G}:=\Pi^W \text{ and } J^0=J^0_{\pi,G}:=\Pi^{W,0},
\end{equation*} 
where the superindex $0$ stands for taking fibrewise connected components. 
We can also consider the subgroup $J=J_{\pi,G}\subset J^1$ characterized as follows: for every  $X$-scheme $S$, the group  $J(S)$ is the subgroup of $J^1(S)$ consisting of morphisms $g:S\times_X \tilde{X}\rightarrow T$ such that, for every $\tilde{x} \in (S\times_X \tilde{X})(\C)$ with $s_\alpha(\tilde{x})=\tilde{x}$ for some root $\alpha \in \Phi$, we have  $\alpha(g(x))=1$. There are inclusions
\begin{equation*}
J^0 \subset J \subset J^1.
\end{equation*}  
These are group schemes over $X$ which coincide away from the branch locus of the cameral cover. If we assume that $\pi$ has simple Galois ramification then, over a branch point $b\in B_{\alpha}\subset X$, the fibres of these groups are commutative algebraic groups over $\C$ whose maximal reductive quotient is given by
\begin{equation*}
	(J^0_b)^{\mathrm{red}}=T^{s_\alpha,0} \subset (J_b)^{\mathrm{red}}=\ker(\alpha:T\rightarrow \Gm) \subset (J^1_b)^{\mathrm{red}}=T^{s_\alpha}.
\end{equation*} 

We recall the description of the groups of the form $\ker \alpha$ from \cite{Donagi-Pantev}. For every root $\alpha \in \Phi$ we can define the numbers  $\varepsilon_\alpha$ and $\varepsilon_\alpha^\vee$ to be the positive generators of the images of the maps  $\langle \alpha,-\rangle:X_*(T)\rightarrow \Z$ and $\langle \alpha^\vee,-\rangle: X^*(T)\rightarrow \Z$, respectively. There is the following commutative diagram, with short exact rows and columns
\begin{center}
\begin{tikzcd}
 T^{s_\alpha,0} \dar \rar & \ker \alpha \rar \dar & \Z/\varepsilon_\alpha \dar  \\ 
 T^{s_\alpha,0} \dar  \rar & T^{s_\alpha} \dar \rar & \Z/\varepsilon_\alpha \varepsilon_{\alpha^\vee}\dar \\
        1 \rar & \Z/\varepsilon_{\alpha^\vee} \rar & \Z/\varepsilon_{\alpha^\vee}.  
\end{tikzcd}
\end{center}
Recall that, if $G$ is simple, then we have
\begin{align*}
	\varepsilon_\alpha &= 
\begin{cases}
	2 \text{ if } G=\mathrm{Sp}_r \text{ and $\alpha$ is a long root}, \\	
	1 \text{ in all other cases },
\end{cases}
\\
	\varepsilon^\vee_\alpha &= 
\begin{cases}
	2 \text{ if } G=\mathrm{SO}_{2r+1} \text{ and $\alpha$ is a short root}, \\	
	1  \text{ in all other cases }.
\end{cases}
\end{align*} 
In particular, it is important to remark that $\varepsilon_\alpha=2$ for a root of $\SL_2$.
For general semisimple $G$, we make the following observations
\begin{enumerate}
	\item if $r>1$ and the Dynkin diagram of $G$ is simply-laced, then $T^{s_{\alpha}}$ is connected,
	\item if $G$ is simply-connected, then $T^{s_\alpha}=\ker \alpha$,
	\item if $G$ is of adjoint type, then  $T^{s_\alpha,0}=\ker \alpha$.
\end{enumerate}
Globally, we obtain the following diagram, with short exact rows and columns
\begin{center}
\begin{tikzcd}
 J^0 \dar \rar & J \rar \dar & \bigoplus_{\alpha \in \Phi_+} B_\alpha \times \Z/\varepsilon_{\alpha} \dar  \\ 
 J^0 \dar  \rar & J^1 \dar \rar & \bigoplus_{\alpha \in \Phi_+} B_\alpha \times \Z/\varepsilon_{\alpha} \varepsilon_{\alpha^\vee}\dar \\
        1 \rar &\bigoplus_{\alpha \in \Phi_+}B_\alpha \times \Z/\varepsilon_{\alpha^\vee} \rar &\bigoplus_{\alpha \in \Phi_+} B_\alpha \times \Z/\varepsilon_{\alpha^\vee}.  
\end{tikzcd}
\end{center}
We make the following observations
\begin{enumerate}
	\item if $r>1$ and the Dynkin diagram of $G$ is simply-laced, then $J^0=J=J^1$,	
	\item if $G$ is simply-connected, then  $J^1=J$,
	\item if $G$ is of adjoint type, then $J^0=J$.
\end{enumerate}

\subsection{The dual group}
Let $G^\vee$ denote the Langlands dual group of $G$. This group has as maximal torus the dual torus $T^\vee=X_*(T)^*$, and as root system the dual root system $\Phi(G^\vee,T^\vee)=\Phi^\vee$. Recall that the Killing form determines a $W$-equivariant isomorphism $\ft^* \overset{\sim}{\rightarrow} \ft$, $x\mapsto (x,-)$. We pick the usual normalization of the Killing form in which short roots $\alpha$ are imposed to satisfy $(\alpha,\alpha)=2$. Choosing this normalization, the Killing form sends short roots to long coroots, and long roots to multiples of short coroots. Therefore, the root lattice $\mathbf{R}(\Phi)$ is matched with $\mathbf{R}(\Phi^\vee)$ and the weight lattice $\mathbf{P}(\Phi)$ is matched with $\mathbf{P}(\Phi^\vee)$. The isomorphism also allows us to regard the cameral cover $\pi:\tilde{X}\rightarrow X$ as a cameral cover associated with $\Phi^\vee$. Therefore, we can define the group schemes
\begin{equation*}
\check{J}^0=J^0_{\pi,G^\vee} \subset \check{J}=J_{\pi,G^\vee} \subset \check{J}^1=J^1_{\pi,G^\vee},
\end{equation*} 
by replacing $T$ with $T^\vee$. Moreover, note that the Killing form also identifies 
\begin{enumerate}
	\item the component group $J^1/J^0$ with the component group  $\check{J}^1/\check{J}^0$,	
	\item the component group $J/J^0$ with the quotient  $\check{J}^1/\check{J}$.	
\end{enumerate}

\subsection{Cameral curves and abelian duality} \label{sec:CameralCurvesAbelianDuality}
Fix $G$ a complex semisimple group and $T\subset G$ a maximal torus. Let $\tilde{C}$ and $C$ be smooth complex projective curves and let $\pi:\tilde{C}\rightarrow C$ be a cameral cover associated with $\Phi(G,T)$ with simple Galois ramification. We can consider the groups $J^0,J,J^1$ and $\check{J}^0,\check{J},\check{J}^1$ from the previous section, and their associated Picard stacks of torsors
\begin{align*}
\cP^0=\Bun_{J^0/C}, \ \ \cP=\Bun_{J/C}, \ \ \cP^1=\Bun_{J^1/C}, \\
\check{\cP}^0=\Bun_{\check{J}^0/C}, \ \ \check{\cP}=\Bun_{\check{J}/C}, \ \ \check{\cP}^1=\Bun_{\check{J}^1/C}. 
\end{align*} 

The inclusion $J^1 \hookrightarrow \Pi$ induces a morphism of Picard stacks
\begin{equation*}
\iota: \cP^1 \longrightarrow \Bun_{T/\tilde{C}}.
\end{equation*} 
Dually, we get a morphism
\begin{equation*}
\bD(\iota): \bD(\Bun_{T/\tilde{C}}) \longrightarrow \bD(\cP^1).
\end{equation*} 
On the other hand, we have the \emph{norm map}
\begin{align*}
\Nm: T^\vee & \longrightarrow T^\vee \\
t & \longmapsto \prod_{w\in W} w\cdot t.
\end{align*} 
This map induces a morphism $\Pi \rightarrow \check{J}^0$, and in turn a morphism of Picard stacks
\begin{equation*}
\Nm: \Bun_{T^\vee/\tilde{C}} \longrightarrow \check{\cP}^0.
\end{equation*} 
Recall that we have the canonical isomorphism
\begin{equation*}
\bD(\Bun_{T/\tilde{C}}) \overset{\sim}{\longrightarrow} \Bun_{T^\vee/\tilde{C}}.
\end{equation*} 
Consider the map $\mathrm{S}_0: \bD(\cP^1)\rightarrow \check{\cP}^0$ induced by the diagram
\begin{center}
\begin{tikzcd}
\bD(\Bun_{T/\tilde{C}}) \ar{r}{\sim} \ar{d}{\bD(\iota)} & \Bun_{T^\vee/\tilde{C}} \ar{d}{\Nm} \\
\bD(\cP^1) \ar{r}{\mathrm{S}_0} & \check{\cP}^0. 
\end{tikzcd}
\end{center}

The main result of \cite{Donagi-Pantev} is the following.

\begin{thm} \label{thm:duality}
The Picard stacks $\cP^1$ and  $\check{\cP}^0$ are Beilinson $1$-motives and the map $$\mathrm{S}_0:\bD(\cP^1)\rightarrow \check{\cP}^0$$ is an isomorphism. In particular, the Fourier-Mukai transform $\Phi_{\cP^1}$ induces an equivalence of categories
\begin{equation*}
\bS^{\ab}_0=(\mathrm{S}_0^{-1})^*\circ \Phi_{\cP^1}: D^b(\QCoh(\cP^1)) \longrightarrow D^b(\QCoh(\check{\cP}^0)).
\end{equation*} 
\end{thm}

More generally, they prove the following.

\begin{corol} \label{corol:duality}
The Picard stacks $\cP=\Bun_{J/C}$ and  $\check{\cP}=\Bun_{\check{J}/C}$ are Beilinson $1$-motives and the isomorphism $\mathrm{S}_0:\bD(\cP^1)\rightarrow \check{\cP}^0$ induces a natural isomorphism $\mathrm{S}:\bD(\cP)\rightarrow \check{\cP}$. In particular, the Fourier-Mukai transform $\Phi_{\cP}$ determines an equivalence of categories
\begin{equation*}
\bS^{\ab}=(\mathrm{S}^{-1})^* \circ \Phi_{\cP}: D^b(\QCoh(\cP)) \longrightarrow D^b(\QCoh(\check{\cP})).
\end{equation*} 
\end{corol}

\begin{proof}[Proof of Corollary \ref{corol:duality}]
Note that we have an exact sequence (i.e. a distinguished triangle on the associated complexes) of Picard stacks
\begin{center}
\begin{tikzcd}
	0 \rar & H^0(C,J^1/J) \rar & \cP \rar & \cP^1 \rar & 0.
\end{tikzcd}
\end{center}
Dualizing, we obtain an exact sequence
\begin{center}
\begin{tikzcd}
	0 \rar & H^0(C,J^1/J) \rar & \bD(\cP^1) \rar & \bD(\cP) \rar & 0.
\end{tikzcd}
\end{center}
Now, since $\mathrm{S}_0:\bD(\cP^1)\rightarrow \check{\cP}^0$ induces the isomorphism $H^0(C,J^1/J^0) \rightarrow H^0(C,\check{J}^1/\check{J}^0)$ (see appendix \ref{app:Proof}), applying $\mathrm{S}_0$ to the above exact sequence yields
\begin{center}
\begin{tikzcd}
	0 \rar & H^0(C,\check{J}/\check{J^0}) \rar & \check{\cP}^0 \rar & \mathrm{S}_0(\bD(\cP))) \rar & 0,
\end{tikzcd}
\end{center}
so we conclude that $\mathrm{S}_0$ induces an isomorphism $\bD(\cP)\rightarrow \check{\cP}$.
\end{proof}

Theorem \ref{thm:duality} is the most important duality theorem for Picard stacks that we apply in this paper. The proof of this result is due to Donagi and Pantev \cite{Donagi-Pantev}. Chen and Zhu \cite{Chen-Zhu} proved an analogue for positive characteristic. In an aim to make this paper as self-contained as possible, and in order to not leave too many details for the reader to fill in, we dedicate appendix \ref{app:Proof} to review the main arguments of the proof.

\begin{rmk} \label{rmk:dualJ}
More generally, we can consider a commutative group scheme $J' \rightarrow C$ which is fiberwise open and closed inside $J^1$. 
We can then define the group $\check{J}' \rightarrow C$ as the open and closed subscheme of $\check{J}^1$ such that the isomorphism $J^1/J^0 \rightarrow \check{J}^1/\check{J}^0$ matches $J^1/J'$ with $\check{J}'/\check{J}^0$. The same proof of Corollary \ref{corol:duality} shows that the Picard stacks $\cP'=\Bun_{J'/C}$ and $\check{\cP}'=\Bun_{\check{J}'/C}$ are dual Beilinson $1$-motives.
\end{rmk}

\subsection{Abelian Hecke compatibility}
A very important property of geometric Langlands duality is the compatibility of tensorization (or Wilson) operators with Hecke (or t'Hooft) operators. An ``abelian version" of this compatibility is already detectable from Theorem \ref{thm:duality} and the properties of the Fourier-Mukai transform on Beilinson $1$-motives.

With a point $\tilde{p} \in \tilde{C}$ and a cocharacter $\lambda \in X_*(T)$ we can associate a $T$-bundle $\sO_{\tilde{C}}(\lambda \tilde{p})$ on $\tilde{C}$. The natural map $\Bun_{T/\tilde{C}}\rightarrow \cP$ induced from the norm map sends this bundle to a $J$-torsor on $C$ that we denote by $\sS_{\lambda,\tilde{p}}$. The \emph{translation operator} associated with $(\lambda,\tilde{p})$ is the morphism
\begin{align*}
\mathbf{Trans}_{\lambda,\tilde{p}}: \cP & \longrightarrow \cP \\
\mathscr{E} & \longmapsto \mathscr{E} \otimes\sS_{\lambda,\tilde{p}}.
\end{align*} 
The \emph{abelian Hecke operator} associated with $(\lambda,\tilde{p})$ is the morphism of derived categories obtained by taking convolution with the skyscraper sheaf $\sO_{\sS_{\lambda,\tilde{p}}}$
\begin{align*}
\bH^{\ab}_{\lambda,\tilde{p}}:D^b(\QCoh(\cP)) & \longrightarrow D^b(\QCoh(\cP)) \\
\sF & \longmapsto \sF * \sO_{\sS_{\lambda,\tilde{p}}}.
\end{align*} 
Equivalently, $\bH^{\ad}_{\lambda,\tilde{p}}$ is the integral transform defined by the structure sheaf of the graph of $\mathbf{Trans}_{\lambda,\tilde{p}}$ inside $\cP\times \cP$.

Dually, with a point $\tilde{p}\in \tilde{C}$ and a character $\chi \in X^*(T)$ we can associate the line bundle $L_{\chi,\tilde{p}}$ on $\Bun_{T/\tilde{C}}$ whose fibre over a $T$-bundle $E$ is the fibre $(L_{\chi,\tilde{p}})_E=(E_\chi)_{\tilde{p}}$ of the line bundle $E_\chi=E\times^{T,\chi} \bbA^1$. We can regard $L_{\chi,\tilde{p}}$ as an element of $\bD(\Bun_{T/\tilde{C}})$, and consider its image under the map $\bD(\Bun_{T/\tilde{C}})\rightarrow \bD(\cP)$. We denote this image by $\mathscr{L}_{\chi,\tilde{p}}$. The \emph{tensorization operator} associated with $(\chi,\tilde{p})$ is the morphism
\begin{align*}
\mathbf{Tens}_{\chi,\tilde{p}}: \bD(\cP) & \longrightarrow \bD(\cP) \\
\mathscr{E} & \longmapsto \mathscr{E} \otimes\mathscr{L}_{\chi,\tilde{p}}.
\end{align*} 
The \emph{abelian Wilson operator} associated with $(\chi,\tilde{p})$ is the following morphism of derived categories
\begin{align*}
\bW^{\ab}_{\chi,\tilde{p}}: D^b(\QCoh(\cP)) & \longrightarrow D^b(\QCoh(\cP)) \\
\sF & \longmapsto \sF \otimes \sL_{\chi,\tilde{p}}.
\end{align*} 

\begin{corol}
The duality equivalence $\bS^{\ab}:D^b(\QCoh(\cP))\rightarrow D^b(\QCoh(\check{\cP}))$ intertwines the abelian Hecke and Wilson operators. More precisely, for every point $\tilde{p}\in \tilde{C}$ and for every cocharacter $\lambda \in X_*(T)$, the following diagram commutes
\begin{center}
\begin{tikzcd}
D^b(\QCoh(\cP)) \ar{r}{\bS^{\ab}} \ar{d}{\bH^{\ab}_{\lambda,\tilde{p}}} &  \ar{d}{\bW^{\ab}_{\lambda,\tilde{p}}} D^b(\QCoh(\check{\cP}))\\
 D^b(\QCoh(\cP))\ar{r}{\bS^{\ab}} & D^b(\QCoh(\check{\cP})).
\end{tikzcd}
\end{center}
\end{corol}

\begin{proof}
The key is the commutativity (by construction) of the following square
\begin{center}
\begin{tikzcd}
\bD(\Bun_{T/\tilde{C}}) \ar{r}{\sim} \ar{d} & \Bun_{T^\vee/\tilde{C}} \ar{d} \\
\bD(\cP) \ar{r}{\mathrm{S}} & \check{\cP},
\end{tikzcd}
\end{center}
and the properties of the canonical isomorphism $\Bun_{T^\vee/\tilde{C}}\rightarrow\bD(\Bun_{T/\tilde{C}})$ explained in section \ref{sec:PicardStacks}.
This implies that, for any $\tilde{p}\in \tilde{C}$ and for any character $\chi\in X^*(T)$, we have $\mathrm{S}(\mathscr{L}_{\chi,\tilde{p}})=\sS_{\chi,\tilde{p}}$ and, dually, for any cocharacter $\lambda \in X_*(T)$, 
\begin{equation*}
	(\mathrm{S}^{-1})^* \sS_{\lambda,\tilde{p}} = \mathscr{L}_{\lambda,\tilde{p}}.
\end{equation*} 
Here, we are regarding the point $\sS_{\lambda,\tilde{p}}$ of $\cP$ as a line bundle on  $\bD(\cP)$ and $\mathscr{L}_{\lambda,\tilde{p}}$ as a line bundle on $\check{\cP}$. Now, we can just check, for any object $\sF$ of $D^b(\QCoh(\cP))$,
\begin{align*}
\bS^{\ab}(\bH^{\ab}_{\lambda,\tilde{p}}(\sF))=\bS^{\ab}(\sF * \sO_{\sS_{\lambda,\tilde{p}}})=(\mathrm{S}^{-1})^* (\Phi_{\cP}(\sF) \otimes \Phi_{\cP}(\sO_{\sS_{\lambda,\tilde{p}}})) = \bS^{\ab}(\sF)\otimes \bS^{\ab}(\sO_{\sS_{\lambda,\tilde{p}}}).
\end{align*}
Thus, we just need to compute $\bS^{\ab}(\sO_{\sS_{\lambda,\tilde{p}}})$. The Fourier-Mukai transform of a skyscraper sheaf is easy to obtain
\begin{equation*}
\Phi_{\cP}(\sO_{\sS_{\lambda,\tilde{p}}}) =\sS_{\lambda,\tilde{p}}.
\end{equation*} 
Therefore, 
\begin{equation*}
\bS^{\ab}(\sO_{\sS_{\lambda,\tilde{p}}}) = (\mathrm{S}^{-1})^* \sS_{\lambda,\tilde{p}} = \mathscr{L}_{\lambda,\tilde{p}},
\end{equation*} 
so
\begin{equation*}
\bS^{\ab}(\bH^{\ab}_{\lambda,\tilde{p}}(\sF)) = \bS^{\ab}(\sF)\otimes\mathscr{L}_{\lambda,\tilde{p}} = \bW^{\ab}_{\lambda,\tilde{p}}(\bS^{\ab}(\sF)),
\end{equation*} 
as we wanted to show.
\end{proof}

\subsection{Generalized Hitchin fibrations and abelianization} 
All the ``generalized Hitchin fibrations" that we study in this paper have a similar structure. We start with a semisimple group $G$, and consider an affine $G$-variety $M$. We denote by $\fc_M=M\git G$ the corresponding invariant quotient. In order to  ``twist", we also consider a torus $Z$ with an action on $M$ compatible with the $G$-action. We want to study the sequence of stacks
\begin{equation*}
[M/G\times Z] \longrightarrow [\fc_M/Z] \longrightarrow BZ
\end{equation*} 
and its relative version over a curve $C$. That is, given a  $Z$-torsor $L\rightarrow C$, we can pullback this sequence through the map  $C\rightarrow BZ$ determined by  $L$, and consider the stacks
 \begin{equation*}
\cM=\cM_{M,L}= \Map_L(C,[M/G\times Z])
\end{equation*} 
and
 \begin{equation*}
\cA=\cA_{M,L}= \Map_L(C,[\fc_M/ Z]),
\end{equation*} 
and the induced ``generalized Hitchin fibration" 
\begin{align*}
h=h_{M,L}: \cM  \longrightarrow  \cA.
\end{align*} 
Note that the points of $\cM$ are ``Higgs pairs" of the form $(E,\varphi)$, where $E\rightarrow C$ is a  $G$-bundle and $\varphi$ is a section of the associated bundle  $(E,L)\times^{G\times Z} M$.

\begin{rmk}
More precisely, in our discussions we are not only fixing the $Z$-torsor $L$ but actually what we call the ``meromorphic datum", which carries the extra information of a section of some associated vector bundle. This distinction is not relevant for the discussion we have at hand.	
\end{rmk}

In each case, we check that the fibrations we consider have ``abelian symmetries" in the following sense. There exists some commutative group scheme $J\rightarrow C \times \cA$ and an action of the Picard stack $\cP=\Bun_{J/C}\rightarrow \cA$ on $\cM\rightarrow \cA$. Moreover, there is a Zariski open subset $M^{\reg}\subset M$ for which the map $[M^{\reg}/G\times Z]\rightarrow [\fc_M/Z]$ is a union of gerbes, and a smaller Zariski open subset $M^{\gl}\subset M$ for which the map $[M/G\times Z]\rightarrow [\fc_{M^{\gl}}/Z]$ is a single gerbe. This implies that there is a Zariski open subset $\cA^{\gl}\subset \cA$ over which $\cM\rightarrow \cA$ is a $\cP$-torsor. The existence of a ``Hitchin section" $\epsilon:\cA\rightarrow \cM$ provides a trivialization of this torsor. Summing up, we have the following.

\begin{prop}
There is an isomorphism of stacks 
\begin{align*}
u: \cP/\cA^{\gl} & \longrightarrow \cM/\cA^{\gl} \\
\mathscr{E} & \longmapsto \mathscr{E} \otimes \epsilon.
\end{align*} 
\end{prop}

The group $J$ has a very convenient  ``Galois description" as follows. There is a cameral curve
\begin{equation*}
\pi:\tilde{C} \longrightarrow C \times \cA
\end{equation*} 
and a Zariski open subset $\cA^{\sharp}\subset \cA^{\gl}$ such that the fibres $\pi_a:\tilde{C}_a\rightarrow C$, for $a\in \cA$ are smooth covers of $C$ with simple Galois ramification. This cameral curve is associated with some root system $\Phi_H$, and there exists some semisimple group  $H$ with maximal torus  $A\subset H$ such that the group  $J$ is a closed and open subgroup of the invariant Weil restriction
\begin{equation*}
J^1 = \pi_*(\tilde{C}\times A)^{W_H}.
\end{equation*} 
We conclude from the above that the Picard stack $\cP/\cA^\sharp$ is a Beilinson $1$-motive like the ones considered in section \ref{sec:CameralCurvesAbelianDuality}. 

Note that there is a natural embedding $A\subset G$, but that $A$ is not necessarily a maximal torus of $G$. However, in the cases we study we can easily check that the group $W_H$ is a subset of the Weyl group $W_G$ of $G$, and that the embedding $A\hookrightarrow T$, in some maximal torus  $T$ of  $G$, is equivariant. Therefore, if we consider the dominant cocharacters and characters $X_+(A)\subset X_*(A)$ and $X^+(A)\subset X^*(A)$, respectively, we have an inclusion
\begin{equation*}
X_+(A) \hookrightarrow X_+(G)
\end{equation*} 
and a projection
\begin{equation*}
X^+(G) \twoheadrightarrow X^+(A).
\end{equation*} 

Given a dominant character $\chi \in X^+(G)$, we can consider the corresponding highest weight representation $\rho_\chi:G \rightarrow \GL(V_\chi)$. For any point $p\in C$, we can define a vector bundle  $\mathscr{V}_{\chi,p}\rightarrow \cM$ with fibres
\begin{equation*}
	(\mathscr{V}_{\chi,p})_{(E,\varphi)} = (E \times^{\rho_\chi} V_\chi)_p.
\end{equation*} 
The \emph{Wilson operator} associated with $(\chi,p)$ is the following morphism
\begin{align*}
\bW_{\chi,p}: D^b(\cM) & \longrightarrow D^b(\cM) \\
\sF & \longmapsto \sF \otimes \mathscr{V}_{\chi,p}.
\end{align*} 

Consider now a character $\chi \in X^*(A)$, whose $W_H$-orbit is represented by an element of $X^+(A)$, which in turn lifts to a dominant character of $G$ that we also denote by $\chi$. For any point $\tilde{p}\in \tilde{C}_a$ lying over $(p,a)\in C\times \cA^{\sharp}$, the isomorphism $u:\cP/\cA^{\sharp}\rightarrow \cM/\cA^{\sharp}$ identifies the sheaf $\mathscr{L}_{\chi,\tilde{p}}\rightarrow \cP$ with the sheaf $\mathscr{V}_{\chi,p}\rightarrow \cM$ and in turn the abelian Wilson operator $\bW^{\ab}_{\chi,\tilde{p}}$ with the Wilson operator $\bW_{\chi,p}$.

A \emph{Hecke correspondence} between two pairs $(E_1,\varphi_1)$ and $(E_2,\varphi_2)$ in $\cM$ is an isomorphism  
\begin{equation*}
\psi: (E_1,\varphi_1)|_{C\setminus \left\{p\right\}} \rightarrow (E_2,\varphi_2)|_{C\setminus \left\{p\right\}}
\end{equation*} 
of the restrictions of the pairs $(E_i,\varphi_i)$ to the complement of some point $p\in C$. Trivializing  $E_1$ and $E_2$ in a formal neighbourhood $\mathbb{D}_p$ of  $p$, the map $\psi$ determines an element of the formal loop group $G(\mathbb{D}^\times_p)$, well defined up to the left and right action of the formal arc group $G(\mathbb{D}_p)$ induced by changing trivialization. This determines a $G(\mathbb{D}_p)$-orbit of the affine Grassmannian $\mathrm{Gr}_G=G(\mathbb{D}^\times_p)/G(\mathbb{D}_p)$ of $G$, labeled by some dominant cocharacter $\lambda(\psi)\in X_+(G)$. 

Fix a dominant cocharacter $\lambda \in X_+(G)$ and a point $p\in C$. The \emph{$\cM$-Hecke stack} $\cH_{\lambda,p}$ is the stack parametrizing tuples $((E_1,\varphi_1),(E_2,\varphi_2),\psi)$, where the $(E_i,\varphi_i)$ are pairs in $\cM$ and $\psi: (E_1,\varphi_1)|_{C\setminus \left\{p\right\}} \rightarrow (E_2,\varphi_2)|_{C\setminus \left\{p\right\}}$ is a Hecke correspondence between them with $\lambda(\psi)=\lambda$. There are natural projections
\begin{center}
\begin{tikzcd}
& \cH_{\lambda,p} \ar{ld}{p_1} \ar{rd}{p_2} & \\
	\cM && \cM.
\end{tikzcd}
\end{center}
The \emph{Hecke operator} associated with $(\lambda,p)$ is the integral transform
\begin{align*}
\bH_{\lambda,p}: D^b(\cM) & \longrightarrow D^b(\cM) \\
\sF & \longmapsto \R p_{2*}(\mathrm{L} p_1^*(\sF) \otimes \sO_{\cH_{\lambda,p}}).
\end{align*} 

Consider then a cocharacter $\lambda \in X_+(A)$ and its image in $X_+(G)$, that we also denote by $\lambda$. For any point  $\tilde{p}\in \tilde{C}_a$ lying over $(p,a)\in C\times \cA^\sharp$, the isomorphism  $u:\cP/\cA^\sharp \rightarrow \cM/\cA^{\sharp}$ identifies the graph of the translation operator $\mathbf{Trans}_{\lambda,\tilde{p}}$ with the $\cM$-Hecke stack $\cH_{\lambda,p}$, and in turn the abelian Hecke operator $\bH_{\lambda,p}^{\ab}$ with the Hecke operator $\bH_{\lambda,p}$.

\subsection{Generic global abelianized duality}
The isomorphism $u:\cP/\cA^{\sharp}\rightarrow \cM/\cA^{\sharp}$ allows us to understand $\cM/\cA^{\sharp}$ as a Beilinson $1$-motive like the ones from section \ref{sec:CameralCurvesAbelianDuality}. The dual Beilinson $1$-motive is isomorphic to $\check{\cP}/\cA^{\sharp}$, the stack of $\check{J}$-torsors, for $\check{J}\rightarrow C \times \cA^{\sharp}$ the open and closed subscheme of
\begin{equation*}
\check{J}^1 = \pi_*(\tilde{C}\times A^\vee)^W
\end{equation*} 
such that the isomorphism $J^1/J^0\rightarrow \check{J}^1/\check{J}^0$ matches $J^1/J$ with $\check{J}/\check{J}^0$.

In each case that we study in this paper, we find a ``dual generalized Hitchin fibration" $\check{\cM}\rightarrow \cA$ with an isomorphism $\check{u}:\check{\cP}/\cA^{\sharp}\rightarrow \check{\cM}/\cA^{\sharp}$. More precisely, we find a group $\check{G}$ and a $\check{G}$-variety $\check{M}$ with a $Z$-action such that
\begin{itemize}
	\item $M\git G$ is isomorphic to  $\check{M}\git \check{G}$,
	\item the cameral covers associated with $(G,M)$ and $(\check{G},\check{M})$ can be identified, and
	\item the corresponding group $\check{J}\rightarrow C\times \cA^{\sharp}$ is related with $J$ as above.
\end{itemize}

\begin{rmk}
Note that, despite what the notation might suggest, the groups $G$ and $\check{G}$ are not Langlands dual, but they do contain Langlands dual tori $A$ and $A^\vee$, respectively. In most cases, $J$ is in fact the corresponding  $J_{\pi,H}$, associated with the semisimple group $H$. However, there are notable exceptions where $J$ is not of that form, but rather a different open and closed subgroup of $J^1$.
\end{rmk}

The results of this section can be summarized as follows.

 \begin{thm} \label{thm:GlobalAbelianizedDuality}
There exists an equivalence of derived categories
\begin{equation*}
\bS: D^b(\QCoh(\cM_{G,M}/\cA^{\sharp})) \longrightarrow D^b(\QCoh(\cM_{\check{G},\check{M}}/\cA^{\sharp}))
\end{equation*} 
which satisfies the following properties:
\begin{enumerate}
	\item \emph{(Normalization)}. The structure sheaf of the Hitchin section $\epsilon(\cA^{\sharp})$ is mapped to the structure sheaf of the total space. That is,
		\begin{equation*}
		\bS(\sO_{\epsilon(\cA^{\sharp})}) = \sO_{\cM_{\check{G},\check{M}}/\cA^{\sharp}}.
		\end{equation*} 
	\item \emph{(Hecke compatibility)}. The equivalence intertwines Hecke and Wilson operators. That is, for every point $p\in C$ and every cocharacter $\lambda \in X_+(A)$, the following square commutes
\begin{center}
\begin{tikzcd}
D^b(\QCoh(\cM_{G,M}/\cA^{\sharp})) \ar{r}{\bS} \ar{d}{\bH_{\lambda,p}} &  \ar{d}{\bW_{\lambda,p}} D^b(\QCoh(\cM_{\check{G},\check{M}}/\cA^{\sharp}))\\
 D^b(\QCoh(\cM_{G,M}/\cA^{\sharp}))\ar{r}{\bS} & D^b(\QCoh(\cM_{\check{G},\check{M}}/\cA^{\sharp})).
\end{tikzcd}
\end{center}
\end{enumerate}
\end{thm}

\subsection{Summary of dualities}
We finish this section by providing a list of the pairs $(G,M)$ and $(\check{G},\check{M})$ that we are matching in this paper, and for which there is a result of the form of the above Theorem \ref{thm:GlobalAbelianizedDuality}.
\begin{itemize}
	\item \textbf{Untwisted multiplicative Hitchin fibrations with simply-laced group}. In this case, $G$ is any semisimple group of simply-laced type and $M$ is a reductive monoid such that the semisimple part of its unit group is the simply-connected cover $G^{\Sc}$ of $G$. The group $G$ acts on $M$ through the adjoint action. The variety $\check{M}$ is equal to $M$, and $\check{G}=G^\vee$ is the Langlands dual group of $G$. The action of $G^\vee$ on $M$ is again the adjoint action (which is well defined since $G$ is simply-laced and thus $\Ad(G)=\Ad(G^\vee)$).
	\item \textbf{Twisted multiplicative Hitchin fibrations with a simply-connected simply-laced group and untwisted multiplicative Hitchin fibrations with adjoint non-simply-laced group}. In this case, $G$ is a simply-connected semisimple group of simply-laced type and without factors of type $\sfA_{2\ell}$ and $M$ is again a reductive monoid such that the semisimple part of its unit group is $G$. However, the action of $G$ is no longer the adjoint action. Instead, we consider a diagram automorphism $\theta$ of $G$, and let $G$ act on $M$ through the $\theta$-twisted conjugation action. The group $\check{G}$ is the dual group $H^\vee$ of the fixed point group $H=G^\theta$, or, equivalently, the adjoint group of the coinvariant group $G_\theta$ constructed in section \ref{sec:folding}. The variety $\check{M}$ is the monoid $M_\theta$ constructed in section \ref{sec:coinv_monoids}, with unit group $G_\theta$. The action of $\check{G}$ on $\check{M}$ is the adjoint action.
	\item \textbf{Twisted multiplicative Hitchin fibrations. Case $\sfA^{(2)}_{2\ell}$.} In this case $G=\SL_{2\ell+1}$, for some natural number $\ell$, and  $M$ is a reductive monoid such that the semisimple part of its unit group is $G$. The action of  $G$ on $M$ is twisted conjugation under the involution $\theta$ preserving  the standard pinning of $\SL_{2\ell+1}$. The variety $\check{M}$ is equal to $M$, and we also take $\check{G}=G$. However, now we consider twisted conjugation under the order $4$ automorphism $\vartheta$  ``dual" to $\theta$ introduced in section \ref{sec:inv_Aeven2}. 
\end{itemize}

\section{Untwisted multiplicative Hitchin fibrations} \label{sec:UntwistedMultHitchin}

\subsection{Construction}
Let $C$ be a smooth complex projective curve. A \emph{multiplicative Hitchin fibration} over $C$ is determined by the following data:
 \begin{itemize}
	 \item A complex semisimple group $G$ of rank $r$. We denote by $G^{\Sc}$ its simply-connected cover.
	 \item A very flat \cite[Definition 2.3.10]{Griffin_Lemma} reductive monoid $M$, with unit group $M^\times$  and with a prescribed isomorphism 
		 $$M^{\der}\overset{\sim}{\longrightarrow} G^{\Sc},$$
where  $M^\der=(M^\times)^{\der}$ denotes the semisimple part of $M^\times$.
\end{itemize}

The group $G$ acts on  $M^\times$ via the adjoint action, and in turn this action extends to an action on the monoid $M$. We can thus consider the GIT quotient $\fc_{M}=M\git \Ad(G)$, which is isomorphic to $\fc_{G^{\Sc}} \times \bbA_M$; where $\fc_{G^{\Sc}}=G^{\Sc}\git \Ad(G)$ is the GIT quotient of $G^{\Sc}$ by the adjoint action on itself and $\bbA_M:= M \git (G^{\Sc} \times G^{\Sc})$ is the \emph{abelianization} of $M$ \cite{Vinberg}. 
Since $G^{\Sc}$ is simply-connected, the GIT quotient $\fc_{G^\Sc}$ is smooth and isomorphic to the affine space $\bbA^r$. 

Consider the center $Z_{M^\times}$ of $M^\times$ and its neutral connected component  $Z=Z^0_{M^\times}$, which is a torus. The natural left multiplication action of $Z$ on $M$ commutes with the adjoint action of $G$ and induces an action of $Z$ on $\bbA_M$. In fact,  $\bbA_M$ can be understood as a commutative monoid (in this case, a toric variety) with unit group  $$\bbA_M^\times=Z_{M^\times}/(Z_{M^\times} \cap M^\der).$$

Consider now the following sequence at the level of stacky quotients
\begin{center}
\begin{tikzcd}
	\left[M/(G\times Z)\right] \rar & \left[\fc_M / Z\right] \rar & \left[\bbA_M / Z\right] \rar & BZ.
\end{tikzcd}
\end{center}
Let $D$ be a map  $C\rightarrow[\bbA_M/Z]$, and consider the mapping stacks 
\begin{align*}
\cM_{G,M}&:=	\cM_{G,M,D} := \Map_{D}(C,\left[M/(G\times Z)\right]) = \Map_C (C,C \times_{[\bbA_M/Z],D} \left[M/(G\times Z)\right]), \\
\cA_M &:= \cA_{M,D} := \Map_{D}(C,\left[\fc_M/Z\right])= \Map_C (C,C \times_{[\bbA_M/Z],D}\left[\fc_M/Z\right]).
\end{align*} 
Since we are fixing $D$, we generally drop it from the notation, although both constructions depend on it. Moreover, note that $\cA_M$ only depends on $M$, and not on $G$, since $\fc_M=M\git \Ad(M^\times)$.

\begin{defn}
The \emph{multiplicative Hitchin fibration} associated with the data $(G,M)$ and with meromorphic datum $D$ is the natural map
 \begin{align*}
h:=h_{G,M,D}: \cM_{G,M,D} & \longrightarrow \cA_{M,D}, 
\end{align*} 
induced from the map $[M/(G\times Z)]\rightarrow [\fc_M/Z]$.
\end{defn}

\begin{rmk}
Note that a map $D:C\rightarrow [\bbA_M/Z]$ is given by the following data
\begin{itemize}
	\item a $Z$-torsor  $L\rightarrow C$,	
	\item a section $s \in H^0(C,L(\bbA_M))$, where $L(\bbA_M)$ is the vector bundle associated with $L$ and the action of $Z$ on $\bbA_M$. 
\end{itemize}
\end{rmk}

\subsection{Multiplicative Higgs bundles and the Vinberg monoid} Our notion of multiplicative Hitchin fibration is the one studied in the paper of G. Wang \cite{Griffin_Lemma}, considered also previously in the works of Bouthier and J. Chi \cite{Bouthier_Dimension,Bouthier_Hitchin,JChi}. In other works, like Elliott--Pestun \cite{Elliott-Pestun} and Hurtubise--Markman \cite{Hurtubise-Markman}, the multiplicative Hitchin fibration is introduced in terms of ``meromorphic" \emph{multiplicative Higgs bundles} (that is, Higgs pairs which are group-valued, instead of taking values on a vector space, but with rational singularities). 

More precisely, by a multiplicative Higgs bundle we mean a pair $(E,\varphi)$ formed by a $G$-bundle $E\rightarrow C$ and a rational section  $\varphi$ of the adjoint group bundle  $E(G)= E\times_{\Ad} G$. We can also take $\varphi$ to be a section of the group bundle $E(G^{\Sc})=E\times_{\Ad} G^{\Sc}$. This section is rational, meaning that it has singularities at some points $p_1,\dots,p_n \in C$. As a direct consequence of Iwahori's theorem, the singularity of the section $\varphi$ at each one of the $p_i$ is determined by a dominant cocharacter $\lambda_i$ of $G$ (provided a choice of Borel pair $(B,T)$ for $G$).

The choice of $(B,T)$ determines a set of simple roots $\Delta=\left\{\alpha_1,\dots \alpha_r\right\}$ for  $G$, and we let $\bbA_G$ denote the affine space  $\bbA^r$ with the natural $T$-action
\begin{equation*}
t\cdot \bm{x}= t \cdot (x_1,\dots,x_r) = (\alpha_1(t)\cdot x_1,\dots,\alpha_r(t) \cdot x_r) =: \bm{\alpha}(t) \cdot  \bm{x}.
\end{equation*} 
This action factors through an action of $T^{\ad}=T/Z_G$. Let $\sO_{C}(p)$ be the bundle defined by the divisor of a point $p\in C$, with its canonical section $z_{p}$. We can regard this as the map $C\rightarrow [\bbA^1/\Gm]$ sending $p$ to the closed point $0\in [\bbA^1/\Gm]$. A coweight $\lambda$ of  $G$ determines a map  $\lambda:\Gm \rightarrow T^{\ad}$, and we have, for $z_p \neq 0$
\begin{equation*}
\lambda(z_p) \cdot (1,\dots,1) = (z_p^{\langle \alpha_1,\lambda\rangle},\dots,z_p^{\langle \alpha_r,\lambda\rangle} ),
\end{equation*} 
which extends to $z_p=0$ if and only if  $\lambda$ is dominant. In other words, $\bbA_G$ is the affine toric embedding of  $T^{\ad}$ determined by the cone generated by the dominant coweights,
\begin{equation*}
\mathbb{R}_{\geq 0} X_+(T^{\ad}) \subset X_*(T^{\ad}) \otimes_{\mathbb{Z}} \mathbb{R}.
\end{equation*} 

The \emph{Vinberg monoid} or \emph{enveloping monoid} $\Env(G^{\Sc})$ of $G^{\Sc}$ \cite{Vinberg} is a reductive monoid with unit group equal to
\begin{equation*}
G^{\Sc}_+ := (G^{\Sc}\times T^{\Sc})/Z_{G^{\Sc}}
\end{equation*} 
and whose abelianization coincides with $\bbA_G$. Since the cocharacters $\lambda_i$ are dominant, the natural multiplicative map  
\begin{align*}
\bm{\lambda}=(\lambda_1,\dots,\lambda_n): \Gm^n & \longrightarrow T^{\ad} \\
(z_1,\dots,z_n) & \longmapsto \lambda_1(z_1)\cdots \lambda_n(z_n)
\end{align*} 
extends to a map $\bm{\lambda}:\bbA^n \rightarrow \bbA_G$. We can then consider the fibered product
\begin{equation*}
M_{\bm{\lambda}} = \Env(G^{\Sc}) \times_{\bbA_G, \lambda} \bbA^n,
\end{equation*} 
which is a reductive monoid with abelianization $\bbA^n$ and such that $M^\der = G^{\Sc}$. The multiplicative Hitchin fibration associated with the data $(G,M_{\bm{\lambda}})$ is precisely the fibration considered in the papers \cite{Hurtubise-Markman, Elliott-Pestun}, as explained in \cite{Bouthier_Hitchin,Griffin_Lemma}. 

In fact, all multiplicative Hitchin fibrations arise in a similar way. Indeed, Vinberg's classification \cite{Vinberg} implies that the Vinberg monoid is \emph{versal} among very flat reductive monoids, meaning that every very flat reductive monoid $M$ with  $M^\der=G^{\Sc}$ can be obtained as a fibered product of its abelianization with the Vinberg monoid $\Env(G^{\Sc})$.

\begin{rmk}
If $M$ is a very flat reductive monoid with  $M^\der=G^{\Sc}$ and we denote by $Z$ the connected center of $M^\times$, then what Vinberg classification tells us is that there is a natural map $\phi_{Z}:Z\rightarrow T^{\ad}$ extending to a map $\phi_M: \bbA_M \rightarrow \bbA_G$ of the abelianizations and that $M$ is obtained as the pullback of $\Env(G^{\Sc})$ through this map.
\end{rmk}

\begin{rmk}
The abelianization map $\Env(G^{\Sc})\rightarrow \bbA_G$ of the Vinberg monoid admits a natural section $\delta_{\Env(G^{\Sc})}: \bbA_G \rightarrow \Env(G^{\Sc})$ extending the map $T^{\ad} \rightarrow T^{\Sc}_+$, $t \mapsto (t,t^{-1})$, where $T^{\Sc}_+ = (T^{\Sc} \times T^{\Sc})/Z_{G^{\Sc}}$ is the maximal torus of $G^{\Sc}_+$. More generally, if $M$ is a very flat reductive monoid with $M^{\der}=G^{\Sc}$ and we consider the maximal torus $T_M=(Z_M \times T^{\Sc})/Z_{G^\Sc}$ of $M^\times$ determined by $T^{\Sc}$, then we can define the map  $Z  \rightarrow T_M$,  $z\mapsto (z,\phi_Z(z)^{-1})$, and this map extends to a section $\delta_M: \bbA_M \rightarrow M$ of the abelianization map  $M\rightarrow \bbA_M$.
\end{rmk}

\subsection{Steinberg quasi-sections} \label{sec:Steinberg}
As above, we fix a Borel pair $(B,T)$ for $G$, which determines a set of simple roots $\Delta$. Furthermore, we also fix a pinning of $G$. This is determined by the extra choice of $1$-parameter groups
$u_\alpha: \Ga  \rightarrow  U_\alpha$
parametrizing the root groups $U_\alpha$ associated with each root  $\alpha \in \Delta$ in such a way that, for every $t\in T$ and for every $x\in \Ga$,
 \begin{equation*}
\Ad_t(u_\alpha(x))=u_\alpha(\alpha(t) x).
\end{equation*} 

A \emph{Coxeter datum} $(\sigma,\dot{\bm{s}})$ \cite[Definition 2.2.7]{Griffin_Lemma} consists of the following
\begin{itemize}
	\item a bijection $\sigma:\left\{1,\dots,r\right\}\rightarrow \Delta$ (i.e. a total ordering of the simple roots)	
	\item a choice $\dot{\bm{s}}=\left\{\dot{s}_\alpha: \alpha \in \Delta\right\}$ of a representative $\dot{s}_\alpha \in N_G(T)$ of each simple reflection $s_\alpha \in W$ associated with a simple root $\alpha \in \Delta$.
\end{itemize}
Once $\Delta$ is fixed, we say that an element $w\in W$ of the Weyl group is a ($\Delta$-)\emph{Coxeter element} if it can be written as $w=w_\sigma=s_{\sigma(1)}\dots s_{\sigma(r)}$, for some total ordering $\sigma$ of $\Delta$. We denote by $\mathrm{Cox}(W,\Delta)$ the set of Coxeter elements of $W$.

If we fix a Coxeter datum $(\sigma,\dot{\bm{s}})$, we can consider the associated \emph{Steinberg quasi-section}
\begin{align*}
\epsilon^{(\sigma,\dot{\bm{s}})}: \fc_G \cong \bbA^r & \longrightarrow G^{\Sc} \\
(x_1,\dots,x_r) & \longmapsto \prod_{i=1}^r u_{\sigma(i)}(x_i) \dot{s}_{\sigma(i)}.
\end{align*} 
This is a quasi-section in the sense that, composed with the projection $G^{\Sc}\rightarrow \fc_G$ it is not necessarily equal to the identity, but rather to an automorphism of $\fc_G$. The image  $\epsilon^{(\sigma,\dot{\bm{s}})}(\fc_G)$ is called the \emph{Steinberg cross-section} associated with $(\sigma,\dot{\bm{s}})$. 

The dependence on the Coxeter datum is not too strong: if $(\sigma, \dot{\bm{s}})$ and $(\sigma', \dot{\bm{s}}')$ are such that $\sigma=\sigma'$, then there exists some  $t\in T$ such that the corresponding cross-sections are conjugate under $t$; on the other hand, if $\dot{\bm{s}}=\dot{\bm{s}}'$, then, for any $x, x' \in \bbA^r$ with $x_{\sigma(i)}=x_{\sigma'(i)}$ for every $i=1,\dots,r$, the elements  $\epsilon^{\sigma,\dot{\bm{s}}}(x)$ and $\epsilon^{\sigma',\dot{\bm{s}}}(x')$ are conjugate in $G^{\Sc}$. 

There is a natural way \cite[Proposition 2.4.4]{Griffin_Lemma} of extending a Steinberg quasi-section $\epsilon^{(\sigma,\dot{\bm{s}})}$ to a reductive monoid $M$ with derived group equal to $G^{\Sc}$, by putting
\begin{align*}
\epsilon_M^{(\sigma,\dot{\bm{s}})}: \fc_M \cong \bbA_M \times \fc_G & \longrightarrow M \\
(a,x) & \longmapsto \delta_M(a) \epsilon^{(\sigma,\dot{\bm{s}})}(x).
\end{align*} 
This provides a quasi-section of the GIT quotient $M\rightarrow \fc_M$. 
However, in this case the dependence of the section on the Coxeter datum is strong: indeed, for two different Coxeter data $(\sigma,\dot{\bm{s}})$ and $(\sigma',\dot{\bm{s}})$, the elements $\epsilon^{(\sigma,\dot{\bm{s}})}(0)$ and $\epsilon^{(\sigma',\dot{\bm{s}})}(0)$ are not conjugate (see \cite[Remark 2.2.18]{JChi}).

\subsection{Symmetries} An element $x\in M$ is \emph{regular} (under the $G$-action) if its stabilizer
\begin{equation*}
I_{G,M,x}:=\left\{g\in G: \Ad_g(x)=x\right\}
\end{equation*} 
has the minimal possible dimension for all elements of $M$. We denote by $M^{\reg}\subset M$ the subset of regular elements. We denote by $\cM_{G,M}^{\reg}$ the subset of sections in $\cM_{G,M}$ factorizing through $[M^{\reg}/(G\times Z)]$. Given any $a\in \cA_M(\C)$, we denote by $\cM^{\reg}_{G,M,a}$, the fibre of $a$ through the restriction $h|_{\cM_{G,M}^{\reg}}$.

The stabilizers define a group scheme $I_{G,M}\rightarrow M$ called the \emph{centralizer group scheme}, which restricts to a smooth commutative group scheme over $M^{\reg}$. Under some adaptation of the typical descent argument of Ngô, J. Chi \cite[Lemma 2.4.2]{JChi} proved the following (for the Vinberg monoid, but it generalizes immediately to any very flat monoid).

\begin{prop}[J. Chi]
There is a unique commutative group scheme $J_{G,M}\rightarrow \fc_M$ with a $G$-equivariant isomorphism
\begin{equation*}
\chi^*_M J_{G,M}|_{M^{\reg}} \overset{\sim}{\longrightarrow} I_{G,M}|_{M^{\reg}},
\end{equation*} 
which can be extended to a homomorphism $\chi_M^* J_{G,M}\rightarrow I_{G,M}$. Here, $\chi_M:M\rightarrow \fc_M$ denotes the natural projection to the GIT quotient. We call this $J_{G,M}$ the \emph{regular centralizer group scheme}.
\end{prop}

For any Coxeter datum $(\sigma,\dot{\bm{s}})$ and any element $a\in \fc_M$, the element  $\epsilon^{(\sigma,\dot{\bm{s}})}(a)$ is regular. Moreover, it is a result of J. Chi \cite[Section 2.4.1]{JChi}, that every regular element of $M$ is conjugate to an element of the form $\epsilon^{(\sigma,\dot{\bm{s}})}(a)$ for some $a\in \fc_M$ and some Coxeter datum  $(\sigma,\dot{\bm{s}})$.
If $a\in \fc_{M^\times}:=M^\times \git \Ad(M^\times)$ comes from the unit group, then the preimage of $a$ through the restriction $M^{\times,\reg}\rightarrow \fc_{M^\times}$ consists of a single orbit, so it is a homogeneous space, and the map $[M^{\times,\reg}/G]\rightarrow \fc_{M^\times}$ is a $J_{G,M^\times}$-gerbe neutralized by any Steinberg quasi-section $\epsilon^{(\sigma,\dot{\bm{s}})}$.
However, in general the fibres of the GIT quotient $M\rightarrow \fc_M$ contain several regular orbits, so the fibres of $M^{\reg}\rightarrow \fc_M$ are no longer homogeneous spaces. The number of irreducible components of the fibres is bounded above by the cardinal of $\mathrm{Cox}(W,\Delta)$.
An immediate consequence of the above proposition is that the natural map $[M^{\reg}/G]\rightarrow \fc_M$ is a finite union of $J_{G,M}$-gerbes, and each of these gerbes is neutralized by some extended Steinberg quasi-section $\epsilon^{(\sigma,\dot{\bm{s}})}_M$.
The moral of this is that understanding the symmetries of the multiplicative Hitchin fibres will amount to understanding the structure of the regular centralizer $J_{G,M}$. 

\subsection{Galois description} We choose now a maximal torus $T_M\subset M^\times$, which in turn determines a choice of maximal torus  $T\subset G$, and denote by  $\bar{T}_M$ the closure of $T_M \subset M$ inside the monoid  $M$. We let $\Phi:=\Phi(G,T)$ denote the corresponding root system (of both $G$ and  $M^\times$) and $W$ its Weyl group.

The Chevalley isomorphism $\fc_{M^\times}\cong T_M/W$ extends to an isomorphism $\fc_{M}\cong \bar{T}_M/W$ \cite[Theorem 2.4.2]{Griffin_Lemma}, and thus allows us to consider the \emph{cameral cover}
\begin{equation*}
\pi_{M}: \bar{T}_M \longrightarrow \bar{T}_M/W \cong \fc_M.
\end{equation*} 
We introduce the Weil restriction
$\Pi:=\Pi_{G,M}= \pi_{M,*} (\bar{T}_M \times T)$,
on which $W$ acts diagonally, and also the invariant subgroup
$J^1:= J^1_{G,M} = \Pi^W$
and its neutral connected component
$J^0:= J^0_{G,M} = \Pi^{W,0}$.
Consider now the subgroup scheme $J':=J'_{G,M}\hookrightarrow J^1$ defined as follows. For any $\fc_M$-scheme $S\rightarrow \fc_M$, the elements of  $J'(S)$ are points $f\in \Pi(S)=\Hom(S\times_{\fc_M} \bar{T}_M,T)$ which belong to $J^1(S)$ and such that, for every $x\in (S\times_{\fc_M} \bar{T}_M)(\C)$, if $s_\alpha(x)=x$ for a root $\alpha \in \Phi$, then $\alpha(f(x))=1$.
The Grothendieck--Springer simultaneous resolution of singularities implies the following result, which can be found in \cite[Theorem 2.4.12]{Griffin_Lemma}. 

\begin{prop}[G. Wang]
There is a canonical open embedding $J_{G,M}\hookrightarrow J^1_{G,M}$ identifying $J_{G,M}$ with $J'_{G,M}$.	
\end{prop}

The proof of the above proposition follows the same argument as the proof of Proposition 2.4.7 in \cite{Ngo_Lemme}. We also refer the reader to the sketch of the proof of Proposition \ref{prop:galoisJ} in this paper, where a very similar argument is explained.

\subsection{The group-like locus} Inside $\fc_M$ there are two special divisors that one must take into account. The first is the \emph{boundary divisor} $\fB_M$, defined as the complement of
 \begin{equation*}
\fc_{M}^\times= \fc_{G^{\Sc}} \times \bbA^\times_M.
\end{equation*} 
Note that the elements of $M$ lying over  $\fc_{M}^\times$ are those in the unit group $M^\times$. There is also the \emph{extended discriminant divisor} $\fD_M$, defined for $M=\Env(G^\Sc)$ and extended by pullback to every $M$. The divisor $\fD_{\Env(G^{\Sc})}\subset \fc_{\Env(G^{\Sc})}$ is defined as the vanishing set of the following $W$-invariant function on $\bar T_{\Env(G^{\Sc})}$,
 \begin{equation*}
\mathrm{Disc}_{\Env(G^{\Sc})}:=\prod_{\beta \in \Phi_+}(\beta,0) \prod_{\alpha \in \Phi} (1-(0,\alpha))= 
(2\rho,0)\prod_{\alpha \in \Phi} (1-(0,\alpha)), 
\end{equation*} 
where $\rho$ denotes the half-sum of the positive roots. It is a result of G. Wang \cite[Lemma 2.4.7]{Griffin_Lemma} that these two divisors intersect properly, and thus that the codimension of the intersection $\fD_M^\times:= \fD_M \cap \fB_M$ is at least $2$. We say that an element of $M$ is \emph{group-like} if it belongs to the preimage of $\fc_M^{\gl}=\fc_M\setminus \fD_M^\times$, and denote by $M^{\gl}=\chi_M^{-1}(\fc_M^{\gl})$ the locus of group-like elements. 

\begin{rmk}
In Hurtubise and Markman's paper \cite[Definition 6.6]{Hurtubise-Markman}, they introduce a certain notion of regularity in the context of the multiplicative Hitchin fibration which is related to our notion of being group-like. Our notion of regularity is the one introduced by G. Wang \cite[Section 2.4.5]{Griffin_Lemma}, which is completely parallel to the group case and compatible with the usual notion of regularity of a group action.
\end{rmk}

We say that an element  $x\in M$ is \emph{semisimple} if its adjoint orbit contains an element of $\bar T_M$. An element  $x\in M$ is \emph {regular semisimple} if it is both regular and semisimple. The same notions apply to adjoint orbits, and we consider the subsets $M^{\reg}$,  $M^{\Ss}$ and  $M^{\rs}$ of regular, semisimple and regular semisimple elements, respectively.
It turns out that $\fc_M^{\rs}$, the image of $M^{\rs}$ under the GIT quotient, is the complement of the discriminant divisor $\fD_{M}$. Moreover, each fiber of $\chi_M$ over  $\fc_M^{\rs}$ consists of a single orbit. Therefore, an element of $M$ is group-like if it is either in the unit group or regular semisimple, that is
 \begin{equation*}
M^{\gl}= M^\times \cup M^{\rs}.
\end{equation*} 

Over the group like locus $\fc_M^{\gl}$ we can obtain a more detailed description of the regular centralizer $J$. Indeed, if  $a\in \fc_M^{\gl}$, then either $a\in \fc_{M^\times}$ or $a \in \fc_M^{\rs}$. Now, if $a \in \fc_M^{\rs}$ then the centralizer of any element in its fibre is isomorphic to the maximal torus $T$; in other words, over  $\fc_M^{\rs}$ there is an isomorphism $J|_{\fc_M^{\rs}} \cong J^1|_{\fc_M^{\rs}}$. Therefore, the only possible difference is over the points of the divisor $\fD_{M^\times}=\fc_{M^\times}\setminus \fc_M^{\rs}$. This is the discriminant divisor of the group $M^\times$, determined  by the following $W$-equivariant function on $T_M$
\begin{equation*}
\mathrm{Disc}_{M^\times} := \prod_{\alpha \in \Phi} (1-\alpha).
\end{equation*} 
Let $\fD_{M^\times}^{\mathrm{sing}}\subset \fD_{M^\times}$ denote the singular locus. The restriction of $\bar{T}_M \cap M^{\gl} \rightarrow \fc_M^{\gl}$ over the complement of $\fD_{M^\times}^{\mathrm{sing}}$ is a cameral cover of smooth varieties associated with the root system $\Phi=\Phi(G,T)$, and with simple Galois ramification. Therefore, we can identify the groups $J^0_{G,M}$, $J_{G,M}$ and $J^{1}_{G,M}$ with the groups $J^0$, $J$ and $J^1$ introduced in section \ref{sec:cameral_simple_Galois}.

\subsection{Hitchin fibres and global Steinberg section} One of the advantages of working over the group-like locus is that the multiplicative Hitchin fibres become simpler. First, note that, for every $a\in \fc_M^{\gl}$, its fibre in $M$ contains a single regular orbit. Indeed, if $a\in \fc_{M^\times}$, then this is explained above in section \ref{sec:Steinberg}, while, if $a\in \fc_M^{\rs}$, this follows from the work of J. Chi \cite[Lemma 2.2.13]{JChi}. Therefore, the map $[M^{\gl}/G]\rightarrow \fc_M^{\gl}$ is in fact a $J_{G,M}$-gerbe, neutralized by any Steinberg quasi-section $\epsilon^{(\sigma,\dot{\bm{s}})}$, for any choice of Coxeter datum $(\sigma,\dot{\bm{s}})$.

Let $\cA_M^{\gl}\subset \cA_M$ denote the Zariski open locus whose points $a\in \cA_M(\C)$ map the curve $C$ entirely inside $\fc_M^{\gl}$.
The above implies that, for every $a \in \cA^{\gl}_M(\C)$, the fibre $\cM^{\reg}_{G,M,a}$ is a torsor under the Picard stack
\begin{equation*}
\cP_{G,M,a} := \Bun_{J_{G,M,a}/C},
\end{equation*} 
for $J_{G,M,a} = a^* J_{G,M}$. Globally, we can consider the Picard stack $\cP_{G,M}\rightarrow \cA^{\gl}_M$ and say that the regular part of the multiplicative Hitchin fibration over the group-like locus $h:\cM_{G,M}^{\reg}\rightarrow \cA_M^{\gl}$ is a $\cP_{G,M}$-torsor. Moreover, this torsor can be trivialized through a \emph{global Steinberg section}. 

We recall now the construction of the global Steinberg section from J. Chi \cite[Section 4.1.5]{JChi} and G. Wang \cite[Section 2.4.22]{Griffin_Lemma}. First, we need the following result \cite[Proposition 2.4.21]{Griffin_Lemma}.

\begin{prop}[G. Wang] \label{prop:preglobalsection}
For each Coxeter datum $(\sigma, \dot{\bm{s}})$, one can define an action $\tau_M^{(\sigma, \dot{\bm{s}})}$ of $Z$ on $M$ such that
\begin{equation*}
\epsilon_M^{(\sigma,\dot{\bm{s}})} \circ \tau_{\fc_M}(z^c) = \tau_M^{(\sigma, \dot{\bm{s}})}(z) \circ \epsilon_M^{(\sigma, \dot{\bm{s}})},
\end{equation*} 
where $\tau_{\fc_M}$ is the natural action of $Z$ on $\fc_M$, and  $c=|Z_{G^\Sc}|$. Moreover, for a fixed $z$, the map $\tau_M^{(\sigma, \dot{\bm{s}})}(z)$ is a composition of translation by $z^c$ and conjugation by an element of $T^{\Sc}$ determined by a homomorphism $Z\rightarrow T^{\Sc}$ independent of $z$.
\end{prop}

Consider the map $Z \rightarrow T^{\Sc}$ from the above proposition, and denote by $\psi:Z \rightarrow T$ its composition with the natural projection $T^{\Sc}\rightarrow T$. Let
\begin{align*}
\Psi:Z & \longrightarrow T \\
z & \longmapsto \psi(z)z_T^{-c},
\end{align*} 
where $z_T$ is the projection of $z$ from $Z$ to $T$. We define the stack $[M/(G\times Z)]_{[c]}$ as the pullback of $[M/(G\times Z)]$ through the map $BZ\rightarrow BZ$ induced by  $Z\rightarrow Z$,  $z\mapsto z^c$. We define  $[\fc_M/Z]_{[c]}$ similarly. The above proposition gives the existence of a quasi-section
\begin{align*}
[\epsilon_M^{(\sigma,\dot{\bm{s}})}]_{[c]}:[\fc_M/Z]_{[c]} \longrightarrow [M/(\Psi\times \id)(Z)]_{[c]}, 
\end{align*} 
which, by composition, induces a quasi-section
\begin{align*}
[\epsilon_M^{(\sigma,\dot{\bm{s}})}]_{[c]}:[\fc_M/Z]_{[c]} \longrightarrow [M/G\times Z]_{[c]}. 
\end{align*} 

Consider then a $Z$-torsor $L$. If we denote
\begin{equation*}
\cA_{M,L} = \Map_L(C,[\fc_M/Z])=\Map_C(C, C \times_{BZ,L}[\fc_M/Z]),
\end{equation*} 
then, by definition, there is a canonical map $[\mathrm{ev}]_L:\cA_{M,L} \times C \rightarrow [\fc_M/Z]$ lying over the map $C\rightarrow BZ$ determined by $L$. Suppose that there exists another $Z$-torsor $L'$ such that $(L')^{\otimes c}=L$. This torsor determines a morphism $[\mathrm{ev}]_{L'}:\cA_{M,L} \times C \rightarrow [\fc_M/Z]_{[c]}$ lifting $[\mathrm{ev}]_L$. Composing with the quasi-section $[\epsilon_M^{(\sigma,\dot{\bm{s}})}]_{[c]}$, we obtain a section
\begin{equation*}
\epsilon_{L'}^{(\sigma,\dot{\bm{s}})}: \cA_{M,L} \rightarrow \cM_{G,M,L} = \Map_L(C,[M/G\times Z])
\end{equation*} 
of the natural map $\cM_{G,M,L}\rightarrow \cA_{M,L}$. Pulling back by a meromorphic datum $D:C\rightarrow [\bbA_M/Z]$ with associated $Z$-torsor $L$, we obtain a section
\begin{equation*}
\epsilon_{L'}^{(\sigma,\dot{\bm{s}})}: \cA_{M,D} \rightarrow \cM_{G,M,D}
\end{equation*} 
of the multiplicative Hitchin map $h_D:\cM_{G,M,D}\rightarrow \cA_{M,D}$. This is called the \emph{global Steinberg section} associated with the torsor $L'$ and with the Coxeter datum $(\sigma,\dot{\bm{s}})$.

We conclude the following.

\begin{thm}
The stack $\cM^{\reg}_{G,M}$ is isomorphic to the Picard stack $\cP_{G,M}$ over the Zariski open locus $\cA_M^{\gl}\subset \cA_M$.	
\end{thm}

\subsection{Cameral curve} \label{sec:cameralcurve}
We want to describe the Picard stack $\cP_{G,M,a}=\Bun_{J_{G,M,a}/C}$, for $a \in \cA^{\gl}_M(\C)$. We do so in terms of the cameral curve $\pi_a:\tilde{C}_a \rightarrow C$, obtained as a pullback of the cameral cover $\pi_M:\bar{T}_M \rightarrow \fc_M$ through the map $a:C\rightarrow [\fc_M/Z]$. 

Consider the Zariski open locus $\cA_M^\sharp \subset \cA_M$ \cite[Section 6.3.12]{Griffin_Lemma} whose points $a\in \cA_M(\C)$ satisfy the following two conditions:
\begin{enumerate}
	\item the image $a(C)$ of $C$ is contained entirely in  $\fc_M^{\gl}$ (that is, $a \in \cA_M^{\gl}$), and	
	\item the image $a(C)$ intersects transversely the discriminant divisor $\fD_M$.
\end{enumerate}
The first condition implies that the cameral cover is indeed a cameral cover associated with the root system $\Phi(G,T)$ and with simple Galois ramification. The second condition implies that the cameral curve $\tilde{C}_a$ is smooth. 

A direct consequence of this is that, for $a\in \cA_M^\sharp(\C)$, there is a natural map of Picard stacks $\cP_{G,M,a} \rightarrow \Bun_{T/\tilde{C}_a}$, which provides a description of $\cP_{G,M,a}$ in terms of  $W$-equivariant $T$-bundles over the cameral curve $\tilde{C}_a$. In particular, this implies that $\cP_{G,M,a}$ is a Beilinson $1$-motive.	

We can make sure that $\cA_M^\sharp(\C)$ is non-empty under some conditions on the meromorphic datum  $D$. 
Let $L$ be the  $Z$-torsor associated with the meromorphic datum $D$. Since the monoid $M$ can be obtained as a pullback from the Vinberg monoid, there is a natural map $Z\rightarrow T^{\Sc}$ and, for every dominant weight $\omega$ of $G$, we obtain a character $\omega_Z \in X^*(Z)$. This character determines a line bundle $\omega(L):=L\times^{Z,\omega_Z} \Gm$ over $C$. We say that  $L$ is \emph {very $G$-ample} if
\begin{equation*}
\deg \omega(L) > 2g,
\end{equation*} 
for every dominant weight $\omega$ of $G$. We have the following result \cite[Proposition 6.3.13]{Griffin_Lemma}

\begin{prop}[G. Wang]
If $L$ is very  $G$-ample, then $\cA_M^\sharp(\C)$ is non-empty.
\end{prop}

\subsection{Duality in the simply-laced case} Suppose now that the group $G$ is simply-laced, meaning that all of its roots are of the same length, or, equivalently, that each of the connected components of its Dynkin diagram are of types $\sfA$,  $\sfD$ or  $\sfE$. Fix a maximal torus $T\subset G$ and consider the Langlands dual group $G^\vee$, with maximal torus $T^\vee$.

Since $G$ is assumed to be simply-laced, the map $\alpha \mapsto \alpha^\vee$ sending roots to coroots determines an isomorphism of the root system of $G$ with the root system of $G^\vee$. In turn, this implies that there are natural isomorphisms $G^{\Sc}\cong (G^\vee)^{\Sc}$ and $G^\ad \cong (G^\vee)^\ad$. Therefore, we can also consider the adjoint action of $G^\vee$ on the reductive monoid $M$, and there are natural isomorphisms
\begin{equation*}
\fc_M = M \git \Ad(G) \cong M \git \Ad(G^\vee).
\end{equation*} 
Thus, we can identify $\cA_M$ with  $\cA_{G^\vee,M}$. However, in general the space $\cM_{G^\vee,M}$ is different from $\cM_{G,M}$, although it also fibres over  $\cA_M$. We consider, for every $a\in \cA_M(\C)$, the corresponding Picard stack $\cP_{G^\vee,M,a}$ acting on the (regular) fibre of $a$.

The cameral curves also depend uniquely on the monoid $M$, and not on the group $G$. Therefore, the Picard stack $\cP_{G^\vee,M,a}$ for $a\in \cA_M(\C)$ can also be described in terms of the same cameral curve $\tilde{C}_a$. More precisely, as in the previous section, when $a\in \cA_M^\sharp(\C)$, there is a natural map of Picard stacks $\cP_{G^\vee,M,a} \rightarrow \Bun_{T^\vee/\tilde{C}_a}$, which provides a description of $\cP_{G^\vee,M,a}$ in terms of $W$-equivariant $T^\vee$-bundles over $\tilde{C}_a$.

Recall that there is a natural isomorphism of Picard stacks $\bD(\Bun_{T/\tilde{C}_a}) \rightarrow \Bun_{T^\vee/\tilde{C}_a}$. If $a\in \cA_M^\sharp(\C)$, the natural map $\cP_{G,M,a} \rightarrow \Bun_{T/\tilde{C}_a}$ induces a map $\bD(\Bun_{T/\tilde{C}_a})\rightarrow \bD(\cP_{G,M,a})$. On the other hand, the norm map induces a map $\Nm:\Bun_{T^\vee/\tilde{C}_a}\rightarrow \cP_{G^\vee,M,a}$. Putting this together, we obtain the duality morphism $\mathrm{S}_a:\bD(\cP_{G,M,a}) \rightarrow \cP_{G^\vee,M,a}$, as in the following square
\begin{center}
\begin{tikzcd}
\bD(\Bun_{T/\tilde{C}_a}) \ar{r} \ar{d} & \Bun_{T^\vee/\tilde{C}_a} \ar{d}{\Nm} \\
\bD(\cP_{G,M,a}) \ar{r}{\mathrm{S}_a} & \cP_{G^\vee,M,a}.
\end{tikzcd}
\end{center}
As an immediate application of Theorem \ref{thm:duality} we obtain the following. This is Theorem \ref{thm:duality_simplylaced} in the introduction. 

\begin{thm} \label{thm:duality_simplylaced_text}
For every $a\in \cA_M^\sharp(\C)$, the map $\mathbf{S}_a:\bD(\cP_{G,M,a})\rightarrow \cP_{G^\vee,M,a}$ is an isomorphism of Beilinson $1$-motives. In particular, the Fourier-Mukai transform  $\Phi_{\cP_{G,M,a}}$ induces an isomorphism of derived categories
\begin{equation*}
\mathbf{S}_a: D^b(\QCoh(\cP_{G,M,a})) \longrightarrow D^b(\QCoh(\cP_{G^\vee,M,a})).
\end{equation*} 
\end{thm}

\subsection{An example in low rank: \texorpdfstring{$\SL_2$ and $\PGL_2$}{SL2 and PGL2}} 

We consider now the untwisted multiplicative Hitchin fibration associated with $\SL_2$ and with the Vinberg monoid $M=\Env(\SL_2)=\Mat_2$, which for $\SL_2$ is just the monoid $\Mat_2$ of  $2\times 2$ matrices, with unit group  $\Mat_2^\times=\GL_2$. The center $Z_M\cong \Gm$ is just the center of  $\GL_2$. 

The GIT quotient  $\SL_2\git\SL_2$ is isomorphic to  $\bbA^1$, the natural projection being the trace map $\tr: \SL_2 \rightarrow \bbA^1$. On the other hand, the abelianization $\bbA_M=\GL_2 \git (\SL_2\times \SL_2)$ is as well isomorphic to $\bbA^1$, with the natural projection being the determinant map $\det:\GL_2 \rightarrow \bbA^1$. The  $\Gm$ action induced on  $\bbA^1$ from the central $\Gm$ action on $\GL_2$ is the squaring action  $z \cdot a = z^2 a$. We will denote $\bbA^1_{(2)}$ to indicate that $\bbA^1$ is endowed with the squaring $\Gm$ action. A map $D: C \rightarrow [\bbA^1_{(2)}/\Gm]$ is determined by the following data  
\begin{itemize}
	\item a line bundle $L\rightarrow C$, 
	 \item a section $b \in H^0(C,L^2)$.
\end{itemize}

Therefore, we can identify the stack
\begin{equation*}
\cM_{\SL_2,\Mat_2,(L,b)} = \Map_{(L,b)}(C,[\Mat_2/\SL_2 \times \Gm])
\end{equation*} 
with the classifying stack of pairs $(E,\varphi)$ with
\begin{itemize}
	\item $E\rightarrow C$ a vector bundle of rank  $2$ and trivial determinant  $\det E \overset{\sim}{\rightarrow} \sO_C$,	
	\item a twisted endomorphism $\varphi:E\rightarrow E\otimes L$ with  $\det \varphi = b$.
\end{itemize}

The multiplicative Hitchin base can be identified as
 \begin{equation*}
\cA_{\Mat_2,(L,b)}= \Map_{(L,b)}(C,[\bbA^1_{(1)}\times \bbA^1_{(2)}/\Gm]) = H^0(C,L) \times \left\{b \in H^0(C,L^2)\right\}.
\end{equation*} 
The multiplicative Hitchin fibration in this case is then the map
\begin{align*}
h=h_{\SL_2,\Mat_2,(L,b)}: \cM_{\SL_2,\Mat_2,(L,b)} & \longrightarrow \cA_{\Mat_2,(L,b)} \\
(E,\varphi) & \longmapsto (\tr(\varphi),\det(\varphi)=b).
\end{align*} 

Note that in this case we are simply studying a subvariety of the usual Hitchin moduli space for $L$-twisted $\GL_2$-Higgs bundles. More precisely, the multiplicative Hitchin base naturally embeds inside the usual Hitchin base $H^0(C,L)\times H^0(C,L^2)$, and we are imposing the extra restriction that $\det E$ is trivial. The usual Hitchin fibration restricts to the multiplicative Hitchin fibration in this situation. Therefore, the description of the fibres of $h$ can be reduced to the well known description of the fibres of the usual Hitchin fibration in terms of spectral curves \cite{Hitchin_Systems}. We obtain the following.

\begin{prop}
If $a\in H^0(C,L)$ and $b\in H^0(C,L^2)$ are such that $a^2-4b \neq 0$, then the coarse moduli space of the multiplicative Hitchin fibre $h^{-1}(a,b)$ is a torsor under the Prym variety
\begin{equation*}
P_{(a,b)} = \mathrm{Prym}(Y_{(a,b)}),
\end{equation*} 
where $\pi:Y_{(a,b)}\rightarrow C$ is the spectral curve inside the total space of $L\rightarrow C$ determined locally by the equation
 \begin{equation*}
y^2 + ay + b = 0.
\end{equation*} 
\end{prop}

We now want to make sense of the stack
\begin{equation*}
\cM_{\PGL_2,\Mat_2,(L,b)} = \Map_{(L,b)}(C,[\Mat_2/\PGL_2 \times \Gm]),
\end{equation*} 
and identify the fibres of the $\PGL_2$ multiplicative Hitchin fibration
 \begin{equation*}
\check{h}:\cM_{\PGL_2,\Mat_2,(L,b)} \rightarrow \cA_{\Mat_2,(L,b)}.
\end{equation*} 
The standard way to do this is to consider the action of the group $\Jac(C)[2]$, parametrizing order $2$ line bundles on $C$, on the stack $\cM_{\SL_2,\Mat_2,(L,b)}$ and defined as follows
\begin{align*}
\Jac(C)[2] \times \cM_{\SL_2,\Mat_2,(L,b)}& \longrightarrow \cM_{\SL_2,\Mat_2,(L,b)} \\
(\xi,(E,\varphi)) & \longmapsto (E\otimes \xi, \varphi).
\end{align*} 
The stack $\cM_{\PGL_2,\Mat_2,(L,b)}$ is identified with the quotient stack $[\cM_{\SL_2,\Mat_2,(L,b)}/\Jac(C)[2]]$. 

Moreover, the spectral correspondence allows us to identify the fibres $\check{h}^{-1}(a,b)$ as torsors under the abelian varieties
\begin{equation*}
\hat{P}_{(a,b)} = P_{(a,b)}/\pi^*\Jac(C)[2].
\end{equation*} 
As explained in appendix \ref{app:Prym}, the abelian varieties $P_{(a,b)}$ and $\hat{P}_{(a,b)}$ are dual.

\section{Twisted invariant theory} \label{sec:TwistedInvTheory}

\subsection{Folding} \label{sec:folding}
Let $G$ be a semisimple simply-laced and simply-connected group over $\C$. We also assume here that the Dynkin diagram of $G$ does not have any component of type  $\sfA_{2\ell}$, for $\ell \in \mathbb{N}$, but we lift this assumption in a different section \ref{sec:inv_Aeven2}. 

The symmetry group of the Dynkin diagram of $G$ is identified with the group $\Out(G)$ of outer classes of automorphisms of $G$. This group is the quotient of  $\Aut(G)$ by the group of \emph{inner automorphisms}, that is, automorphisms of the form $\Ad_g$ for  $g\in G$. A choice of pinning $(B,T,\left\{u_\alpha\right\})$ for $G$ determines a section  $\Out(G) \rightarrow \Aut(G)$. Therefore, if we fix a pinning, every symmetry of the Dynkin diagram of  $G$ induces a unique automorphism  $\theta \in \Aut(G)$. We let $m$ denote the order of this automorphism. It is a well-known result that the pairs $(G,\theta)$, for $G$ simple, are classified by twisted affine Dynkin diagrams (see Table \ref{tab:DiagramAutomorphisms}). More generally, for $G$ semisimple, we can also have  ``trivially induced" automorphisms, which permute several isomorphic simple factors.

The colored nodes of the affine Dynkin diagrams from the list determine a Dynkin diagram of finite type, called the \emph{folded Dynkin diagram}. It determines the \emph{folded root system} $\Phi^\theta$, which is the image of the root system $\Phi=\Phi(G,T)$ under the natural projection $X^*(T)\rightarrow X^*(T)_\theta$. Here, $X^*(T)_\theta=X^*(T)/(1-\theta)$ denotes the \emph{coinvariant group}  associated with the action of the symmetry group on $X^*(T)$. Since $G$ is simply-connected  $X^*(T)$ is the weight lattice of $G$, and the coinvariant group $X^*(T)_\theta$ is in fact a lattice. The norm map induces a natural injection $X^*(T)_\theta \hookrightarrow X^*(T)^\theta$, which identifies $\Phi^\theta$ as a root system inside  $\bbE=X^*(T)^\theta \otimes_\Z \Q $. More precisely, we have
\begin{equation*}
\Phi^\theta = \left\{\Nm(\alpha)=\alpha + \theta(\alpha) + \cdots + \theta^{m-1}(\alpha): \alpha \in \Phi\right\}.
\end{equation*} 
The dual root system of $\Phi^\theta$ is the \emph{cofolded root system} $\Phi_\theta$, defined as follows.
Consider the quotient set $\Phi/\theta$. Given any orbit $\xi \in \Phi/\theta$, we define the corresponding root  $\beta_\xi \in \Phi_\theta$ as
 \begin{equation*}
\beta_\xi = \sum_{\alpha \in \xi} \alpha.
\end{equation*} 
Clearly, the element $\beta_\xi$ is fixed under  $\theta$, so there is a natural inclusion  $\Phi_\theta \subset X^*(T)^\theta$. Note that we have, for every $\alpha \in \Phi$,
\begin{equation*}
\Nm(\alpha) = \tfrac{m}{\mathrm{card}([\alpha])} \beta_{[\alpha]}.
\end{equation*} 
Therefore, $\Phi^\theta$ and $\Phi_\theta$ are indeed dual as root systems inside  $\bbE$. 

If $\Delta=\Delta(G,B,T)=\left\{\alpha_1,\dots,\alpha_r\right\}$ is the set of simple roots, then the orbit set $\Delta/\theta = \left\{\xi_1,\dots,\xi_\ell\right\}$ determines the simple roots $\Delta_\theta=\left\{\beta_1,\dots,\beta_\ell\right\}$ by putting $\beta_i = \sum_{\alpha \in \xi_i} \alpha$. Similarly, if $\left\{\omega_1,\dots,\omega_r\right\} \subset X^*(T)$ is the set of fundamental weights, with orbit set $\left\{\eta_1,\dots,\eta_\ell\right\}$, we obtain a set of fundamental weights $\left\{\varpi_1,\dots,\varpi_\ell\right\}\subset X^*(T)^\theta$ for $\Phi_\theta$ by putting  $\varpi_i=\sum_{\omega \in \eta_i} \omega$. From here it follows that the root lattice $\mathbf{R}(\Phi_\theta)$ associated with the cofolded root system $\Phi_\theta$ is the invariant lattice $\mathbf{R}(\Phi)^\theta$, where $\mathbf{R}(\Phi)=X_*(T)=X^*(T^{\ad})$ is the root lattice of $\Phi$. On the other hand, the root lattice $\mathbf{R}(\Phi^\theta)$ associated with the folded root system $\Phi^\theta$ is the coinvariant lattice  $\mathbf{R}(\Phi)_\theta$ and, dually, the weight lattice of $\Phi_\theta$ is  $X^*(T)^\theta$.

The group $G^\theta$ determined by the root datum
\begin{equation*}
\Psi^\theta = (X^*(T)_\theta, \Phi^\theta, X_*(T)^\theta,\Phi_\theta)
\end{equation*} 
is the fixed point subgroup of $G$ under $\theta$. The subgroup $T^\theta \subset G^\theta$ is a maximal torus.
On the other hand, the group $G_\theta$ determined by the root datum
 \begin{equation*}
\Psi_\theta = (X^*(T)^\theta, \Phi_\theta, X_*(T)_{\theta}, \Phi^\theta)
\end{equation*} 
is called the \emph{coinvariant group}. It has as maximal torus the coinvariant torus $T_\theta=T/(1-\theta)(T)$. Note that the character lattice of $G_\theta$ is the weight lattice of  $\Phi_\theta$. This means that $G_\theta$ is simply-connected. Consider now the Langlands dual group $(G^\theta)^\vee$, determined by the root datum
\begin{equation*}
	(\Psi^\theta)^\vee = (X_*(T)^\theta, \Phi_\theta, X^*(T)_\theta,\Phi^\theta).
\end{equation*} 
The character lattice of $(G^\theta)^\vee$ is the root lattice of $\Phi_\theta$. Therefore, $(G^\theta)^\vee$ is of adjoint type and in fact $(G^\theta)^\vee =(G_\theta)^{\ad}$, or equivalently, $((G^\theta)^\vee)^{\Sc}=G_\theta^{\Sc}$.

\subsection{Twisted conjugation} Given a pair $(G,\theta)$ as above, we can consider the \emph{$\theta$-twisted conjugation} action of $G$ on itself, by setting
 \begin{equation*}
g * x = gx\theta(g)^{-1},
\end{equation*} 
for $g,x \in G$. When $\theta$ is the identity, we recover the adjoint action. Equivalently, we can consider the (untwisted) conjugation action  of $G$ on the exterior component $G\theta$ of the non-connected group  $\tilde{G}=G \rtimes \langle \theta \rangle$. That is, given  $g\in G$ and  $x\theta \in G\theta$, we have
 \begin{equation*}
\Ad_g(x\theta) = g (x\theta) g^{-1} = (g x \theta(g)^{-1}) \theta = (g* x) \theta.
\end{equation*} 

For each $i=1,\dots,\ell$, consider the highest weight representation $(V_i,\rho_i)$ of $G_\theta$ associated with the fundamental weight  $\varpi_i$, and let  $\chi_i=\mathrm{tr}(\rho_i)$. Since $\varpi_i$ is also a dominant weight of  $G$, we can also consider the representation  $(V_{\varpi_i},\rho_{\varpi_i})$ of $G$, with highest weight  $\varpi_i$. Let  $\tilde{\rho}_i$ the lift of this representation to a representation of the non-connected group $\tilde{G}$, and let $\tilde{\chi}_i=\mathrm{tr}(\tilde{\rho}_i)$. The restriction $\tilde{\chi}_i|_{T\theta}$ can be regarded as a function in $k[T]$ via the natural isomorphism $T\rightarrow T\theta$. It is a result of Mohrdieck \cite{Mohrdieck} that this function factors through $T_\theta$ and coincides with $\chi_i$ as an element of $k[T_\theta]\subset k[T]$.

We can now state the main result of twisted invariant theory for a group, which relates the GIT quotient $\chi_{G\theta}:G\rightarrow\fc_{G\theta}:=G\git_* G=G\theta\git \Ad(G)$ of a group by its twisted conjugation action, and the GIT quotient of its coinvariant group under the adjoint action. A proof can be found in G. Wang's paper \cite{Griffin_Twisted}, where it is attributed to Mohrdieck \cite{Mohrdieck}.

\begin{thm}[Mohrdieck--Wang]
There are natural isomorphims
\begin{equation*}
\fc_{G\theta} \overset{\sim}{\longleftarrow} T\git ((1-\theta)T\rtimes W^\theta) \overset{\sim}{\longrightarrow} T_\theta /W^\theta \overset{\sim}{\longrightarrow} G_\theta\git \Ad(G_\theta)= \fc_{G_\theta}.
\end{equation*} 
Moreover, we have
\begin{equation*}
\Spec k[\tilde \chi_{1},\dots,\tilde\chi_{\ell}] = T^\theta \theta /\tilde{W} \overset{\sim}{\longrightarrow} T_\theta/W^\theta = \Spec k[\chi_1,\dots,\chi_\ell],
\end{equation*} 
where $\tilde{W}= [(1-\theta)(T)\cap T^\theta]\rtimes W^\theta$ denotes the \emph{exterior Weyl group} of $(G,\theta)$ and $W^\theta\subset W$ is the fixed point subgroup under the natural action of  $\theta$ on $W$.
\end{thm}

Consider now a very flat reductive monoid $M$ with  $M^\der=G$, and suppose that the automorphism  $\theta$ of  $G$ extends to an automorphism $\theta_M$ of the whole unit group $M^\times$. In particular, if $M=\Env(G)$ is the Vinberg monoid, we can always extend $\theta$ to an automorphism  $\theta_+ \in \Aut(G_+)$ by putting
\begin{align*}
\theta_+: G_+ & \longrightarrow G_+ \\
[g,t] & \longmapsto [\theta(g),\theta(t)].
\end{align*} 

The $\theta$-twisted conjugation action of  $G$ on itself extends to the $\theta_M$-twisted conjugation action of $M^\times$ on itself, so in turn we can consider the $\theta$-twisted conjugation action of $G$ on  $M$. Even though $\theta_M$ might not extend to an automorphism of the whole monoid  $M$, we can at least formally understand the $\theta$-twisted conjugation action as a conjugation action on an ``exterior component" $M\theta$. We can then consider the corresponding invariant quotient
 \begin{equation*}
\chi_{M\theta}: M \longrightarrow \fc_{M\theta}:=M\git_* G= M\theta \git \Ad(G).
\end{equation*} 
The following is a result of G. Wang \cite{Griffin_Twisted}.

\begin{prop}[G. Wang]
The inclusion $T_M \hookrightarrow M^\times$ of the maximal torus, induces an isomorphism	
\begin{equation*}
\bar{T}_M \theta \git [\Ad(T^{\Sc} \rtimes W^\theta)]\overset{\sim}{\longrightarrow}\fc_{M\theta}.
\end{equation*} 
Moreover, since $M$ is very flat, we also have
\begin{equation*}
\fc_{M\theta} \cong \fc_{G\theta} \times \bbA_M.
\end{equation*} 
\end{prop}

In particular, for $M=\Env(G)$, we have
\begin{align*}
\chi_{\Env(G)\theta}: \Env(G) & \longrightarrow \bbA^\ell \times \bbA^r \\
[(g,t)] & \longmapsto (\tilde{\chi}_1(g\theta),\dots,\tilde{\chi}_\ell(g\theta),\alpha_1(t),\dots,\alpha_r(t)).
\end{align*} 

\subsection{Monoids and coinvariant group} \label{sec:coinv_monoids}
The inclusion of the root lattice $\mathbf{R}(\Phi_\theta)=\mathbf{R}(\Phi)^\theta \subset \mathbf{R}(\Phi)$ induces a surjection at the level of toric varieties $\bbA_G \rightarrow \bbA_{G_\theta}$.
Here, we recall that $\bbA_{G_\theta}$ denotes the abelianization of the Vinberg monoid $\Env(G_\theta)$;  that is, $\bbA_{G_\theta}$ is the affine space $\bbA^\ell$ with the  $T_\theta$-action
 \begin{equation*}
t \cdot \bm{y} = t\cdot (y_1,\dots,y_\ell) = (\beta_1(t)\cdot y_1,\dots,\beta_\ell(t)\cdot y_\ell) =: \bm{\beta}(t)\cdot \bm{y}.
\end{equation*} 
The surjection $\bbA_G\rightarrow \bbA_{G_\theta}$ is equivariant with respect to the actions of $T$ and $T_\theta$, respectively, and the natural projection $T\rightarrow T_\theta$.

Let  $M$ be a very flat monoid with derived group $G$. By Vinberg's classification, there is a map $\bbA_M \rightarrow \bbA_G$ such that  $M$ is obtained from $\Env(G)$ as a pullback through this map. Composing with the map $\bbA_G \rightarrow \bbA_{G_\theta}$, we obtain a map $\bbA_M\rightarrow \bbA_{G_\theta}$. We define $M_\theta$ as the pullback of  $\Env(G_\theta)$ through this map. This $M_\theta$ is by construction a very flat monoid with derived group  $G_\theta$ and abelianization $\bbA_M$. The construction $M\mapsto M_\theta$ determines a functor from the category  $\mathcal{VF}(G)$ of very flat monoids with derived group $G$ and excellent maps (i.e. maps obtained as pullbacks of maps of the abelianizations, see \cite[Definition 2.3.9]{Griffin_Lemma}) to the category $\mathcal{VF}(G_\theta)$ of very flat monoids with derived group $G_\theta$ and excellent maps. Moreover, this functor is fully faithful since, for any two objects $M$ and $N$ of  $\mathcal{VF}(G)$, we have
\begin{equation*}
\Hom_{\mathcal{VF}(G_\theta)}(M_\theta, N_\theta) = \Hom_{\text{comm. monoids}}(\bbA_M,\bbA_N) = \Hom_{\mathcal{VF}(G)}(M,N).
\end{equation*} 
Moreover, note that, for every very flat monoid $M$ with derived group $G$, we have
\begin{equation*}
\fc_{M\theta} \cong \fc_{G\theta} \times \bbA_M \cong \fc_{G_\theta} \times \bbA_M \cong \fc_{M_\theta}.
\end{equation*} 

\subsection{Twisted Steinberg quasi-sections}
By assumption, the automorphism $\theta$ of  $G$ preserves a pinning. In particular, this means that $\theta$ acts on the simple roots and that, for each simple root $\alpha \in \Delta$, the corresponding  $1$-parameter group $u_\alpha:\Ga\rightarrow U_\alpha$ satisfies
 \begin{equation*}
\theta(u_\alpha(1)) = u_{\theta(\alpha)}(1).
\end{equation*} 

A \emph{twisted Coxeter datum} $(q,\sigma,\dot{\bm{s}})$ \cite{Griffin_Twisted} consists of the following
\begin{itemize}
	\item a section $q:\Delta/\theta\rightarrow \Delta$,
	 \item a bijection $\sigma:\left\{1,\dots,\ell\right\}\rightarrow \Delta/\theta$,
	 \item a choice $\dot{\bm{s}}=\left\{\dot{s}_\alpha:\alpha \in \Delta\right\}$ of a representative $\dot{s}_\alpha \in N_G(T)$ of each simple reflection $s_\alpha \in W$ which is  $\theta$-stable, in the sense that
		  \begin{equation*}
		\theta(\dot{s}_\alpha) = \dot{s}_{\theta(\alpha)}.
		 \end{equation*} 
\end{itemize}
Once $\Delta$ is fixed, an element $w\in W \rtimes \langle \theta \rangle$ is called a \emph{twisted Coxeter element} if it can be written as $w=w_{q,\sigma}=s_{q\circ \sigma(1)}\dots s_{q\circ \sigma(r)}\theta$, for some $q$ and $\sigma$ as above. We denote by $\mathrm{TCox}(W,\Delta,\theta)$ the set of twisted Coxeter elements.

With a twisted Coxeter datum $(q,\sigma,\dot{\bm{s}})$ we can associate the corresponding \emph{twisted Steinberg quasi-section}
\begin{align*}
\epsilon^{(q,\sigma,\dot{\bm{s}})}: \fc_{G\theta} \cong \bbA^\ell & \longrightarrow G\theta \\
(x_1,\dots,x_\ell) & \longmapsto \left( \prod_{i=1}^\ell u_{q\circ \sigma(i)}(x_i) \dot{s}_i \right) \theta.
\end{align*} 
As in the untwisted case, the dependence on the Coxeter datum is not too strong. Indeed, G. Wang \cite{Griffin_Twisted} proved the following. If  $(q,\sigma,\dot{\bm{s}})$ and $(q,\sigma',\dot{\bm{s}}')$ are two different twisted Coxeter data with $\sigma=\sigma'$, then the corresponding sections are conjugate under some  $t\in T$. On the other hand, if $\dot{\bm{s}}=\dot{\bm{s}}'$, then, for any $x,x'\in \fc_{G\theta}$ such that $x_i=x'_j$ if  $q\circ \sigma(i)$ and $q\circ \sigma'(j)$ are in the same $\theta$-orbit, the elements $\epsilon^{(q,\sigma,\dot{\bm{s}})}(x)$ and $\epsilon^{(q,\sigma',\dot{\bm{s}})}(x')$ are $G$-conjugate.

Again, if $M$ is a reductive monoid with $M^{\der}=G$, we can extend the twisted Steinberg quasi-section to a quasi-section
\begin{align*}
\epsilon_{M\theta}^{(q,\sigma,\dot{\bm{s}})}: \fc_{M\theta} \cong \bbA_M \times \fc_{M\theta} & \longrightarrow M\theta \\
(a,x) & \longmapsto \delta_M(a)\epsilon^{(q,\sigma,\dot{\bm{s}})}(x)
\end{align*} 
of the GIT quotient $\chi_{M\theta}:M\rightarrow \fc_{M\theta}$. It is clear, already from the untwisted case, that the dependence of this quasi-section on the twisted Coxeter datum is strong.

\subsection{Stabilizers and special loci}
Let $M$ be a very flat monoid with derived group $G$. Similarly to the untwisted case, we say that an element $x \in M$ is \emph{$\theta$-regular} if the stabilizer
\begin{equation*}
I_{G,M\theta,x}:=\left\{g\in G: gx\theta(g)^{-1}=x\right\}
\end{equation*} 
has the minimal possible dimension. We say that $x$ is \emph{$\theta$-semisimple} if its $\theta$-twisted conjugation orbit contains an element of $\bar{T}_M$, and we say that it is \emph{$\theta$-regular semisimple} if it is both $\theta$-regular and  $\theta$-semisimple. 
The same notions apply to $\theta$-twisted adjoint orbits and we consider the subsets  $M^{\treg}$, $M^{\tSs}$ and $M^{\trs}$ of $\theta$-regular,  $\theta$-semisimple and $\theta$-regular semisimple elements, respectively. 

The stabilizers define a group scheme $I_{G,M\theta}\rightarrow M$ called the \emph{$\theta$-centralizer group scheme}, which restricts to a smooth commutative group scheme over $M^{\treg}$, and as in the untwisted case, the descent argument of Ngô can be adapted to yield the following \cite{Griffin_Twisted}.
\begin{prop}[G. Wang] \label{prop:TwistedRegCent}
There is a unique commutative group scheme $J_{G,M\theta}\rightarrow \fc_{G,M\theta}$ with a $G$-equivariant isomorphism
 \begin{equation*}
\chi_{G,M\theta}^* J_{G,M\theta}|_{M^{\treg}} \overset{\sim}{\longrightarrow} I_{G,M\theta}|_{M^{\treg}},
\end{equation*} 
which can be extended to a homomorphism $\chi^*_{G,M\theta}J_{G,M\theta}\rightarrow I_{G,M\theta}$. We call this $J_{G,M\theta}$ the \emph{regular $\theta$-centralizer} group scheme.
\end{prop}

We can repeat the discussion of the untwisted case. G. Wang \cite{Griffin_Twisted}, shows that, for every choice of Coxeter datum, the Steinberg quasi-section $\epsilon^{(q,\sigma,\dot{\bm{s}})}$ maps $\fc_{G\theta}$ inside the regular locus, and in fact it cuts transversally each regular orbit. Therefore $[(M^\times)^{\treg}\theta/G]\rightarrow \fc_{M^\times \theta}$ is a $J_{G,M^\times \theta}$-gerbe neutralized by any twised Steinberg quasi-section. However, for some points of $\fc_{M\theta}$, the fibre can have several regular orbits (whose number is bounded above by the number of twisted Coxeter elements). Thus $[M^{\treg}\theta /G]\rightarrow \fc_{M\theta}$ is in general a finite union of $J_{G,M\theta}$-gerbes, each of which is neutralized by some twisted Steinberg quasi-section. 

Like in the untwisted case, we want to identify two special divisors inside the GIT quotient $\fc_{M\theta}\cong \fc_{M_\theta}$. The \emph{twisted boundary divisor} $\fB_{M\theta}$ is analogously defined as the complement of
\begin{equation*}
\fc^\times_{M\theta}:= \fc_{G\theta} \times \bbA_M^\times.
\end{equation*} 
We can also define the \emph{extended twisted discriminant divisor} $\fD_{M\theta}$ by taking the pullback of the divisor $\fD_{\Env(G_\theta)}\subset \fc_{\Env(G_\theta)}$ through the natural map $\fc_{M_\theta}\rightarrow \fc_{\Env(G_\theta)}$. Recall that the divisor $\fD_{\Env(G_\theta)}$ is the vanishing set of the following $W$-invariant function on $(\bar T_{\theta})_+$,
\begin{equation*}
\mathrm{Disc}_{\Env(G_\theta)}:= \prod_{\beta \in (\Phi_\theta)+} (\beta,0) \prod_{\alpha \in \Phi_\theta} (1-(0,\alpha)) = (2\rho_\theta,0) \prod_{\alpha \in \Phi_\theta} (1-(0,\alpha)),
\end{equation*} 
where $\rho_\theta$ is the half-sum of the roots in $\Phi_\theta$.
Note that this definition makes sense because, by definition, $\Phi_\theta$ is the root system of $G_\theta$. It follows from the untwisted case that these two divisors intersect properly and in turn that the codimension of $\fD_{M\theta}^\times:=\fD_{M\theta}\cap \fB_{M\theta}$ is at least $2$. We say that an element of  $M$ is \emph{$\theta$-group-like} if it belongs to the preimage of $\fc_{M\theta}^{\gl}:=\fc_{M\theta}\setminus \fD_{M\theta}^\times$, and denote by $M^{\tgl}=\chi_{M\theta}^{-1}(\fc_{M\theta}^{\gl})$ the locus of $\theta$-group-like elements.
In the twisted case it is also true that  $\fc_{M\theta}^{\trs}$, the image of $M^{\trs}$ under the GIT quotient, is the complement of the discriminant divisor $\fD_{M\theta}$, and each fibre of $\chi_{M\theta}$ over $\fc_{M\theta}^{\trs}$ consists of a single twisted adjoint orbit. It follows that an element of $M$ is $\theta$-group-like if it is either in the unit group or $\theta$-regular semisimple, that is
\begin{equation*}
M^{\tgl} = M^\times \cup M^{\trs}.
\end{equation*} 
More generally, the same argument from the untwisted case shows that each fibre of $\chi_{M\theta}$ over $\fc_{M\theta}^{\tgl}$ has a single regular orbit, so the map $[M^{\tgl} \theta/G]\rightarrow \fc_{M\theta}^{\tgl}$ is a $J_{G,M\theta}$-gerbe, neutralized by any twisted Steinberg quasi-section $\epsilon^{(q,\sigma,\dot{\bm{s}})}$, for any choice of twisted Coxeter datum $(q,\sigma,\dot{\bm{s}})$.

\subsection{Galois description} Recall the isomorphisms
$\fc_{M\theta} \cong \fc_{M_\theta} \cong \bar{T}_{M_\theta}/W^\theta$. We can thus consider the cameral cover
\begin{equation*}
\pi_{M\theta}: \bar{T}_{M_\theta} \longrightarrow \fc_{M\theta}.
\end{equation*} 
We introduce the Weil restriction
\begin{equation*}
\Pi_{M\theta} = \pi_{M\theta,*} (\bar{T}_{M_\theta} \times T^\theta),
\end{equation*} 
on which $W$ acts diagonally, and also the invariant subgroup
 \begin{equation*}
J^1_{M\theta} = \Pi_{M\theta}^W
\end{equation*} 
and its neutral connected component
\begin{equation*}
J^0_{M\theta}= \Pi_{M\theta}^{W,0}.
\end{equation*} 

As explained in \cite{Griffin_Twisted}, a twisted version of the Grothendieck--Springer resolution provides an open embedding $J_{M\theta} \hookrightarrow J^1_{M\theta}$, which restricts to an isomorphism over $\fc^{\trs}_{M\theta}$. In fact, the isomorphism extends over the whole base, that is, we have the following.
\begin{prop} \label{prop:galoisJ}
The embedding $J_{M\theta}\hookrightarrow J^1_{M\theta}$ is an isomorphism.	
\end{prop}

\begin{proof}[Sketch of the proof]
For the sake of clarity, we explain the main argument of the proof of the proposition. First of all note that, since the complement of the group-like locus $\fc^{\tgl}_{M\theta}$ has codimension at least $2$, it suffices to show the isomorphism for $a$ in this locus. Actually, we can further restrict to the complement $\fc^{\tgl}_{M\theta} \setminus \fD_{G_\theta}^{\mathrm{sing}}$, since the singular locus of the discriminant divisor also has codimension at least $2$.
For such an $a$, either $a$ is $\theta$-regular semisimple, and there is nothing to prove, or  $a \in \fD_{G_\theta}\setminus \fD_{G_\theta}^{\mathrm{sing}}$.

Suppose then that $a\in \fD_{G_\theta}\setminus \fD_{G_\theta}^{\mathrm{sing}}$. In this case we can find an element $t\in T_\theta$ mapping to $a\in \fc_{G_\theta}\cong T_\theta/W^\theta$, and such that there exists a unique root $\beta\in \Phi_\theta$ with $t^{\beta}=1$.
By construction of the cofolded root system, $\beta=\beta_{\xi}$, where $\xi \in \Phi/\theta$ is the an  orbit of roots in $\Phi$. We let now $G_\beta \subset G$ be the reductive $\theta$-stable subgroup generated by the maximal torus $T$ and by the root groups corresponding to the roots in $\xi$. 

Consider the quotient $\fc_\beta:= T_\theta/s_\beta$, with the projection $\pi_\beta:T_\theta \rightarrow \fc_\beta$. We let $I_\beta \rightarrow G_\beta$ denote the group scheme of stabilizers of the $\theta$-twisted action of $G_\beta$ on itself and  $J_\beta\rightarrow \fc_\beta$ the corresponding regular  $\theta$-centralizer. We can also consider $J^1_\beta=[\pi_{\beta,*} (T_\theta \times T^\theta)]^{s_\beta}$.

If we let $\bar{a}\in \fc_{\beta}$ be the image of $t$ under the natural projection, it is clear that the natural finite flat map  $p_{\beta}:\fc_{\beta}\rightarrow \fc_{G\theta}$ sends $a$ to $\bar{a}$ and is étale at $a$. This means that, near $a$, the canonical map $p_\beta^* J_{M\theta}\rightarrow J_{\beta}$ is an isomorphism, and the same is true for $p_\beta^* J^1_{M\theta}\rightarrow J^1_{\beta}$. Therefore, we are reduced to show that the embedding $J_\beta\hookrightarrow J^1_\beta$ is an isomorphism.

Without loss of generality, we can assume that $\beta$ is a simple root, and thus the orbit $\xi$ corresponds to an orbit of nodes in the Dynkin diagram of $G$. From our assumptions on  $G$  (that is, since $G$ does not have factors of type $\sfA_{2\ell}$), the nodes in $\xi$ are disjoint. Therefore, $G_\beta$ is obtained up to isogeny as a product of isomorphic reductive groups of rank $1$ with some tori.  The automorphism $\theta$ acts on  $G_\beta$ by permuting these factors and by inversion in some of the tori. Every reductive group of rank $1$ is isomorphic to either $\GL_2$,  $\SL_2$ or $\PGL_2$. The isomorphism of $J_\beta$ and $J^1_\beta$ can then be verified by direct computation on a case by case study. We refer to \cite{Griffin_Twisted} for more details.
\end{proof}

Recall that, when restricted to the group-like locus, the cameral cover $\pi_{M_\theta}:\bar{T}_{M_\theta}\rightarrow \fc_{M_\theta}$ is a cameral cover associated with the root system $\Phi_\theta$. The above shows that  $J_{M\theta}=J_{\pi_{M\theta},G^\theta}$. Again, this makes sense because $\Phi_\theta$ is the coroot system of $G^\theta$.

\subsection{The special case \texorpdfstring{$\sfA^{(2)}_{2\ell}$}{A2l}} \label{sec:inv_Aeven2}
We return now to the case excluded at the beginning of section \ref{sec:folding}. Assume then for simplicity that $G=\SL_{2\ell+1}$, for $\ell \in \mathbb{N}$. Consider the standard pinning of $G$; that is, $T$ consists of diagonal matrices, $B$ of upper triangular ones, the simple roots are
\begin{equation*}
\alpha_i(\mathrm{diag}(x_1,\dots,x_{2\ell})) = x_i x_{i+1}^{-1},
\end{equation*} 
and the root groups are parametrized by
\begin{equation*}
u_i(x) = I_{2n} + xE_{i,i+1},
\end{equation*} 
for $E_{i,i+1}$ the matrix with $1$ on the position $(i,i+1)$ and $0$ elsewhere.

The only non-trivial element of $\Out(\SL_{2\ell+1})$ is the equivalence class of the automorphism
\begin{equation*}
A \mapsto (A^T)^{-1}.
\end{equation*} 
This automorphism does not preserve the standard pinning, but the following does
\begin{equation*}
\theta_\ell: A \mapsto N_{2\ell+1} (A^T)^{-1} N_{2\ell+1}^{-1},
\end{equation*} 
where $N_{2\ell+1}$ is the antidiagonal matrix $N_{2\ell+1}=\mathrm{antidiag}(1,-1,\dots,-1,1)$.

Twisted conjugation and Proposition \ref{prop:TwistedRegCent} can also be formulated for this situation, since they do not use the assumption of $G$ not having components of type $\sfA_{2\ell}$ (see \cite{Griffin_Twisted}). The main difference is that the embedding $J_{M\theta} \hookrightarrow J^1_{M\theta}$ from Proposition \ref{prop:galoisJ} is no longer an isomorphism. Indeed, if we try to emulate the proof of Proposition \ref{prop:galoisJ} without the assumption, we are led to consider the extra case in which the orbit $\xi$ consists of two nodes  of the Dynkin diagram joined by an edge. Therefore,  $G_{\beta}$ is a reductive group with semisimple part isomorphic to $\SL_3$, on which  $\theta$ acts as
 \begin{equation*}
\theta|_{G_{\beta}^{\der}}=\theta_1:A \mapsto N_3 (A^T)^{-1} N_3^{-1},
\end{equation*} 
for $N_3=\mathrm{antidiag}(1,-1,1)$.
We are thus reduced to study the pair $(\SL_3,\theta_1)$.

The action of $\theta=\theta_1$ on the maximal torus of $\SL_3$ sends $\mathrm{diag}(x_1,x_2,x_3)$ to the element $\mathrm{diag}(x_{3}^{-1},x_{2}^{-1},x_{1}^{-1})$. Therefore, the fixed point torus $T^\theta$ is
 \begin{equation*}
T^\theta = \left\{X = \mathrm{diag}(x_1,1,x_1^{-1}) : x_1 \in \Gm\right\}
\end{equation*} 
This determines an isomorphism $T^\theta \cong \Gm$.
We can write the coordinate ring $k[T]=k[x_1^{\pm 1}, x_2^{\pm 1}]$. The action of $\theta$ on this ring permutes $x_1$ with $x_1 x_2$ so, writing $z=x_1^2 x_2$ we have $k[T_\theta]=k[z^{\pm 1}]$, which determines an isomorphism $T_\theta \cong \Gm$.

The only non-trivial fixed element of the Weyl group is represented by the matrix $N_3$, whose conjugation action on  $T$ permutes $x_{1}$ and $x_{3}$. Thus, $W^\theta$ acts on $T^\theta\cong \Gm$ and on $T_\theta \cong \Gm$ as the inversion $z\mapsto z^{-1}$ on $\Gm$. The quotient $\Gm/(z\mapsto z^{-1})$ is isomorphic to the affine line $\bbA^{1}$, with the projection map
\begin{align*}
\pi:\Gm  \longrightarrow \bbA^1, \ z  \longmapsto a=z + z^{-1}.
\end{align*} 
This is a finite cover ramified at the points $a=2$ and  $a=-2$. Therefore, the fibres $J^1_{a}$ of the group scheme
 \begin{equation*}
J^1 = \pi_*(T_\theta \times T^\theta)^{W^\theta}
\end{equation*} 
are isomorphic to $\Gm$ for  $a\neq \pm 2$ and to  $\Ga \times \mu_2$ for  $a=2$ or  $a=-2$.

The GIT quotient $\SL_3\theta \git \SL_3$ is isomorphic to  $\bbA^1$ and we can consider the regular  $\theta$-centralizer  $J \rightarrow \bbA^1$, which as usual is defined as the descent over the quotient $\bbA^1$ of the group scheme of stabilizers of the  $\theta$-twisted conjugation action on $\SL_3$. There is an open and closed embedding $J \hookrightarrow J^1$ which restricts to an isomorphism over $\bbA^1 \setminus \left\{\pm 2\right\}$. Therefore, $J\rightarrow \bbA^1$ must be isomorphic to one of the following
 \begin{enumerate}
	 \item the whole group $J^1\rightarrow \bbA^1$,	
	 \item the fiberwise neutral connected component $J^0\rightarrow \bbA^1$, 
	 \item the group $J_{+}\rightarrow \bbA^1$, with one connected component over $a=2$ and two connected components over $a=-2$, 
	 \item the group $J_{-}\rightarrow \bbA^1$, with one connected component over $a=-2$ and two connected components over $a=2$.
\end{enumerate}
An explicit computation by G. Wang \cite{Griffin_Twisted} shows that $J$ is in fact isomorphic to $J_{-}$. 

For general rank $(\SL_{2\ell +1},\theta=\theta_\ell)$, we can define the cofolded root system by putting
\begin{equation*}
\beta_{ \left\{\alpha_i,\alpha_{2\ell+1-i}\right\}} = \alpha_i + \alpha_{2\ell+1-i},
\end{equation*} 
for $i \neq \ell$, and 
\begin{equation*}
\beta_{\left\{\alpha_\ell, \alpha_{\ell+1}\right\}} = 2 (\alpha_\ell + \alpha_{\ell +1}),
\end{equation*} 
which determines a root system of type $\sfC_\ell$. The Weyl group of $\sfC_\ell$ is generated by $\ell$ simple reflections $s_1,\dots,s_{\ell}$. The first $\ell-1$ of these reflections $s_i$ act on  $T_\theta \cong \Gm^\ell$ by permuting the $i$-th and $(i+1)$-th coordinates. The last reflection $s_l$ sends the last coordinate $z_\ell$ to its inverse $z_\ell^{-1}$. The subset of fixed points in $\Gm^\ell$ under $s_l$ projects on $T_\theta/W(\sfC_\ell) \cong \bbA^\ell$ to the pair of hyperplanes
\begin{equation*}
H_{\pm}=\left\{(a_1,\dots,a_\ell)\in \bbA^\ell: a_\ell = \pm 2\right\}.
\end{equation*} 
The GIT quotient  $\SL_{2\ell +1} \theta \git \SL_{2\ell +1}$ is isomorphic to $\bbA^\ell$, and it follows from our discussion for $\SL_3$ that $J^1$ has two fibrewise connected components over $H_{\pm}$ and one elsewhere, and that the regular centralizer $J\rightarrow \bbA^\ell$ is isomorphic to the subgroup $J_-\rightarrow \bbA^\ell$ of $J^1$ with one connnected component over  $a_\ell=-2$ and two connected components over  $a_\ell=2$. 

Note that the ``dual"  of the group scheme $J_-$, in the sense of Remark \ref{rmk:dualJ} at the end of section \ref{sec:CameralCurvesAbelianDuality} is the group scheme $\check{J}_+$, for $\check{J}=\pi_*(T_\theta \times (T^\theta)^\vee)^{W^\theta}$.
There is a way of obtaining $\check{J}_+$ as the regular centralizer of something, which is consistent with the approach of Tachikawa \cite{Tachikawa}.	
We can introduce an automorphism $\vartheta=\vartheta_\ell$ of  $\SL_{2\ell +1}$ which plays the role of a ``dual" for $\theta_\ell$.  This automorphism $\vartheta$ is the ``standard" automorphism of order  $4$ introduced by Kac in his book \cite{Kac} (see also the paper of Besson and Hong \cite{Besson-Hong}), which is almost compatible the standard pinning, but sends the root group parametrization $u_{\ell}(x)$ to $u_{\ell+1}(-ix)$, for $i$ a square root of $-1$. More precisely, we define $\vartheta_\ell$ as
 \begin{equation*}
\vartheta_\ell = \Ad_{R_\ell} \circ \theta_\ell,
\end{equation*} 
where $R_\ell$ is the diagonal matrix $R_{\ell}=\mathrm{diag}(1,i,-i,\dots,i,-i)$. The fixed points of $\vartheta_\ell$ form a group whose neutral connected component is isomorphic to the symplectic group $\mathrm{Sp}_{2\ell}$. Therefore, the tori $T^{\theta}$ and $T^{\vartheta}$ are in natural duality. When we restrict to $\ell=1$, a direct computation for $\SL_3$, similar to the one in \cite{Griffin_Lemma} shows that the regular $\vartheta$-centralizer $J_\vartheta\rightarrow \bbA^1$ coincides with $\check{J}_{+}$. The same arguments as in Proposition \ref{prop:galoisJ} allow us to give an isomorphism $J_\vartheta \cong \check{J}_+$ for general $\ell$.

\section{Twisted multiplicative Hitchin fibrations} \label{sec:TwistedMultHitchin}
\subsection{Construction} Let $C$ be a smooth complex projective curve. A \emph{twisted multiplicative Hitchin fibration} over $C$ is associated with the following data:
\begin{itemize}
	\item A semisimple simply-connected and simply-laced group $G$ of rank $r$ over  $\C$, with a prescribed pinning $(B,T,\left\{u_\alpha\right\})$.
	 \item An automorphism $\theta$ of  $G$ of order $m$ preserving the prescribed pinning.
	 \item A reductive monoid $M$ with derived group $G$.
\end{itemize}

\begin{rmk}
A slight difference with the untwisted case is that we are imposing the group $G$ to be simply-connected, and thus it is determined by the monoid $M$, since $G=M^\der$. In the untwisted case, we assume that $M^\der$ is simply-connected, but the group $G$ acting on $M$ through the adjoint action can be $M^\der$ or any other semisimple group centrally isogenous to it. Possibly, the restriction on  $G$ being simply-connected that we are imposing in the twisted case can be lifted as well, and one can consider the twisted action on $G$ of other isogenous groups, mimicking the untwisted situation. We do not pursue this direction here. 
\end{rmk}

The $\theta$-twisted conjugation action of $G$ on itself extends to an action of  $G$ on the reductive monoid $M$. As in the untwisted case, the natural left multiplication action of $Z=Z^0_{M^\times}$ on $M$ commutes with the action of $G$. Therefore, we can consider the following sequence of stacky quotients
\begin{center}
\begin{tikzcd}
	\left[M\theta/(G\times Z)\right] \rar & \left[\fc_{M\theta} / Z\right] \rar & \left[\bbA_M / Z\right] \rar & BZ.
\end{tikzcd}
\end{center}
Let $D$ be a map  $C\rightarrow[\bbA_M/Z]$, and consider the mapping stacks 
\begin{align*}
\cM_{M\theta}&:=	\cM_{M\theta,D} := \Map_{D}(C,\left[M\theta/(G\times Z)\right]) = \Map_C (C,C \times_{[\bbA_M/Z],D} \left[M\theta/(G\times Z)\right]), \\
\cA_{M\theta} &:= \cA_{M\theta,D} = \Map_{D}(C,\left[\fc_{M\theta}/Z\right])= \Map_C (C,C \times_{[\bbA_M/Z],D}\left[\fc_{M\theta}/Z\right]).
\end{align*} 
Note that we have a natural isomorphism $\cA_{M\theta} \cong \cA_{M_\theta}$, induced by the isomorphism $\fc_{M\theta} \cong \fc_{M_\theta}$.
\begin{defn}
The (twisted) \emph{multiplicative Hitchin fibration} associated with the data $(G,\theta,M)$ and with meromorphic datum $D$ is the natural map
 \begin{align*}
h_{M\theta}:=h_{M\theta,D}: \cM_{M\theta,D} & \longrightarrow \cA_{M\theta,D}, 
\end{align*} 
induced from the map $[M\theta/(G\times Z)]\rightarrow [\fc_{M\theta}/Z]$.
\end{defn}

\begin{rmk}
The twisted multiplicative Hitchin fibration can also be interpreted in terms of ``twisted" multiplicative Higgs bundles. By this we mean a pair $(E,\varphi)$ formed by a $G$-bundle $E\rightarrow C$ and a rational section  $\varphi$ of the bundle of groups  $E_\theta(G)=E\times_{\Ad} G\theta$ associated with the $\theta$-twisted conjugation action of $G$ on itself. 
The section $\varphi$ is rational, so it has singularities at some points $p_1,\dots,p_n \in C$. Let  $p_i$ be one of this points and consider a local coordinate $z$ near $p_i$. As in the untwisted situation, by Iwahori's theorem, the singularity of $\varphi$ is determined by $\lambda_i$, where $\lambda_i\in X_*(T)_+$ is a dominant cocharacter. 
\end{rmk}

\subsection{Global twisted Steinberg section}
In order to construct an equivariant Steinberg section, G. Wang \cite{Griffin_Twisted} gives a twisted analogue of Proposition \ref{prop:preglobalsection}.
\begin{prop} \label{prop:preglobalsection_twisted}
For each twisted Coxeter datum $(q,\sigma, \dot{\bm{s}})$, one can define an action $\tau_{M\theta}^{(q,\sigma, \dot{\bm{s}})}$ of $Z$ on $M\theta$ such that
\begin{equation*}
\epsilon_{M\theta}^{(q,\sigma,\dot{\bm{s}})} \circ \tau_{\fc_M\theta}(z^{c_\theta}) = \tau_{M\theta}^{(q,\sigma, \dot{\bm{s}})}(z) \circ \epsilon_{M\theta}^{(q,\sigma, \dot{\bm{s}})},
\end{equation*} 
where $\tau_{\fc_M}$ is the natural action of $Z$ on $\fc_M$, and  $c_\theta=|Z_{G}^\theta|$. Moreover, for a fixed $z$, the map $\tau_{M\theta}^{(q,\sigma, \dot{\bm{s}})}(z)$ is a composition of translation by $z^{c_\theta}$ and conjugation by an element of $T$ determined by a homomorphism $Z\rightarrow T$ independent of $z$.
\end{prop}

Fix a twisted Coxeter datum $(q,\sigma,\dot{\bm{s}})$ and let $T_0\subset T$ denote the common kernel of all the roots $\alpha_1,\dots,\alpha_\ell$ in  $q(\Delta/\theta)$. Here, we are reordering the simple roots $\Delta=\left\{\alpha_1,\dots,\alpha_r\right\}$ in such a way that $\alpha_i=q(\sigma(i))$. Consider the fundamental weights $\omega_1,\dots,\omega_r$ and recall that we have the corresponding weights
\begin{equation*}
\varpi_i = \Nm_{\theta} (\omega_i), \ \ i=1,\dots,\ell
\end{equation*} 
of $G_\theta$. There is a lattice homomorphism
\begin{align*}
-\theta^{-1}+\Nm_\theta: X^*(T) & \longrightarrow X^*(T) \\
\mu=\sum_{i=1}^r n_i \omega_i & \longmapsto \mu_0:= -\theta^{-1}(\mu) + \sum_{i=1}^\ell n_i \varpi_i,
\end{align*} 
whose kernel contains the fixed point sublattice $X^*(T)^\theta\subset X^*(T)$. The image of the map $(1-\theta):T_0 \rightarrow T$ coincides with $(1-\theta)(T)$, since $T_0$ and $T^\theta$ generate $T$. Since  $T_0\cap T^\theta \subset Z_G^\theta$, the kernel of the induced map $(1-\theta)^*: X^*(T)\rightarrow X^*(T_0)$ is exactly $X^*(T)^\theta$, and the image has finite index divisible by $c_\theta$. We conclude that there is a unique arrow  $X^*(T_0)\rightarrow X^*(T)$ making the following square commute
\begin{center}
\begin{tikzcd}
X^*(T) \ar{r}{\mu \mapsto \mu_0} \ar{d}{(1-\theta)^*} & X^*(T) \ar{d}{\mu \mapsto c_\theta \mu} \\
X^*(T_0) \ar{r} & X^*(T). 
\end{tikzcd}
\end{center}
This induces a homomorphism $T\rightarrow T_0$,  $t\mapsto t_0$ such that
\begin{equation*}
\mu_0(t^{c_\theta}) = \mu((1-\theta)(t_0));
\end{equation*} 
and also $\alpha_i(t_0)=1$, for $i=1,\dots,\ell$, by definition.

Consider the map $Z\rightarrow T$ from Proposition \ref{prop:preglobalsection_twisted} above, and let
\begin{align*}
\Psi: Z & \longrightarrow  T\\
z & \longmapsto \psi(z)z_T^{-c_\theta}(z_T)_0.
\end{align*} 
We obtain a quasi-section
\begin{equation*}
[\epsilon_{M\theta}^{(q,\sigma,\dot{\bm{s}})}]_{[c_\theta]}: [\fc_{M\theta}/Z]_{[c_\theta]} \longrightarrow [M\theta/\Psi(Z)\times Z]_{[c_\theta]}
\end{equation*} 
which, by composition, induces a quasi-section
\begin{equation*}
[\epsilon_{M\theta}^{(q,\sigma,\dot{\bm{s}})}]_{[c_\theta]}: [\fc_{M\theta}/Z]_{[c_\theta]} \longrightarrow [M\theta/G\times Z]_{[c_\theta]}.
\end{equation*} 

Consider then a $Z$-torsor $L$ and suppose that there exists another $Z$-torsor $L'$ such that $(L')^{\otimes c_\theta}=L$. Like in the untwisted case, this torsor determines a morphism $[\mathrm{ev}]_{L'}:\cA_{M\theta,L}\times C \rightarrow [\fc_{M\theta}/Z]_{[c_\theta]}$ lifting the canonical map $[\mathrm{ev}]_{L}:\cA_{M\theta,L}\times C \rightarrow [\fc_{M\theta}/Z]$. Composing with $[\epsilon_{M\theta}^{(q,\sigma,\dot{\bm{s}})}]_{[c_\theta]}$, we obtain a section
 \begin{equation*}
\epsilon_{L'}^{(q,\sigma,\dot{\bm{s}})}: \cA_{M\theta,L} \longrightarrow \cM_{M\theta,L}
\end{equation*} 
of the natural map $\cM_{M\theta,L}\rightarrow \cA_{M\theta,L}$. Pulling back by a meromorphic datum $D:C\rightarrow[\bbA_M/Z]$, we obtain a section
 \begin{equation*}
\epsilon_{L'}^{(q,\sigma,\dot{\bm{s}})}: \cA_{M\theta,D} \longrightarrow \cM_{M\theta,D}
\end{equation*} 
of the twisted multiplicative Hitchin map $h:\cM_{M\theta,D}\rightarrow \cA_{M\theta,D}$. This is the \emph{global twisted Steinberg section} associated with $L'$ and $(q,\sigma,\dot{\bm{s}})$.

We conclude the following. 
\begin{thm}
Over the Zariski open locus $\cA_{M\theta}^{\gl}\subset \cA_{M\theta}$ of sections mapping $C$ entirely inside $\fc_{M\theta}^{\tgl}$, the stack $\cM_{M\theta}^{\reg}$ of sections factorizing through the $\theta$-regular locus $M^{\treg}$ is isomorphic to the Picard stack $\cP_{M\theta}\rightarrow \cA^{\gl}_{M\theta}$ with fibres
\begin{equation*}
\cP_{M\theta,a}= \Bun_{J_{M\theta,a}/C},
\end{equation*} 
for $J_{M\theta,a}=a^* J_{M\theta}$.
\end{thm}

\subsection{Cameral description} \label{sec:cameralcurve_twisted}
Given a point $a\in \cA_{M\theta}(\C)$, we want to give a Galois description of $\cP_{M\theta,a}$ in terms of the cameral curve $\pi_a:\tilde{C}_a \rightarrow C$ obtained as a pullback of the cameral cover $\pi_{M\theta}=\pi_{M_\theta}: \bar{T}_{M_\theta}\rightarrow \fc_{M\theta}$. Recall from section \ref{sec:cameralcurve} that we had the Zariski open locus $\cA_{M_\theta}^\sharp$ of sections that map $C$ inside the group-like locus, and intersecting transversely the discriminant divisor. Moreover, this locus is non-empty if the $Z$-torsor $L$ associated with the meromorphic datum $D$ is very $G_\theta$-ample.

Recall from Proposition \ref{prop:galoisJ} that $J_{M\theta}$ is isomorphic to $J^1_{M\theta}=\pi_{M\theta,*}(\bar{T}_{M_\theta}\times T^\theta)^{W^\theta}$. Therefore, we have an isomorphism 
\begin{equation*}
\cP_{M\theta,a} \cong \cP^1_{M\theta,a}:=\Bun_{J^1_{M\theta,a}/C},
\end{equation*} 
for $J^1_{M\theta,a}=a^* J^1_{M\theta}= \pi_{a,*}(\tilde{C}_a \times T^\theta)^{W^\theta}$.
A direct consequence of this cameral description is that there is a natural map of Picard stacks $\cP_{M\theta,a}\rightarrow \Bun_{T^\theta/\tilde{C}_a}$, which provides a description of $\cP_{M\theta,a}$ in terms of $W^\theta$-equivariant  $T^\theta$-bundles over the cameral curve $\tilde{C}_a$. In particular, this implies that $\cP_{M\theta,a}$ is a Beilinson $1$-motive.

\subsection{Duality: twisted vs untwisted} 
Assume for now that the group $G$ does not have a factor of type $\sfA_{2\ell}$.
As we note in the previous sections, the GIT quotients $\fc_{M\theta}$ and $\fc_{M_\theta}$ are isomorphic, and the $\theta$-regular centralizer admits a Galois description in terms of the cameral cover $\pi_{M\theta}:\bar{T}_{M_\theta}\rightarrow \fc_{M_\theta}$, which is also the cameral cover for the adjoint action of any semisimple group centrally isogenous to the coinvariant group $G_\theta$ on the monoid  $M_\theta$.

Let  $H=G^\theta$ be the fixed point subgroup of  $G$ under the $\theta$-action. Let us denote by $A=T^\theta$ its maximal torus. As we explain in \ref{sec:folding}, its Langlands dual group $H^\vee=(G^\theta)^\vee$ is of adjoint type, and in fact $H^\vee=(G_\theta)^\ad$. Therefore, we can consider the adjoint action of $H^\vee$ on the monoid $M_\theta$, and the corresponding untwisted multiplicative Hitchin fibration
\begin{equation*}
h_{H^\vee, M_\theta} : \cM_{H^\vee,M_\theta} \rightarrow \cA_{M_\theta}\cong \cA_{M\theta},
\end{equation*} 
with the same meromorphic datum $D:C\rightarrow [\bbA_{M}/Z]$. For every $a\in \cA_{M_\theta}(\C)$, we can consider the Picard stack $\cP_{H^\vee,M_\theta,a}$ acting on the regular fibre of $a$. Recall that if  $a\in \cA^\sharp_{M_\theta}(\C)$, then there is a natural map of Picard stacks $\cP_{H^\vee,M_\theta,a}\rightarrow \Bun_{A^\vee/\tilde{C}_a}$, which provides a description of $\cP_{H^\vee,M_\theta,a}$ in terms of $W^\theta$-equivariant $A^\vee$-bundles over  $\tilde{C}_a$. More precisely, since $H^\vee$ is of adjoint type, we have that the regular centralizer for the action of $H^\vee$ on $M_\theta$ is  
\begin{equation*}
J_{H^\vee,M_\theta,a} \cong J^0_{H^\vee,M_\theta,a} = \pi_*(\tilde{C}_a \times A^\vee)^{W^\theta,0}.
\end{equation*} 
Therefore,
\begin{equation*}
\cP_{H^\vee,M_\theta,a} \cong \cP^0_{H^\vee,M_\theta,a} = \Bun_{J^0_{H^\vee,M_\theta,a}/C}.
\end{equation*} 

Summing up, $\pi_a:\tilde{C}_a\rightarrow C$ is a cameral cover of smooth curves, with simple Galois ramification, and associated with the root system $\Phi_\theta$. The group $J_{H^\vee,M_\theta,a}$ is then identified with the group $J_{\pi_a,H^\vee}$, and in section \ref{sec:cameralcurve_twisted} we show that the group $J_{M\theta,a}$ is identified with $J_{\pi_a,H}$.

We consider then the natural isomorphism $\bD(\Bun_{A/\tilde{C}_a}) \rightarrow \Bun_{A^\vee/\tilde{C}_a}$. If $a\in \cA_{M_\theta}^\sharp(\C)$, the natural map $\cP_{M\theta,a} \rightarrow \Bun_{A/\tilde{C}_a}$ induces a map $\bD(\Bun_{A/\tilde{C}_a})\rightarrow \bD(\cP_{M\theta,a})$. On the other hand, the norm map induces a map $\Nm:\Bun_{A^\vee/\tilde{C}_a}\rightarrow \cP_{H^\vee,M_\theta,a}$. Putting this together, we obtain the duality morphism $\mathrm{S}_a:\bD(\cP_{M\theta,a}) \rightarrow \cP_{H^\vee,M_\theta,a}$, as in the following square
\begin{center}
\begin{tikzcd}
\bD(\Bun_{A/\tilde{C}_a}) \ar{r} \ar{d} & \Bun_{A^\vee/\tilde{C}_a} \ar{d}{\Nm} \\
\bD(\cP_{M\theta,a}) \ar{r}{\mathrm{S}_a} & \cP_{H^\vee,M_\theta,a}.
\end{tikzcd}
\end{center}
As an immediate application of Theorem \ref{thm:duality} we obtain the following. This is Theorem \ref{thm:duality_twisted} in the introduction.

\begin{thm}\label{thm:duality_twisted_text}
For every $a\in \cA_M^\sharp(\C)$, the map $\mathrm{S}_a:\bD(\cP_{M\theta,a})\rightarrow \cP_{H^\vee,M_\theta,a}$ is an isomorphism of Beilinson $1$-motives. In particular, the Fourier-Mukai transform  $\Phi_{\cP_{M\theta,a}}$ induces an isomorphism of derived categories
\begin{equation*}
\mathbf{S}_a: D^b(\QCoh(\cP_{M\theta,a})) \longrightarrow D^b(\QCoh(\cP_{H^\vee,M_\theta,a})).
\end{equation*} 
\end{thm}

\subsection{Duality in the case \texorpdfstring{$\sfA^{(2)}_{2\ell}$}{A2l}} 
Let us then consider the case where $G=\SL_{2\ell +1}$ and $\theta=\theta_\ell$ is the involution introduced in section \ref{sec:inv_Aeven2}. There we also defined the ``dual" automorphism $\vartheta=\vartheta_\ell=\Ad_{R_\ell}\circ \theta_\ell$.  For every $a\in \cA_{M\theta}(\C)$ we can consider the corresponding cameral cover $\tilde{C}_a \rightarrow C$.  Since $\theta$ and $\vartheta$ are in the same outer class, we can identify  $\cA_{M\theta}(\C)$ with $\cA_{M\vartheta}(\C)$, and the corresponding cameral covers.
We can then consider the following group schemes over $C$
\begin{equation*}
J_{M\theta,a}:= a^* J_{M\theta} = a^* J_{-}
\end{equation*} 
and
\begin{equation*}
J_{M\vartheta,a}:= a^* J_{M\vartheta} = a^* \check{J}_{+}.
\end{equation*} 
Note that these two group schemes are dual in the sense of Remark \ref{rmk:dualJ} at the end of section \ref{sec:CameralCurvesAbelianDuality}, since $T^\theta$ and  $T^\vartheta$ are dual tori.

Consider the natural isomorphism $\bD(\Bun_{T^\theta/\tilde{C}_a}) \rightarrow \Bun_{T^\vartheta/\tilde{C}_a}$. If $a\in \cA_{M\theta}^\sharp(\C)$, the natural map $\cP_{M\theta,a} \rightarrow \Bun_{T^\theta/\tilde{C}_a}$ induces a map $\bD(\Bun_{T^\theta/\tilde{C}_a})\rightarrow \bD(\cP_{M\theta,a})$. On the other hand, the norm map induces a map $\Nm:\Bun_{T^\vartheta/\tilde{C}_a}\rightarrow \cP_{M\vartheta,a}$. Putting this together, we obtain the duality morphism $\mathrm{S}_a:\bD(\cP_{M\theta,a}) \rightarrow \cP_{M\vartheta,a}$, as in the following square
\begin{center}
\begin{tikzcd}
\bD(\Bun_{T^\theta/\tilde{C}_a}) \ar{r} \ar{d} & \Bun_{T^\vartheta/\tilde{C}_a} \ar{d}{\Nm} \\
\bD(\cP_{M\theta,a}) \ar{r}{\mathrm{S}_a} & \cP_{M\vartheta,a}.
\end{tikzcd}
\end{center}
As an immediate application of Corollary \ref{corol:duality} (and the subsequent Remark \ref{rmk:dualJ}) we obtain the following. This is Theorem \ref{thm:duality_Aeven2} in the introduction.

\begin{thm}\label{thm:duality_Aeven2_text}
For every $a\in \cA_M^\sharp(\C)$, the map $\mathrm{S}_a:\bD(\cP_{M\theta,a})\rightarrow \cP_{M\vartheta,a}$ is an isomorphism of Beilinson $1$-motives. In particular, the Fourier-Mukai transform  $\Phi_{\cP_{M\theta,a}}$ induces an isomorphism of derived categories
\begin{equation*}
\mathbf{S}_a: D^b(\QCoh(\cP_{M\theta,a})) \longrightarrow D^b(\QCoh(\cP_{M\vartheta,a})).
\end{equation*} 
\end{thm}

\begin{rmk}
Note that, since $\theta$ and $\vartheta$ determine the same class in $\Out(G)$, the stacks $[M\theta/G]$ and $[M\vartheta/G]$, and in turn the moduli stacks $\cM_{M\theta}$ and $\cM_{M\vartheta}$ are isomorphic. The above result can then be reinterpreted as the existence of an involution $\iota:\cM_{M\theta}\rightarrow \cM_{M\theta}$, inducing an involution $\bar{\iota}:\cA_{M\theta} \rightarrow \cA_{M\theta}$ such that, for $a\in \cA^\sharp_{M\theta}(\C)$ the fibres $\cM_{M\theta,a}$ and $\cM_{M\theta,\bar{\iota}(a)}$ are torsors under dual Beilinson $1$-motives. We give an explicit description of this involution for the case $\sfA_2^{(2)}$ in section \ref{sec:A22}.	
\end{rmk}

\section{Rank 3 bilinear bundles: \texorpdfstring{$\sfA_2^{(2)}$}{A22} explicitly} \label{sec:A22}
\subsection{Invariant theory of bilinear forms}
Let $R$ be a commutative ring, let $V$ be a rank $n$ free module over $R$ and let  $\Bil(V)=\left\{b:V\times V \rightarrow R, \text{ $R$-bilinear}\right\}$ denote the space of bilinear forms on $V$. If we fix a basis of  $V$, a bilinear form $b \in \Bil(V)$ is represented by a $n\times n$ matrix  $B\in \Mat_{n}(R)$. Changing basis we obtain a different matrix $B'$, of the form  $B'=A^T B A$, where  $A\in \GL_n(R)$ is the matrix of change of basis. We are concerned with this action of the reductive group $\GL_n$ on the reductive monoid $\Mat_{n}$. If we fix a volume form $\nu \in \wedge^n V^*$, we can consider volume-preserving change of basis, which corresponds to the restricted action of the subgroup $\SL_n \subset \GL_n$.

Suppose now that $n=3$ and that $R=\C$.
There is a natural involution on  $\Bil(V)$ given by dualizing: $b\mapsto b^*$, for  $b^*(v,w)=b(w,v)$. We can decompose $\Bil(V)=S^2 V^* \oplus \wedge^2 V^*$ by putting $b=g+\omega$, for  $g=(b+b^*)/2 \in S^2 V^*$ and $\omega=(b-b^*)/2 \in \wedge^2 V^*$. 
The volume form $\nu \in \wedge^3 V^*$ induces a natural isomorphism 
 \begin{align*}
\tilde{\nu}:V & \longrightarrow \wedge^2 V^*, \  v  \longmapsto \nu(v,-,-).
\end{align*} 

Spaltenstein \cite{Spaltenstein} proved that the invariant ring $k[\Bil(V)]^{\SL_3}=k[a_0,a_1]$ is generated by the functions $a_0(b), a_1(b): \Bil(V)\rightarrow k$ determined by
\begin{equation*}
p_b(x_0,x_1)=\det(x_0 B + x_1 B^T) = a_0(b) x_0^3 + a_1(b) x_0^2 x_1 + a_1(b) x_0 x_1^2 + a_0(b) x_1^3,
\end{equation*} 
where $B$ is a matrix representing  $b$ in any basis. Note that $a_0(b)$ is equal to the determinant of $B$. Now, it is easy to show that we can rewrite
\begin{equation*}
a_0(b) = \det(g) + g(v_\omega,v_\omega),
\end{equation*} 
where $v_\omega \in V$ is such that  $\omega=\tilde{\nu}(v_\omega)$. Therefore, if we denote $s(b)=\det(g)$ and $t(b)=g(v_\omega,v_\omega)$, we can write
\begin{equation*}
p_b(x_0,x_1)=(x_0+x_1)( s(b)(x_0+x_1)^2 + t(b)(x_0-x_1)^2).
\end{equation*} 
Changing variables $y_0=x_0+x_1$ and $y_1=x_0-x_1$ we can rewrite the above polynomial as
\begin{equation*}
q_b(y_0,y_1)=y_0(s(b)y_0^2 + t(b)y_1^2).
\end{equation*} 
Thus, the functions $s$ and $t$ are also generators of  $k[\Bil(V)]^{\SL_3}$.

We wish to consider the natural sequence of maps
\begin{center}
\begin{tikzcd}
	\Bil(V) \rar{(s,t)} & \bbA^2 \rar{\mathrm{pr}_1} & \bbA^1.
\end{tikzcd}
\end{center}
Here, we are regarding $\bbA^2$ as the GIT quotient  $\bbA^2 \cong \Bil(V)\git \SL_3$, and $\bbA^1$ as the image of  $\Bil(V)$ under the determinant map.
For every $\lambda \in \Gm$, we have $(s(\lambda b),t(\lambda b)) = \lambda^3(s(b),t(b))$. In other words, the above sequence induces the following sequence of stacks
\begin{center}
\begin{tikzcd}
	\left[\Bil(V)/\SL_3 \times \Gm \right] \rar{(s,t)} & \left[\bbA_{(3)}^1\times\bbA_{(3)}^1/\Gm \right] \rar{\mathrm{pr}_1} & \left[\bbA_{(3)}^1/\Gm \right].
\end{tikzcd}
\end{center}

\subsection{Bilinear bundles and Hitchin-type fibration}
Consider a smooth complex projective curve $C$ and let $L\rightarrow C$ be a line bundle. By a rank $3$ \emph{($L$-twisted) bilinear bundle} we mean a pair $(V,\gamma)$, where $V \rightarrow C$ is a rank $3$ vector bundle and  $\gamma: V \rightarrow V^*\otimes L$ is an ``$L$-twisted  bilinear form" on  $V$. We can identify such a pair with a map $C \rightarrow [\Bil(V)/\GL_3 \times \Gm]$, lying over the map $C\rightarrow B\Gm$ determined by  $L$.
For any bilinear bundle $(V,\gamma)$, we can always decompose $\gamma = g + \omega$, with  $g\in H^0(C,S^2 V^* \otimes L)$ and $\omega \in H^0(C,\wedge^2 V^* \otimes L)$. 

Given any bilinear bundle $(V,\gamma=g+\omega)$, the determinant of its symmetric part defines a section $\det g \in H^0(C,(\det V)^{-2} \otimes L^3)$. For any non-trivial section $s\in H^0(C,L^3)$ (assuming it exists), we let $\cM=\cM_{(L,s)}$ denote the classifying stack of $L$-twisted bilinear bundles $(V,\gamma=g+\omega)$ with a trivialization $\nu:\det V \cong \sO_C$ identifying $\det g = s$. Equivalently, $\cM$ is the stack of maps $C\rightarrow [\Bil(V)/\SL_3 \times \Gm]$ lying over the map $C\rightarrow [\bbA_{(3)}^1/\Gm]$ determined by the pair $(L,s)$.

We can then consider the corresponding twisted multiplicative Hitchin fibration
\begin{align*}
\cM  \longrightarrow \cA:=H^0(C,L^3),\ (V,\gamma) \longmapsto t(\gamma).
\end{align*} 
The section $t(\gamma) \in H^0(C,L^3)$ can be explicitly constructed as follows. The trivialization $\nu: \det V \rightarrow \sO_C$ can be regarded as a section $\nu \in H^0(C,\wedge^3 V^*)$ and in turn determines an isomorphism $\tilde{\nu}:V\rightarrow \wedge^2 V^*$. Therefore, $\omega$ defines a section of $V \otimes L$ or, equivalently, a morphism  $v_\omega:L^{-1}\rightarrow V$ . Composing with $g:V\otimes V \rightarrow L$, we obtain a morphism $g(v_\omega,v_\omega):L^{-2}\rightarrow L$, that is, a section $t(\gamma):=g(v_\omega,v_\omega) \in H^0(C,L^3)$.

\subsection{Spectral data of a bilinear bundle}
Let us denote $X=C\times \bbP^1$ with $\pi:X\rightarrow C$ being the natural projection. Given any two sections $s,t \in H^0(C,L^3)$, we can consider the section $\sigma_{s,t} \in H^0(X, \pi^* (L^3))$ defined as
\begin{equation*}
\sigma_{s,t}((p,(y_0,y_1))) = s(p) y_0^2 + t(p) y_1^2.
\end{equation*} 
The zero locus of this section $S_{s,t}=(\sigma_{s,t})_0$ is a one-dimensional subscheme of $X$ that we call the \emph{spectral curve} associated with the pair $(s,t)$. The natural projection $\pi:X\rightarrow C$ induces a finite cover $\pi_{s,t}:S_{s,t}\rightarrow C$.

Let $(V,\gamma=g+\omega)$ be a pair in $\cM=\cM_{(L,s)}$. If $t=t(\gamma)$ is not the zero section, then $\omega$ is generically non-zero and the morphism  $v_\omega:L^{-1}\rightarrow V$ is necessarily injective. We denote by $V_0 \subset V$ the image of this morphism, and by $E=V/V_0$ the quotient bundle. Note that the trivialization $\nu:\det V \rightarrow \sO_C$ induces an isomorphism $\eta:\wedge^2 E \rightarrow V_0^{-1}\cong L$. In turn this determines an isomorphism $\tilde{\eta}:E \rightarrow E^* \otimes L$,  $e\mapsto \eta(e,-)$.

If $p\in C$ is such that $t(p)=g_p(v_{\omega,p},v_{\omega,p})\neq 0$, then ``diagonalizing by congruence"  we can $g$-orthogonally decompose $V_p= V_{0,p} \perp E_p$, and we obtain an induced symmetric bilinear form $h_{p}:S^2 E_p \rightarrow L_p$ such that $g= g|_{V_{0,p}^2} \perp h_p$. Globally, we get an $L$-twisted symmetric bilinear form $h:E_t:=E|_{C\setminus (t)_0} \rightarrow (E^* \otimes L)|_{C\setminus (t)_0}$. Composing with $\tilde{\eta}^{-1}$, we obtain a morphism $$\Phi_t=\tilde{\eta}^{-1}\circ h: E_t\rightarrow E_t.$$ Since $h$ is symmetric, we have  $\tr \Phi_t=0$. Moreover, the orthogonal decomposition of $g$ implies that  $\det \Phi_t = s/(t|_{C\setminus (t)_0}) \in H^0(C \setminus (t)_0,\sO_C)$.  Therefore, the restriction $\Phi_{ts}:=\Phi_t|_{C\setminus ((t)_0 \cup (s)_0)}$ is an automorphism of $E_{st}:=E_t|_{C\setminus ((t)_0 \cup (s)_0)}$. We can then consider its inverse $\Phi_{st}=(\Phi_{ts})^{-1}$. Since $\det \Phi_{st}=t/(s|_{C\setminus (s)_0})$, the automorphism $\Phi_{st}$ extends to a morphism $\Phi_s: E_s:= E|_{C\setminus (s)_0}\rightarrow E_s$.

The data $(E,\Phi_t,\Phi_s)$ determines a graded $k[x_0,x_1]$-module structure on $E$, that is, it allows us to regard  $E$ as a direct image $E=\pi_* U$ of some sheaf $U$ on $C\times \bbP^1$. Since $\Phi$ is traceless, $\det \Phi_t = s/t$ and  $\det \Phi_s = t/s$, we deduce that the sheaf $U$ must in fact be supported in the spectral curve $S_{s,t}\subset C$, so $E=\pi_{s,t,*} U$, for some line bundle $U\rightarrow S_{s,t}$. The isomorphism $\eta:\det E \rightarrow L$ induces an isomorphism
\begin{equation*}
L \otimes ( \det \pi_{s,t,*}\sO_{S_{s,t}})^{-1} \overset{\sim}{\longrightarrow} \Nm_{\pi_{s,t}} U.
\end{equation*} 
Now, $\det \pi_{s,t,*}\sO_{S_{s,t}}\cong L^{-1}$, so in fact we obtain an isomorphism
\begin{equation*}
L^4 \overset{\sim}{\longrightarrow} \Nm_{\pi_{s,t}} U.
\end{equation*} 

\subsection{Reconstruction} Suppose now that we are given the data of a line bundle $U\rightarrow S_{s,t}$, together with an isomorphism
$\eta:L \otimes ( \det \pi_{s,t,*}\sO_{S_{s,t}})^{-1} \rightarrow \Nm_{\pi_{s,t}} U$. Taking direct image we recover the rank $2$ vector bundle $E=\pi_{s,t,*} U$, endowed with the morphisms $\Phi_s:E_s\rightarrow E_s$ and  $\Phi_t:E_t\rightarrow E_t$, with  $\Phi_{st}=(\Phi_{ts})^{-1}$, and with an isomorphism $\tilde{\eta}:E \rightarrow E^* \otimes L$. Away from $(t)_0$ we can recover $V$ as the direct sum $L^{-1}\oplus E$ and reconstruct $g :V\rightarrow V^* \otimes L$ from $t$ and $h:= \tilde{\eta} \circ \Phi_t$. However, near $(t)_0$ generally we only have an extension
\begin{center}
\begin{tikzcd}
	0 \rar & L^{-1} \rar & V \rar & E \rar & 0
\end{tikzcd}
\end{center}
and to recover $g$ we need the symmetric bilinear form $h':=\Phi_s \circ \tilde{\eta}^{-1}$ (which is equivalent to $h$ away from $(t)_0 \cup (s)_0$), but we also need the extra data of a morphism $\psi:L^{-1} \rightarrow E^*\otimes L$ over $(t)_0$, or equivalently, of an element $\psi \in \Hom_{(t)_0}(L^{-2},E^*)$. The choice of $\psi$ is not arbitrary: we need $s=\det g$, so $\psi$ must satisfy the compatibility condition
\begin{equation}\label{eq:compatibility}
s = \psi^* \circ h' \circ \psi:L^{-1}\rightarrow L^2. 
\end{equation}

Dually, we can recover $V^*$ as an extension
\begin{center}
\begin{tikzcd}
	0 \rar & E^* \rar & V^* \rar & L \rar & 0.
\end{tikzcd}
\end{center}
Such an extension is determined by a class $\xi \in \mathrm{Ext}^1(L,E^*)$. Regarding $t \in H^0(C,L^3)$ as a morphism $t:L^{-2}\rightarrow L$, we obtain the following short exact sequence of sheaves
\begin{center}
\begin{tikzcd}
	0 \rar & \underline{\Hom}(L,E^*) \rar{- \circ t} &\underline{\Hom}(L^{-2},E^*) \rar & \underline{\Hom}(L^{-2},E^*)_{(t)_0} \rar & 0,
\end{tikzcd}
\end{center}
which induces a map
\begin{align*}
\delta: \Hom_{(t)_0}(L^{-2},E^*)\longrightarrow  \mathrm{Ext}^1(L,E^*).
\end{align*} 
We claim that $\xi=\delta(\psi)$. Indeed, over $C\setminus ((t)_0 \cup (s)_0)$ we can split $V^*=E^*\oplus L$ and, via the isomorphism $g:V\rightarrow V^* \otimes L$,
we can identify the inclusion  $L^{-1}\hookrightarrow V$, with a map
\begin{align*}
L^{-2} & \longrightarrow  V^* = E^* \oplus L \\
v & \longmapsto (u(v), t(v)),
\end{align*} 
where $u$ is an extension of $\psi$ from  $(t)_0$ to $C\setminus (s)_0$. Therefore a \v{C}ech cocycle for the extension class is defined by
\begin{equation*}
u(v)/t(v) \in H^0(C\setminus (s)_0,\Hom(L^{-2},E^*)).
\end{equation*} 
But this is precisely a representative of $\delta(\psi)$.

Summing up, we found the following.

\begin{thm}
Let $s,t \in H^0(C,L^3)$ be two sections such that the spectral curve $S_{s,t}$ is smooth.
The constructions above give an equivalence of categories between the following:
\begin{itemize}
	\item The category of $L$-twisted bilinear bundles  $(V,\gamma)$ with a trivialization $\nu:\det V \rightarrow \sO_C$ and $s(\gamma)=s$ and $t(\gamma)=t$.
	\item The category of triples $(U,\eta,\psi)$ given by
		\begin{itemize}
			\item a line bundle $U\rightarrow S_{s,t}$,	
			\item an isomorphism $\eta:L^4 \rightarrow \Nm_{\pi_{s,t}} U$,
			\item a vector $\psi_p \in (\pi_{s,t,*}U)_p^* \otimes L_p^2$ for each  $p\in (t)_0$,
		\end{itemize}
		and satisfying the compatibility condition \eqref{eq:compatibility}.
\end{itemize}
\end{thm}

\begin{corol}
The ``Hitchin fibre" over a general point $(s,t)\in H^0(C,L^3)^2$ is isomorphic to the variety $\tilde{P}_{(t)_0}$ constructed as in section \ref{sec:AbelianVariety_Subset} from the Prym variety $P(S_{s,t},C)$ and the set $(t)_0\subset (st)_0 \subset C$. In particular, the fibre is a torsor under the abelian variety $P_{s,t}:=P_{(t)_0}$. Moreover, this abelian variety is dual to the abelian variety $P_{t,s}=P_{(s)_0}$.
\end{corol}

\begin{proof}
The theorem tells us that a pair $(V,\gamma)$ in the fibre corresponds to $(U,\eta,\psi)$ as above. We have the isomorphism $\eta: L^4 \rightarrow \Nm_{\pi_{s,t}}U$, and also note that $L^4 \cong \Nm_\pi(\pi^* L^2)$, so $\eta$ in fact yields an isomorphism  $\Nm_\pi(U\otimes \pi^* L^{-2}) \rightarrow \sO_{X}$. Therefore, the line bundle $U_0:=U\otimes \pi^* L^{-2}$ together with $\eta$ determines a point in $\ker \Nm_\pi$. Moreover, a choice of vector $\psi_p \in (\pi_{s,t,*}U)^*_p \otimes L^2_p$ for each $p\in (t)_0$ determines a section of
\begin{equation*}
\pi^*((\pi_{s,t,*}U)^* \otimes L^2) \cong U^{-1} \otimes \pi^* L^2 \cong U_0^{-1}
\end{equation*} 
over the set $(t)_0$. The compatibility condition \eqref{eq:compatibility} implies that the square of this section coincides with the associated section $\tilde{\eta}|_{(t)_0}\in H^0((t)_0,U_0^{-2})$.
\end{proof}

\appendix
\section{Multiplicative geometric Langlands program} \label{app:MultLanglands}
\subsection{Multiplicative nonabelian Hodge theory}
Let $C$ denote a compact Riemann surface and $K$ a compact Lie group, with Lie algebra $\mathfrak{k}$ and with complexification $K^{\C}=G$. We let $\bM_{\Hit}(C,G)$ denote the moduli space of solutions to the \emph{Hitchin equations}: these are tuples $(E,A,\Phi)$ formed by a $K$-bundle $E\rightarrow C$, a  $K$-connection $A$ on $E$ and a $1$-form $\Phi\in \Omega^1(C,\ad(E))$ such that
\begin{equation*}
	\begin{cases}
F_A - \tfrac{1}{2}[\Phi,\Phi] =0 \\
d_A \Phi = d_A^* \Phi = 0.
	\end{cases}
\end{equation*} 
Nonabelian Hodge theory can be interpreted as the existence of a twistor $\CP^1$-family of Kähler structures on $\bM_{\Hit}(C,G)$. For $\lambda=0 \in \CP^1$, we obtain  $\bM_{\mathrm{Dol}}(C,G)$, the ``Dolbeault" moduli space of stable $G$-Higgs bundles on $C$. For $\lambda \in \C^*$, we obtain $\bM_{\lambda\text{-dR}}(C,G)$, the ``$\lambda$-de Rham" moduli space parametrizing stable $G$-$\lambda$-connections on $C$. 

The Hitchin equations were originally obtained by Hitchin \cite{Hitchin_SelfDuality} when he studied a $2$-dimensional reduction of the self-dual Yang-Mills equations on $\RR^4$. Since the resulting equations are conformally invariant, they can be considered over any Riemann surface. On the other hand, multiplicative nonabelian Hodge theory arises by instead studying a $3$-dimensional reduction of the Yang-Mills equations. The Yang-Mills equations reduced to $\RR^3$ become the \emph{Bogomolny equations}
\begin{equation*}
F_A - \star d_A \Phi = 0,
\end{equation*} 
for a $K$-bundle $E$ on $\RR^3$, a $K$-connection $A$ on $E$, and a section $\Phi \in \Omega^0(\RR^3,\ad(E))$. The solutions $(E,A,\Phi)$ to the Bogomolny equations are called ($K$-)\emph{monopoles}.

We want to consider monopoles that are either \emph{periodic} (\emph{rational} case), \emph{doubly-periodic} (\emph{trigonometric} case) or \emph{triply-periodic} (\emph{elliptic} case). The corresponding moduli spaces of monopoles will be denoted by $\bM^1_{\Mon}(G)$, for periodic monopoles, $\bM^2_{\Mon}(G)$, for doubly-periodic monopoles, and $\bM^3_{\Mon}(G)$, for triply-periodic ones. Mochizuki showed in a series of papers \cite{Mochizuki_Doubly,Mochizuki_Periodic,Mochizuki_Triply} that each of these moduli spaces $\bM^i_{\Mon}(G)$ admits a $1$-parameter family of complex structures, corresponding to \emph{difference modules}.

If $X$ is a complex analytic space and $\tau$ is an automorphism of $X$, a $\tau$-difference module is a module for the sheaf $\mathcal{D}_{\tau}=\sO_X[\Psi,\Psi^{-1}]$ of non-commutative rings generated by $\sO_X$ and an invertible generator  $\Psi$ with the relation  $\Psi f = \tau^* (f) \Psi$. We denote by $\Diff_\tau(X)$ the category of quasi-coherent $\tau$-difference modules on $X$. We can also consider $G$-versions of (locally-free) difference modules. These are \emph{difference connections}. A $\tau$-difference $G$-connection consists of a $G$-bundle $E\rightarrow X$ and an isomorphism of $G$-bundles $E|_{X\setminus D} \rightarrow \tau^* E|_{X\setminus D}$ over the complement of some divisor $D\subset X$.

More precisely, we let $C$ be one of the following Riemann surfaces, endowed with automorphisms $\tau:C\rightarrow C$:
\begin{itemize}
	\item $C=\C$ in the rational case, $\tau_\lambda(z)=z+\lambda$, $\lambda \in \C$,	
	\item $C=\C^*$ in the trigonometric case, $\tau_q(z)=qz$, $q\in \C^*$,
	\item $C=E_\Lambda=\C/\Lambda$, an elliptic curve, with lattice $\Lambda$, in the elliptic case, $\tau_a([z])=[z+a]$, $a\in \C$.
\end{itemize}
A $\tau$-difference module on $\C$, $\C^*$ or $E_\Lambda$, for $\tau$ one of the automorphisms above, is called a rational, trigonometric or elliptic difference module, respectively. Similarly, we define rational, trigonometric or elliptic difference $G$-connections. We denote their corresponding moduli spaces (with the appropriate notion of stability) by $\bM^1_{\Diff,\tau}(G)$, $\bM^2_{\Diff,\tau}(G)$ and $\bM^3_{\Diff,\tau}(G)$, respectively. For each $\tau$, the complex structure over $\bM^i_{\Mon}(G)$ in Mochizuki's family identifies it with $\bM^i_{\Diff,\tau}(G)$. 
Note that for $\tau=\id$, a $\tau$-difference $G$-connection on $C$ is just a multiplicative $G$-Higgs bundle. We denote $\bM_{\mHiggs}(C,G)=\bM^i_{\Diff,\id}(G)$. 

\begin{rmk}
Throughout all the discussion we are dismissing the existence of meromorphic data. These can be easily incorporated in the story. At the level of monopoles we need to consider \emph{Dirac-type singularities}. As explained in \cite{Charbonneau-Hurtubise}, these type of singularities induce on the multiplicative Higgs field precisely the kind of meromorphic singularities that we are interested in. At the level of $q$-difference connections, Dirac-type singularities induce what Mochizuki calls \emph{parabolic structures}. If we incorporate these singularities, the moduli space $\bM_{\mHiggs}(C,G)$ can be interpreted as a moduli space associated with the stacks appearing in the multiplicative Hitchin fibrations of section \ref{sec:UntwistedMultHitchin}.
\end{rmk}

\begin{rmk}
The above discussion also ignores twisting by a diagram automorphism of $G$. This can as well be easily incorporated by considering monopoles which are ``twisted-periodic" in one of the directions, meaning that $(E,A,\Phi)(x,y,t+T)=\theta(E,A,\Phi)(x,y,t)$. Such monopoles can be obtained as fixed points in $\bM^i_{\Mon}(G)$ under a combined action of the cyclic group $\langle \theta \rangle$ on $G$ and on the circle.	
\end{rmk}

\begin{rmk}
In order to relate the discussion of this appendix with the results of the paper we are understanding the Riemann surface $C$ as the analytification of the algebraic curve $C$ considered throughout the rest of the paper. However, we are always assuming that $C$ is projective, but $\C$ and  $\C^*$ are not compact. This issue can be fixed by considering \emph{framed} multiplicative Higgs bundles on $\CP^1$, as explained in \cite{Elliott-Pestun}.
\end{rmk}

\subsection{Derived geometry point of view}
To establish their conjectures, Elliott and Pestun \cite{Elliott-Pestun} use the language of \emph{derived algebraic geometry}; we refer to Remark 1.10 in their paper for a compilation of relevant references.  In this section, we use this language very roughly, and avoiding many technicalities, just with the aim of giving an intuitive explanation of the program.

In particular they consider the ``Betti circle" $S^1_B$, which is a derived stack corresponding to the homotopy type of the circle. If $X$ is a complex analytic space and $\tau$ is an automorphism of $X$, we can consider the corresponding ``mapping torus"   
\begin{equation*}
X_{\tau} = X\times_\tau S^1_B.
\end{equation*} 
This is constructed as a derived fibre product $X_\tau = X\times_{X\times X} X$, where the two maps $X \rightarrow X\times X$ are the diagonal and the  $\tau$-twisted diagonal $x\mapsto (x,\tau(x))$. The derived stack of $\tau$-difference $G$-connections on $G$ can then be recovered as the mapping stack
$$\Conn_G^{\tau}(X)=\Map(X_\tau,BG).$$
In particular, the derived stack of multiplicative $G$-Higgs bundles is the mapping stack
$$\bmHiggs_G(X)=\Map(X\times S^1_B,BG)= \Map(X,G/G).$$
We can also consider the derived category of $\tau$-difference modules, which can be interpreted as
\begin{equation*}
\dDiff_\tau(X)= \QC(X_\tau),
\end{equation*} 
where $\QC$ denotes the derived category of quasi-coherent sheaves.

\begin{rmk}
Again, we are dismissing the presence of singularities. If we do take them into account then $\bmHiggs$ can be thought of as (a derived version of) the stacks $\cM$ that we study in this paper. 	
\end{rmk}

We can also consider twists by outer automorphisms in the derived picture. Let $\theta \in \Aut(G)$ be a diagram automorphism, and consider the natural action of the cyclic group $\langle \theta \rangle\cong \Z/m$ on the circle by rotating with angle $2\pi/m$. The derived stack of $\theta$-twisted multiplicative $G$-Higgs bundles is recovered as the ``fixed points" 
\begin{equation*}
\bmHiggs_{G,\theta}(X) = \bmHiggs_G(X)^{\theta} = \Map(X\times S^1_B,BG)^\theta = \Map(X,G\theta/G).
\end{equation*} 

\subsection{S-duality for 5D theories twisted on a circle}
Elliott and Pestun explain that the stack $\bmHiggs(C,G)$, where $C$ is a Calabi-Yau Riemann surface and $G$ is a complex semisimple simply-laced group, arises from a super Yang-Mills theory in 5 dimensions. More precisely, they construct some physical supersymmetric gauge theory such that its  ``space of solutions" can be interpreted as the shifted cotangent complex $T^*[1]\bmHiggs(C,G)$. The basis for their argument is based on a dimensional reduction from a 6-dimensional theory. 

A little bit more precisely. This  6-dimensional theory is defined on a compact Calabi-Yau $3$-fold $X_6$, and has as space of solutions the shifted cotangent complex of the stack
\begin{equation*}
\mathbf{T}_6(X_6)=\mathbf{Bun}_G(X_6)=\Map(X_6,BG).
\end{equation*} 
Consider a $5$-dimensional manifold $X_5$ with automorphism $\tau$ such that the mapping torus $X_5 \times_{\tau} S_R^1$ is Calabi-Yau, for $S^1_R$ a circle of radius $R$. We obtain a $5$-dimensional theory by evaluating
\begin{equation*}
\mathbf{T}_5^{\tau}(X_5)=\lim_{R\rightarrow 0} \mathbf{T}_6(X_5 \times_{\tau} S_R^1)=\Map(X_5,BG) \times_{\tau^*} S^1_B = \mathrm{Bun}_G(X_5)_{\tau^*}.
\end{equation*} 
Therefore, we can identify
\begin{equation*}
\bmHiggs_G(C)_{\tau^*} = \Map(C \times S^1_B,BG) \times_{\tau^*} S^1_B = \mathbf{T}^{\tau}(C\times S^1 \times \bbR^2).
\end{equation*} 

For the twisted case this is even more interesting. Suppose that the diagram automorphism $\theta \in \Aut(G)$ acts on the $6$-manifold  $C\times_{\tau} S_{R_1}^1 \times S_{R_2}^1 \times \bbR^2$ by rotation on the circle $S^1_{R_2}$. We can then consider two different $2$-dimensional theories
\begin{align*}
\mathbf{T}^\tau_A(C)&=\mathbf{T}_5^\tau(C\times S^1_{R_2} \times \bbR^2) = \lim_{R_1 \rightarrow 0} \mathbf{T}_6(C\times_{\tau} S^1_{R_1} \times S^1_{R_2} \times \bbR^2)^\theta \\ & = \Map(C\times S^1_B,BG)_{\tau^*}^\theta = \bmHiggs_{G,\theta}(C)_{\tau^*}, 
\end{align*} 
and
\begin{align*}
	\mathbf{T}^\tau_B(C)&=\mathbf{T}^{\id}_5(C\times_{\tau} S^1_{R_1} \times \bbR^2) = \lim_{R_2 \rightarrow 0} \mathbf{T}_6(C\times_\tau S^1_{R_1} \times S^1_{R_2} \times \bbR^2)^\theta \\ & = \Map(C\times_\tau S^1_B,B(G^\theta))_{\id} = \Conn^\tau_{G^\theta}(C)_{\id}.
\end{align*} 

The linearizations of these two theories are the \emph{$\mathbf{A}^\tau$-model}
\begin{equation*}
\mathbf{A}^\tau_{G,\theta}(C) = \QC(\bmHiggs_{G,\theta}(C)_{\tau^*}) = \dDiff_{\tau^*}(\bmHiggs_{G,\theta}(C))
\end{equation*} 
and the $\mathbf{B}^{\tau}$-\emph{model}
\begin{equation*}
\mathbf{B}^\tau_{G^\theta}(C) = \QC(\Conn^\tau_{G^\theta}(C)_{\id}) = \dDiff_{\id}(\Conn^\tau_{G^\theta}(C)_{\id}).
\end{equation*} 

Elliott and Pestun conjecture that $\mathbf{S}$-duality for these two physical theories should exchange $G^\theta$ with its Langlands dual, and the parameter  $\tau$ with its inverse $\tau^{-1}$. This is the \emph{multiplicative geometric Langlands correspondence}.

\begin{conj}
Let $G$ be a semisimple simply-laced complex group, $\theta$ a diagram automorphism of $G$, and $C$ a Calabi-Yau Riemann surface. For any $\tau \in \Aut(C)$ there is a derived equivalence
\begin{equation*}
\mathbf{mGLC}:\dDiff_{\tau^{-1}}(\bmHiggs_{G,\theta}(C)) \longrightarrow \dDiff_{\id}(\Conn^{\tau}_{(G^\theta)^\vee}(C)).
\end{equation*} 
\end{conj}

Furthermore, we can now consider a ``classical limit" of this correspondence, by taking $\tau=\id$:
\begin{equation*}
 \mathbf{mGLC}_{\mathrm{cl}}:\dDiff_{\id}(\bmHiggs_{G,\theta}(C)) \longrightarrow \dDiff_{\id}(\bmHiggs_{(G^\theta)^\vee}(C)).
\end{equation*} 
In particular, this implies the existence of an equivalence
\begin{equation*}
\QC(\bmHiggs_{G,\theta}(C)) \longrightarrow \QC(\bmHiggs_{(G^\theta)^\vee}(C)).
\end{equation*} 
Now, this last formula is precisely the form of the main statements of this paper.

\section{The proof of Theorem \ref{thm:duality}} \label{app:Proof}
\subsection{Matching finite groups.} The first step is showing that $\mathrm{S}_0$ induces an isomorphism at the level of the inertia and the component groups of $\bD(\cP^1)$ and $\check{\cP^0}$. The inertia group of $\cP^1$ is
 \begin{equation*}
\Aut_0(\cP^1) = H^0(C,J^1) = T^W,
\end{equation*} 
so the component group of $\bD(\cP^1)$ is its Cartier dual 
\begin{equation*}
\pi_0(\bD(\cP^1))=(T^W)^* = X^*(T)_W.
\end{equation*} 
That is, the component group is the coinvariant group of the character lattice $X^*(T)$. The map $\bD(\iota)$ induces the natural projection $X^*(T)\rightarrow X^*(T)_W$ at the level of component groups. On the other hand, as explained in a paper of Ngô \cite[Section 6]{Ngo_Hitchin}, applying a lemma of Kottwitz one gets an isomorphism $\pi_0(\check{\cP}^0)\cong X^*(T)_W$. Since the norm map on $X^*(T)$ factors through the natural projection to the coinvariant group, we conclude that $\mathrm{S}_0$ induces an isomorphism from $\pi_0(\bD(\cP^1))$ to $\pi_0(\check{\cP}^0)$.

Consider now the exact sequence
\begin{center}
\begin{tikzcd}
	H^0(C,J^1/J^0)=\pi_0(H^0(C,J^1/J^0)) \rar & \pi_0(\cP^0)\cong X_*(T)_W \rar & \pi_0(\cP^1) \rar & 0.	
\end{tikzcd}
\end{center}
Its Cartier dual is the exact sequence
\begin{center}
\begin{tikzcd}
	0 \rar & \pi_0(\cP^1)^*=\Aut_0(\bD(\cP^1)) \rar & (T^\vee)^W \rar & H^0(C,J^1/J^0).
\end{tikzcd}
\end{center}
Compare with the exact sequence
\begin{center}
\begin{tikzcd}
	0 \rar & \Aut_0(\check{\cP}^0) \rar & \Aut_0(\check{\cP}^1)=(T^\vee)^W \rar & H^0(C,\check{J}^1/\check{J}^0).
\end{tikzcd}
\end{center}
Recall from \ref{sec:cameral_simple_Galois} that the Killing form matches the component group $J^1/J^0$ with $\check{J}^1/\check{J}^0$. This induces an isomorphism $H^0(C,J^1/J^0)\rightarrow H^0(C,\check{J}^1/\check{J}^0)$ and, in turn, an isomorphism $\Aut_0(\bD(\cP^1)) \rightarrow \Aut_0(\check{\cP}^0)$.

\subsection{\texorpdfstring{$\cP^1$}{P1} and equivariant bundles} Given any $T$-bundle $E\rightarrow \tilde{C}$, and an element $w\in W$, we can define another $T$-bundle by putting $w\cdot E= w_{\tilde{C}}^* w^{-1}_T(E)$, where $w_{\tilde{C}}$ is the corresponding monodromy action on $\tilde{C}$ and $w_T$ the automorphism of $T$ induced by the $W$-action. A \emph{strong $W$-equivariant structure} on $E$ is a set $\left\{\eta_w: w\in W\right\}$ of isomorphisms $\eta_w:w\cdot E \rightarrow E$ such that
\begin{equation*}
\eta_{ww'}=\eta_w \circ w\cdot \eta_{w'}.
\end{equation*} 
A \emph{strongly $W$-equivariant $T$-bundle} on $\tilde{C}$ is a pair $(E,\eta)$ formed by a $T$-bundle $E$ and a strong  $W$-equivariant structure on it. We denote by $\Bun_{T/\tilde{C}}^W$ the stack of strongly $W$-equivariant bundles on $\tilde{C}$. 

Every element of $\cP^1$ can be lifted to a strongly $W$-equivariant $T$-bundle on $\tilde{C}$ and a typical descent argument (see, for example, Lemma 3.1.2 in the paper of Chen--Zhu \cite{Chen-Zhu}) in fact provides an exact sequence 
\begin{center}
\begin{tikzcd}
	0 \rar & \cP^1 \rar & \Bun_{T/\tilde{C}}^W \rar & \bigoplus_{\alpha \in \Phi_+} B_{\alpha} \times T_{s_{\alpha}}.
\end{tikzcd}
\end{center}
The strong $W$-equivariant structures on a trivial $T$-bundle are classified by the group cohomology group $H^1(W,T)$. Therefore, the kernel of the map $\iota:\cP^1 \rightarrow \Bun_{T/\tilde{C}}$ coincides with the kernel of the induced map $H^1(W,T)\rightarrow \bigoplus_{i=1}^b T_{w_i}$. However, a lemma of Chen and Zhu \cite[Lemma 3.6.2]{Chen-Zhu} implies that this map is in fact injective. As a consequence, we can identify the neutral connected component $P^1=|\cP^1|^0$ as the abelian variety 
\begin{equation*}
P^1 = (\Jac(\tilde{C})\otimes X_*(T))^{W,0}.
\end{equation*} 
The dual of this abelian variety is the coinvariant abelian variety
\begin{equation*}
	(P^1)^\wedge = (\Jac(\tilde{C})\otimes X^*(T))_W.
\end{equation*} 
Therefore, we have reduced the proof of the theorem to the following lemma.

\begin{lemma} \label{lemma:coinvariant}
The duality map $\mathrm{S}_0$ induces an isomorphism
 \begin{equation*}
\mathrm{S}_0: (\Jac(\tilde{C})\otimes X^*(T))_W \longrightarrow \check{P}^0:= |\check{\cP}^0|^0.
\end{equation*} 
\end{lemma}

\subsection{Homology of \texorpdfstring{$\check{P}^0$}{P0}} 
\begin{rmk}
	So far, our arguments could be easily generalized to positive characteristic. The main difference between the case of characteristic $0$ and positive characteristic is in the proof of the previous lemma.
Donagi and Pantev \cite{Donagi-Pantev} proved it over the complex numbers and Chen and Zhu \cite{Chen-Zhu} over positive characteristic. The proof given in \cite{Donagi-Pantev} uses a Hodge theoretic argument and singular homology, and thus only works over the complex numbers. The proof in positive characteristic describes the Tate module of $\check{P}^0$ and requires techniques from $\ell$-adic cohomology that we do not wish to deploy here.
\end{rmk}

To show that the map 
$\mathrm{S}_0: (\Jac(\tilde{C})\otimes X^*(T))_W \rightarrow \check{P}^0$ is an isomorphism of abelian varieties, we just need to show that it defines an isomorphism on the singular homology groups. For the first abelian variety, we have
\begin{equation*}
H_1 (\Jac(\tilde{C})\otimes X^*(T))_W = H^1(\tilde{C},X^*(T))_{W,\mathrm{tf}}.
\end{equation*} 
where the subscript tf indicates taking the torsion-free part.

Consider now the following local system $\fL=(\pi^0_* X^*(T))^W$ over $U=C\setminus B$, where $\pi^0$ is the restriction of the ramified cover  $\pi:\tilde{C}\rightarrow C$ to the étale locus $\pi^{-1}(U)$. We consider as well the inclusion $j:U\hookrightarrow C$ and the direct image  $j_*\fL$. The sheaf $\check{\mathscr{J}}^0$ of sections of $\check{J}^0$ is naturally isomorphic to the sheaf $j_*\fL \otimes \sO_C^*$. Taking the exponential exact sequence
\begin{center}
\begin{tikzcd}
	0 \rar & \Z \rar & \sO_C \rar & \sO_C^* \rar & 1,
\end{tikzcd}
\end{center}
and tensoring with $j_* \fL$ we obtain a short exact sequence
\begin{center}
\begin{tikzcd}
	0 \rar & j_*\fL \rar & j_* \fL \otimes \sO_C \rar & \check{\mathscr{J}}^0 \rar & 1,
\end{tikzcd}
\end{center}
since the sheaf $\underline{\mathrm{Tor}}_1(j_*\fL,\sO_C^*)$ is supported on the finite set $B$, while $j_*\fL$ has no compactly supported sections. The long exact sequence in cohomology provides an isomorphism
 \begin{equation*}
H^1(C,j_* \fL \otimes \sO_C)/H^1(C,j_*\fL)_{\mathrm{tf}} \overset{\sim}{\longrightarrow} \check{P}^0.
\end{equation*} 
Now, from Hodge theory we know that $H^1(C,j_* \fL \otimes \sO_C) \cong H^1(C,j_* \fL)_{\mathrm{tf}} \otimes \mathbb{R}$, so $$\check{P}^0=H^1(C,j_* \fL)_{\mathrm{tf}} \otimes S^1$$ and thus
\begin{equation*}
H_1(\check{P}^0) \cong H^1(C,j_* \fL)_{\mathrm{tf}}.
\end{equation*} 
Consider now the cohomology group $H^1(U,\fL)$, and the compactly supported cohomology group $H^1_c(U,\fL)$. There is a natural map $H_c^1(U,\fL)\rightarrow H^1(U,\fL)$, which induces a map on the torsion-free parts.

The following easy lemma can be found in Donagi--Pantev \cite[Lemma 6.3]{Donagi-Pantev}.

\begin{lemma}[Donagi--Pantev]
	The cohomology of $j_*\fL$ satisfies
\begin{equation*}
H^1(C,j_*\fL) = \mathrm{im}[H^1_c(U,\fL) \rightarrow H^1(U,\fL)].
\end{equation*} 	
\end{lemma}

Finally, the last step for proving Lemma \ref{lemma:coinvariant} is the following.

\begin{lemma}
There is a natural isomorphism
\begin{equation*}
H^1(\tilde{C},X^*(T))_{W,\mathrm{tf}} \cong \mathrm{im}[H_c^1(U,\fL)_{\mathrm{tf}}\rightarrow H^1(U,\fL)].
\end{equation*} 
\end{lemma}
\begin{proof}
There is a natural injective map $H^1(\tilde{C},X_*(T))\hookrightarrow H^1(\pi^{-1}(U),X_*(T))$	given by restriction. Taking invariants, we obtain an injective map $$H^1(\tilde{C},X_*(T))^W \hookrightarrow H^1(\pi^{-1}(U),X_*(T))^W.$$ Now, since $\pi$ is étale over $U$, the second of these groups is in fact isomorphic to  $H^1(U,\fL^\vee)$, for the local system $\fL^\vee=(\pi_*^0 X_*(T))^W$. We obtain an injective map $H^1(\tilde{C},X_*(T))^W \hookrightarrow H^1(U,\fL^\vee)$ and, dualizing lattices, we get a surjection
\begin{equation*}
	H^1(U,\fL^\vee)^\vee_{\mathrm{tf}} \longrightarrow H^1(\tilde{C},X_*(T))_{W,\mathrm{tf}}.
\end{equation*} 
Now, by a version of Poincaré-Verdier duality \cite[Lemma 6.1]{Donagi-Pantev}, we have $H^1(U,\fL^\vee)^\vee_\mathrm{tf}=H^1_c(U,\fL)_{\mathrm{tf}}$. By the same argument, we have an injective map $H^1(\tilde{C},X^*(T))^W \hookrightarrow H^1(U,\fL)$. The norm map
\begin{align*}
\Nm: H^1(\tilde{C},X^*(T)) & \longrightarrow H^1(\tilde{C},X^*(T)) \\
\gamma & \longmapsto \sum_{w\in W} w\cdot \gamma
\end{align*} 
induces a map $\Nm:H^1(\tilde{C},X^*(T))_{W,\mathrm{tf}}\rightarrow H^1(\tilde{C},X^*(T))^W$. The lemma follows from the fact that the square
\begin{center}
\begin{tikzcd}
H^1_c(U,\fL)_{\mathrm{tf}} \ar{r} \ar[twoheadrightarrow]{d} & H^1(U,\fL)  \\
H^1(\tilde{C},X^*(T))_{W,\mathrm{tf}} \ar{r}{\Nm} &H^1(\tilde{C},X^*(T))^W. \ar[hook]{u}
\end{tikzcd}
\end{center}
is commutative.
\end{proof}

\section{Review of Prym varieties} \label{app:Prym}
\subsection{The Prym variety}
Let $S$ and $C$ be smooth projective curves over $k$ and let $\pi:S\rightarrow C$ be a degree $2$ finite map. We denote by $B\subset C$ the branch divisor of $\pi$. It is well known that $B$ consists of an even number of single points, $|B|=2k$. The Galois group of the cover $\pi$ is generated by an involution $\sigma: S\rightarrow S$. We can identify $B$ with its preimage $\pi^{-1}(B)$, which is the set of fixed points $\pi^{-1}(B)=S^\sigma$.

Consider the norm map
\begin{align*}
\Nm_\pi: \Pic S & \longrightarrow \Pic C \\
\sO_S(D) & \longmapsto \sO_C(\pi(D)).
\end{align*} 
The \emph{Prym variety} $P^1=P(S,C)$ is defined as the neutral connected component of the kernel of the norm map, $P^1:=(\ker \Nm_\pi)^0$. Note that we can write
\begin{equation*}
\ker \Nm_\pi = \left\{(U,\eta): U \in \Pic S, \eta: U^{-1} \overset{\sim}{\rightarrow} \sigma^* U \right\}/\mathrm{iso}.
\end{equation*} 
For such a point $(U,\eta) \in \ker \Nm_\pi$ we obtain an isomorphism $\eta|_B : U^{-1}|_B \rightarrow U|_B$ or, equivalently, a section $\tilde{\eta} \in H^0(B,U^2)$. We can then identify the Prym variety as
\begin{equation*}
P^1 = \left\{(U,\eta) \in \ker \Nm_\pi \text{ such that there exists } \xi \in H^0(B,U) \text{ with } \tilde{\eta}=\xi^2\right\}/\text{iso}.
\end{equation*} 
We denote by $P^0=\bD(P(S,C))$ its dual abelian variety.

\subsection{Polarizations}
Recall that a polarization of an abelian variety $P$ is an isogeny $\rho:P\rightarrow \hat{P}$ between $P$ and its dual abelian variety $\hat{P}$. If it is an isomorphism, then it is called a principal polarization.
The Jacobian of a smooth projective curve is principally polarized. That is, we have isomorphisms $\Jac(C)\cong \bD(\Jac(C))$ and $\Jac(S)\cong \bD(\Jac(S))$ between the Jacobians of $S$ and $C$ and their duals. Under these isomorphisms, the dual of the norm map $\Nm_\pi$ is the pullback  $\pi^*: \Jac(C)\rightarrow \Jac(S)$. 
Therefore, we obtain an isomorphism $P^0 \cong \Jac(S)/\pi^* \Jac(C)$ and we can construct a polarization $\rho:P^1\rightarrow P^0$ through the diagram
 \begin{center}
\begin{tikzcd}
	P^1 \rar{\rho} \dar & P^0 \cong \Jac(S)/\pi^* \Jac(C)  \\
	\Jac(S) \rar & \Jac(S). \uar
\end{tikzcd}
\end{center}
But note now that $\Nm_\pi \pi^* U = U^2$, so in fact $\rho(P^1)\subset P^1$ and we can write
 \begin{equation*}
P^0 \cong P^1 / P^1 \cap \pi^* \Jac(C). 
\end{equation*} 
Moreover, $\pi^* U\in \pi^* \Jac(C) \cap P^1$ means that $U^2=\Nm_\pi \pi^* U = \sO_C$, so in fact we get 
\begin{equation*}
P^0 = P^1 / \pi^* \Jac(C)[2],
\end{equation*} 
where $\Jac(C)[2]$ is the finite subgroup of order $2$ elements of the Jacobian of $C$.

\subsection{The squaring map} Consider the squaring map 
\begin{align*}
2: P^1 & \longrightarrow P^1 \\
(U,\eta) & \longmapsto (U^2, \eta^2).
\end{align*} 
This map is surjective, since $P^1$ is connected, and has a finite kernel $P^1[2]$, the set of order two elements of $P^1$. The inclusion $\pi^* \Jac(C) \hookrightarrow P^1$ determines an inclusion $\iota:\pi^* \Jac(C)[2] \hookrightarrow P^1[2]$. We can determine now the cokernel of this inclusion. 

First of all, recall that for any pair $(U,\eta) \in \ker \Nm_\pi$ we have the section $\tilde{\eta}\in H^0(B,U^2)$. Now, if $(U,\eta)$ is $2$-torsion element, then by definition $U^2 \cong \sO_S$ and $\tilde{\eta} \in H^0(B,\mu_2)$. Since $(U,\eta)$ is defined up to isomorphism, the element $\tilde{\eta}$ is well-defined up to sign. Therefore we obtain a diagram
\begin{center}
\begin{tikzcd}
& 0 \dar & 0 \dar & 0 \dar & \\
	0 \rar & \pi^* \Jac(C)[2] \rar{\iota} \dar{\id} & P^1[2] \rar \dar & \coker \iota \dar \rar & 0 \\
	0 \rar & \pi^* \Jac(C)[2] \rar \dar & \ker \Nm_\pi [2] \rar \dar & H^0(B,\mu_2)/\mu_2 \dar \rar & 0 \\
	& 0 \rar & \mu_2 \rar & \mu_2 \rar & 0 ,
\end{tikzcd}
\end{center}
so we can identify
\begin{equation*}
\coker \iota = \left\{m=(m_b)_{b \in B}: m_b \in \mu_2, \prod_{b\in B} m_b = 1\right\} / (m \sim -m) \cong (\mu_2)^{2k-1}/\mu_2 = (\mu_2)^{2k-2}.
\end{equation*} 

We obtain the following.
\begin{corol}
The Prym variety $P^1$ is principally polarized if and only if  $k=0$ or  $k=1$.	
\end{corol}

\subsection{Abelian variety associated with a subset of the ramification divisor} \label{sec:AbelianVariety_Subset}
Consider now a proper non-empty subset $A\subset B$. We have a natural short exact sequence 
\begin{center}
\begin{tikzcd}
	0 \rar & H^0(B\setminus A, \mu_2) \rar & H^0(B,\mu_2) \rar & H^0(A,\mu_2) \rar & 0,
\end{tikzcd}
\end{center}
which induces a short exact sequence
	\begin{equation*}
	0 \longrightarrow \ker(H^0(B\setminus A,\mu_2)\rightarrow \mu_2) \longrightarrow \ker(H^0(B,\mu_2)/\mu_2 \rightarrow \mu_2) \longrightarrow H^0(A,\mu_2)/\mu_2 \longrightarrow 0.
	\end{equation*} 
Associated with this subset $A$ we can define the variety
 \begin{equation*}
\tilde{P}_A = \left\{(U,\eta,\xi): (U,\eta) \in \ker \Nm_\pi, \xi \in H^0(A,U), \text{ such that } \tilde{\eta}|_{A}=\xi^2\right\}/ \text{iso.}
\end{equation*} 
This variety is a finite cover of $P^1$, and we also have a natural map
\begin{align*}
u_A:P^1 & \longrightarrow \tilde{P}_A \\
(U,\eta) & \longmapsto (U^2,\eta^2,\tilde{\eta}|_A).
\end{align*} 
The kernel of this map is 
\begin{align*}
	\ker u_A &= \left\{(U,\eta)\in P^1: U^2\cong \sO_S, \eta^2=1, \tilde{\eta}|_A = 1 \right\}/ \text{iso.} 
\end{align*} 
Note that there is an inclusion $\iota_A:\ker u_A\rightarrow P^1[2]$, and a short exact sequence
\begin{center}
\begin{tikzcd}
	0 \rar & \ker u_A \rar & P^1[2] \rar & H^0(A,\mu_2)/\mu_2 \rar & 0.
\end{tikzcd}
\end{center}
So we also have a short exact sequence
\begin{center}
\begin{tikzcd}
	0 \rar & \pi^* \Jac(C)[2] \rar & \ker u_A \rar & \ker(H^0(B\setminus A,\mu_2)\rightarrow \mu_2) \rar & 0.
\end{tikzcd}
\end{center}
We define
\begin{equation*}
P_A = P^1/\ker u_A.
\end{equation*} 
This is a new abelian variety, which acts freely and transitively on $\tilde{P}_A$.

Note that, since $P^0=P^1/\pi^* \Jac(C)[2]$, we have an isomorphism $$P_A \cong P^0/\ker(H^0(B\setminus A,\mu_2)\rightarrow \mu_2).$$ Thus, by dualizing on the exact sequence defining this quotient, we obtain a short exact sequence
\begin{center}
\begin{tikzcd}
	0 \rar & H^0(B\setminus A, \mu_2)/\mu_2 \rar & \hat{P}_A \rar & P^1 \rar & 0.
\end{tikzcd}
\end{center}
But note that we have the diagram
\begin{center}
\begin{tikzcd}
	 \ker u_{B\setminus A} \rar \dar & \ker u_{B\setminus A} \rar \dar & 0 \dar  \\
	 P^1[2] \rar \dar & P^1 \rar{2} \dar & P^1 \dar  \\
	 H^0(B\setminus A,\mu_2)/\mu_2 \rar & P_{B\setminus A} \rar & P^1 ,
\end{tikzcd}
\end{center}
with short exact rows and columns. This yields an isomorphism 
\begin{equation*}
\hat{P}_A \cong P_{B\setminus A}.
\end{equation*} 
In particular, if $A$ has exactly  $k$ elements, then we can use a bijection between $A$ and $B\setminus A$ to determine an isomorphism $\hat{P}_A \cong P_A$, and thus a principal polarization on $P_A$.

\subsection{Cameral interpretation}
The finite map $\pi:S\rightarrow C$ is a cameral cover associated with the root system $\sfA_1$,  $\Phi=\left\{\alpha,-\alpha\right\}$. Clearly, it has simple Galois ramification, and we are assuming that $S$ is smooth. 

The kernel of the norm map $\Nm_\pi:\Pic S \rightarrow \Pic C$ can be understood as the coarse moduli space of strongly $\Z/2$-equivariant $\Gm$-torsors on $S$.
The Prym variety $P^1$ is the coarse moduli space of the Picard stack of  $J^1$, torsors, for 
\begin{equation*}
J^1 = \pi_*(S\times \Gm)^{\Z/2},
\end{equation*} 
and the inclusion $P^1\rightarrow \ker \Nm_\pi$ is induced by the inclusion of $\cP^1=\Bun_{J^1}$ in $\Bun_{\Gm/S}^{\Z/2}$. Indeed, the extra condition of the section $\tilde{\eta}\in H^0(B,U^2)$ associated with an isomorphism $\eta:U^{-1}\rightarrow \sigma^* U$ being a square, corresponds to $(U,\eta)$ being in the kernel of the map $$\ker \Nm_\pi \rightarrow B \times H^1(\Z/2,\Gm)\cong H^0(B,\mu_2).$$
The Picard stack $\cP^1$ is connected, but it has a nontrivial inertia group, coming from the fact that $\Gm^{\Z/2}=\mu_2$.

Over each ramification point, the group $J^1$ has two connected components, and we let  $J^0\subset J^1$ be the subgroup of fiberwise neutral connected components. This inclusion induces a long exact sequence in cohomology
\begin{center}
\begin{tikzcd}
	0 \rightarrow H^0(C,J^0)  \rar & H^0(C,J^1) \rar & H^0(B,\mu_2) \rar & H^1(C,J^0) \rar & H^1(C,J^1)\rightarrow 0.
\end{tikzcd}
\end{center}
The Picard stack $\cP^0=\Bun_{J^0}$ is actually a scheme: it has trivial inertia group, but it has two connected components $\pi_0(\cP^0)=\Z/2$. Its neutral connected component is the variety $P^0$, the dual of the Prym variety. Indeed, the above map induces an isogeny $P^0\rightarrow P^1$ with kernel $(\mu_2)^{2k-2}$, and we saw already that this kernel is precisely the cokernel of the inclusion $\pi^* \Jac(C)[2]\hookrightarrow P^1[2]$.

More generally, for each subset $A$ of the branch locus $B$, we can consider the subgroup $J_A$ of  $J^1$ containing $J^0$ and with one connected component over the points of $A$ and two connected components over the points of the complement $B\setminus A$. The inclusion $J^0\subset J_A$ induces an exact sequence
\begin{center}
\begin{tikzcd}
	0 \rightarrow H^0(C,J^0)  \rightarrow H^0(C,J_A) \rightarrow H^0(B\setminus A,\mu_2) \rightarrow H^1(C,J^0) \rightarrow H^1(C,J_A)\rightarrow 0.
\end{tikzcd}
\end{center}
This sequence induces an isogeny $P^0\rightarrow |\Bun_{J_A}|^0$ with kernel equal to $\ker(H^0(B\setminus A,\mu_2)\rightarrow \mu_2)$. Therefore, the abelian variety $|\Bun_{J_A}|^0$ is precisely the abelian variety $P_A$. Its dual  $\hat{P}_A$ is exactly the abelian variety  $P_{B\setminus A} = |\Bun_{J_{B\setminus A}}|^0$. Indeed, $J^1/J_A= A \times \mu_2 = J_{B\setminus A}/J^0$.

\printbibliography
\vfill
\hrule
\end{document}